\documentclass[12pt]{amsart}
\usepackage{a4wide}
\usepackage{amsthm,amsfonts,amsmath,amssymb}
\usepackage{mathtools}
\usepackage{bbm}
\usepackage{hyperref}
\usepackage{booktabs}
\usepackage{longtable}
\usepackage{xcolor}
\usepackage{dsfont}
\usepackage{tikz,tikz-cd}

\allowdisplaybreaks[1]
\newcommand{\eps}{\varepsilon}
\newcommand{\CC}{\mathbb{C}}
\newcommand{\HH}{\mathbb{H}}
\newcommand{\RR}{\mathbb{R}}

\newcommand{\ZZ}{\mathbb{Z}}
\newcommand{\NN}{\mathbb{N}}
\newcommand{\Aa}{\mathcal{A}}
\newcommand{\bN}{\mathbf{N}}
\newcommand{\bH}{\mathbf{I}}
\newcommand{\bF}{\mathcal{F}}

\newcommand{\SL}{{\rm SL}}
\newcommand{\PSL}{{\rm PSL}}

\newcommand{\im}{\operatorname{Im}}
\newcommand{\re}{\operatorname{Re}}
\newcommand{\sgn}{\mathrm{sgn}}

\newcommand*\smat[4]{\begin{smallmatrix}#1&#2\\#3&#4\end{smallmatrix}}
\newcommand*\pmat[4]{\begin{pmatrix}#1&#2\\#3&#4\end{pmatrix}}

\newcommand{\ol}[1]{\overline{#1}}

\newcommand{\Mod}[1]{\ (\text{mod}\ #1)}
\newcommand{\change}[1]{{#1}}
\newcommand{\sm}{\smallsetminus}
\renewcommand{\phi}{\varphi}

\newtheorem*{theoremA}{Theorem A}

\newtheorem{theorem}{Theorem}[section]
\newtheorem{proposition}{Proposition}[section]
\newtheorem{lemma}{Lemma}[section]
\newtheorem{corollary}{Corollary}[section]
\theoremstyle{definition}
\newtheorem*{remark}{Remark}

\numberwithin{equation}{section}
\numberwithin{figure}{section}

\theoremstyle{plain}

\title[Fourier interpolation with zeros of zeta and $L$-functions]{Fourier interpolation with \\ zeros of zeta and $L$-functions}

\author[Andriy Bondarenko]{Andriy Bondarenko}
\address{Department of Mathematical Sciences \\ Norwegian University of Science and Technology \\ NO-7491 Trondheim \\ Norway}
\email{andriybond@gmail.com}

\author[Danylo Radchenko]{Danylo Radchenko}
\address{ETH Zurich, Mathematics Department, Z\"urich 8092, Switzerland}
\email{danradchenko@gmail.com}

\author[Kristian Seip]{Kristian Seip}
\address{Department of Mathematical Sciences \\ Norwegian University of Science and Technology \\ NO-7491 Trondheim \\ Norway}
\email{kristian.seip@ntnu.no}

\thanks{Bondarenko and Seip were supported in part by Grant 275113 of the Research Council of Norway.}
\date{}
\begin{document}
\maketitle
\begin{abstract}
We construct a large family of Fourier interpolation bases for functions analytic in a strip symmetric about the real line. Interesting examples involve the nontrivial zeros of the Riemann zeta function and other $L$-functions. We establish a duality principle for Fourier interpolation bases in terms of certain kernels of general Dirichlet series with variable coefficients. Such kernels admit meromorphic continuation, with poles at a sequence dual to the sequence of frequencies of the Dirichlet series, and they satisfy a functional equation. Our construction of concrete bases relies on a strengthening of Knopp's abundance principle for Dirichlet series with functional equations and a careful analysis of the associated Dirichlet series kernel, with coefficients arising from certain modular integrals for the theta group.
\end{abstract}

\tableofcontents

\section{Introduction}
The Riemann--Weil explicit formula (sometimes also called the Guinand--Weil
explicit formula) expresses the familiar duality between the prime numbers and
the nontrivial zeros of the Riemann zeta function $\zeta(s)$ in the following
compelling way:
    \begin{align} \frac{1}{2\pi} \int_{-\infty}^{\infty} f(t) & \left(\frac{\Gamma'(1/4+it/2)}{\Gamma(1/4+it/2)} - \log \pi \right)dt+f\Big(\frac{i}{2}\Big)+f\Big(\frac{-i}{2}\Big) \nonumber \\ \label{eq:RW}
     & = \frac{1}{2\pi} \sum_{n=1}^{\infty} \frac{\Lambda(n)}{\sqrt{n}} \left(\widehat{f}\Big(\frac{\log n}{2\pi}\Big)+\widehat{f}\Big(-\frac{\log n}{2\pi}\Big) \right)+  \sum_{\rho} f\left(\frac{\rho-1/2}{i}\right).
     \end{align}
Here $f(z)$ is a function analytic in the strip $|\im\, z|<1/2+\eps$ for
some $\eps>0$, $|f(z)|\ll (1+|z|)^{-1-\delta} $ for some $\delta>0$ when
$|\re\, z| \to \infty$, and
    \[ \widehat{f}(\xi)\coloneqq
    \int_{-\infty}^{\infty} f(x) e^{-2\pi i x \xi} dx ;\]
$\Lambda(n)$ is the von Mangoldt function defined to be $\log p$ if $n=p^k$, $p$ a prime and $k\ge 1$, and zero otherwise, while the second sum in \eqref{eq:RW} runs over the nontrivial zeros $\rho$ of $\zeta(s)$ (counting multiplicities in the usual way). The Riemann--Weil formula generalizes the classical Riemann--von Mangoldt explicit formula \cite[Ch. 17]{D} and arose to prominence from Weil's work \cite{W}, in which it appeared in a considerably more general form.

Our nontraditional emplacement of the two series in \eqref{eq:RW} on one side of the equation is made to connect the Riemann--Weil formula to the object of study of this paper, the prototype of which is another Fourier duality relation involving the nontrivial zeros of~$\zeta(s)$. We follow the convention of denoting these zeros  by $\rho=\beta+i\gamma$, but instead of accounting for multiple zeros (if any) in the usual way, we associate with each $\rho$ the multiplicity $m(\rho)$ of the zero of $\zeta(s)$ at $\rho$. We let $\mathcal{H}_1$ denote the space of functions $f(z)$ that are analytic in the  strip $|\im\, z|<1/2+\eps$ and satisfy the integrability condition
	\[ \sup_{|y|<1/2+\eps}
	   \int_{-\infty}^{\infty} |f(x+iy)| (1+|x|) dx < \infty \]
for some $\eps>0$. Functions $h(z)$ on $\mathbb{C}$ with the property that
	\[ h(x+iy) \ll_{y,l} (1+|x|)^{-l} \]
for every real $y$ and positive $l$ will be said to be rapidly decaying.
To simplify matters, we state our result only for even functions.

\begin{theorem}\label{thm:zeta}
There exist two sequences of rapidly decaying and even entire functions $U_n(z)$, $n=1,2,...$, and $V_{\rho,j}(z)$, $0\le j < m(\rho)$, with $\rho$ ranging over the nontrivial zeros of $\zeta(s)$ with positive
imaginary part, such that for every even function $f$ in $\mathcal{H}_1$
and every $z=x+iy$ in the strip $|y|<1/2$ we have
    \begin{equation} \label{eq:add} f(z)=\sum_{n=1}^{\infty} \widehat{f}\left(\frac{\log n}{4\pi}\right) U_n(z)+\lim_{k\to \infty} \sum_{0<\gamma\le T_k}
    \sum_{j=0}^{m(\rho)-1} f^{(j)}\left(\frac{\rho-1/2}{i}\right) V_{\rho,j}(z)   \end{equation}
for some increasing sequence of positive numbers $T_1, T_2, \dots$ tending to $\infty$ that does not depend on neither $f$ nor on $z$. Moreover, the functions $U_n(z)$ and $V_{\rho,j}(z)$ enjoy the following interpolatory properties:
    \begin{equation} \label{eq:interp} \begin{array}{rclrcl} U_n^{(j)}\left(\frac{\rho-1/2}{i}\right) & = & 0,   & \widehat{U}_n\left(\frac{\log n'}{4\pi}\right)& =& \delta_{n,n'} ,\\
    V^{(j')}_{\rho,j}\left(\frac{\rho'-1/2}{i}\right)& = & \delta_{(\rho,j),(\rho',j')},  & \widehat{V}_{\rho,j}\left(\frac{\log n}{4\pi}\right)& = &0, \end{array}  \end{equation}
with $\rho, \rho'$ ranging over the nontrivial zeros of $\zeta(s)$
with positive imaginary part, $j,j'$ over all nonnegative integers less than
or equal to respectively $m(\rho)-1, m(\rho')-1$, and $n, n'$ over all
positive integers.
\end{theorem}
As an immediate corollary we get the following result that appears to be difficult to
obtain without relying on the interpolation formula from Theorem~\ref{thm:zeta}.
\begin{corollary}\label{cor:zeta}
(i) If an even function $f$ in $\mathcal{H}_1$ satisfies
$\widehat{f}(\frac{\log n}{4\pi})=f^{(j)}(\frac{\rho-1/2}{i})=0$ for all $n\ge1$
and $0\le j<m(\rho)$, where $\rho$ ranges over the nontrivial zeros of $\zeta(s)$,
then $f$ vanishes identically.

(ii) An even function $f$ in $\mathcal{H}_1$ that is divisible by \change{
$\zeta(\frac12+is)$} (in
the sense that \change{$f(s)/\zeta(\frac12+is)$} is holomorphic for $|\im\, z|<1/2+\eps$)
is
uniquely determined by the values $\widehat{f}(\frac{\log n}{4\pi})$, $n\ge1$.
\end{corollary}
It is worth emphasizing that both Theorem~\ref{thm:zeta} and the above corollary
are rather sensitive to the choice of interpolation points and break down if one
removes any single point from the set $\{\frac{\log n}{4\pi}\}_{n\ge 1}$ or from the
(multi)set of nontrivial zeros of $\zeta(s)$.

Both  \eqref{eq:RW} and \eqref{eq:add} rely crucially on
the functional equation
    \[ \pi^{-s/2} \Gamma(s/2) \zeta(s)=\pi^{-(1-s)/2} \Gamma((1-s)/2) \zeta(1-s), \]
but a principal distinction between them is that the deduction of the Riemann-Weil formula
starts from the Euler product representation of $\zeta(s)$, while formula \eqref{eq:add}
is tied to the Dirichlet series representation  of $\zeta(s)$. Hence we may think of the
two formulas as expressing respectively a multiplicative and an additive duality relation
between the zeta zeros and a distinguished sequence of integers. We observe that the
sequence of integers involved in~\eqref{eq:RW} (the prime powers \change{$n=p^k$,
corresponding to the nontrivial terms~$\Lambda(n)\widehat{f}(\frac{\log n}{2\pi})$
in~\eqref{eq:RW}}) is a rather sparse subsequence of the corresponding sequence appearing
in \eqref{eq:add} (the square-roots of the positive integers, \change{corresponding
to~$\widehat{f}(\frac{\log\sqrt{n}}{2\pi})$ in~\eqref{eq:add}}).

In view of this inclusion, we may think of \eqref{eq:RW} as arising from \eqref{eq:add} in the following way: The left-hand side of \eqref{eq:RW} defines a linear functional on $\mathcal{H}_1$, while the right-hand side gives the representation of this functional with respect to the basis functions of Theorem~\ref{thm:zeta}. This is the rationale for our way of writing the Riemann--Weil formula.

Assuming for a moment the truth of the Riemann hypothesis and that all the zeta zeros are simple, we may think of a formula like \eqref{eq:add} as a confirmative answer to the following Fourier analytic question: Is it possible to recover, in a non-redundant way, any sufficiently nice function~$f$ on the real line from samples of~$f$ and its Fourier transform~$\widehat{f}$ along two suitably chosen sequences in respectively the time and frequency domain? Here the recovery being non-redundant means that it fails as soon as any point is removed from either of the two sequences. Clearly, such a favorable situation requires a delicate interplay between the two sequences.

Radchenko and Viazovska~\cite{RV} have shown that one obtains a Fourier interpolation formula of desired type by choosing both sequences to be $\pm\sqrt{n}$ with~$n$ ranging over the nonnegative integers. For simplicity, we restrict again to the case of even functions.
\begin{theoremA}[\cite{RV}]
There exists a sequence of even Schwartz functions~$a_n\colon\RR\to\RR$
with the property that for every even Schwartz function
$f:\RR\to\RR$ we have
    \begin{equation} \label{eq:sqrt}
    f(x) = \sum_{n=0}^{\infty} f(\sqrt{n})a_n(x) + \sum_{n=0}^{\infty}\widehat{f}(\sqrt{n}) \widehat{a}_n(x) ,
    \end{equation}
where each of the two series on the right-hand side converges absolutely
for every real~$x$. The functions $a_n$ satisfy the following interpolatory properties: $a_n(\sqrt{m})=\delta_{n,m}$ and
$\widehat{a}_n(\sqrt{m})=0$ when $m\ge 1$, and in addition
\begin{equation} \label{eq:poisson} a_0(0)=\widehat{a}_0(0)=\frac{1}{2},  \quad a_{n^2}(0)=-\widehat{a}_{n^2}(0)=-1 , \quad a_{n}(0)=-\widehat{a}_{n}(0)=0 \quad \text{otherwise}. \end{equation}
\end{theoremA}
The non-redundancy of the representation~\eqref{eq:sqrt} follows from the specific properties of the functions~$a_n$ as shown in~\cite{RV}. We note in passing that the methods developed in the present paper allow us to sharpen this result considerably (see Section~\ref{sec:sqrt} below).

Returning to the general discussion, we note that the properties \eqref{eq:poisson} show that \eqref{eq:sqrt} becomes the Poisson summation formula
when evaluated at $x=0$. We have therefore a curious analogy between the
Poisson summation and Riemann--Weil formulas: They can both be viewed as
representing distinguished linear functionals in terms of a Fourier
interpolation basis. Both formulas owe their existence and importance to an
inherent algebraic structure, which in the first case is additive (periodicity)
and in the latter multiplicative.


To construct our interpolation formulas, we will use weakly holomorphic modular
forms for the theta group. The core ingredient in our construction is a
function of two complex variables~$w$ and~$s$, which in the case of even
functions takes the form
    \[ D(w, s)\coloneqq\sum_{n=1}^{\infty} \beta_{n}(s) n^{-w/2}, \]
with $\beta_n(s)$ being the Fourier coefficients of a certain $2$-periodic analytic function on the upper half-plane that is related through a Mellin
transform to the functions $F_{+}(x,\tau)$ considered in~\cite{RV}.
The Dirichlet series $w\mapsto D(w,s)$ converges absolutely
in a half-plane depending on $s$ and has a meromorphic extension to $\mathbb{C}$ with simple
poles at the three points $1,s,1-s$. A crucial point is that $D(w,s)$ is related to the Riemann
zeta function in the following precise way:
    \begin{equation}
    \label{eq:kernel}
    H(w,s)\coloneqq\frac{\zeta(s)}{\zeta(w)} D(w,s)
    \end{equation}
satisfies the functional equations $H(1-w,s)=-H(w,s)$ and $H(w,1-s)=H(w,s)$. These properties of $H(w,s)$, along with suitable estimates for $D(w,s)$ and a familiar contour integration argument applied to
    \begin{equation} \label{eq:cont}
    \frac{1}{2\pi i} \int_{1/2+\eps-i\infty}^{1/2+\eps+i\infty} f\left(\frac{w-1/2}{i}\right) H(w, iz+1/2) dw,
    \end{equation}
are what we will use to establish Theorem~\ref{thm:zeta}. It is essential that $H(w,s)$ is a Dirichlet series in~$w$ so that \eqref{eq:cont} produces a weighted sum of Fourier transforms of~$f$.

We may now observe that if we replace $\zeta(s)/\zeta(w)$ by $F(s)/F(w)$ in \eqref{eq:kernel}, with $F(s)$ an $L$-function satisfying a functional equation of the form
    \[ Q^{-s} \Gamma(s/2) F(s)=Q^{-(1-s)} \Gamma((1-s)/2) \overline{F(1-\overline{s})} \]
for some positive $Q$, then we still have a Dirichlet series in the variable $w$ that is amenable to our method of proof. This observation allows us to associate Fourier interpolation bases with the nontrivial zeros of all such $L$-functions. Hence the single function $D(w,s)$ generates an abundance of Fourier interpolation bases. We stress that this situation relies on a special multiplicative structure inherent in the present setting, namely that the class of Dirichlet series over exponentials of the form $n^{-w/2}$ is closed under multiplication.

The general phenomenon of Fourier interpolation bases may be thought of as ranging from Theorem~A via our Theorem~\ref{thm:zeta} to the ``degenerate'' situation related to the cardinal series
    \[ \sum_{n=-\infty}^{\infty} f(n) \frac{\sin \pi (z-n)}{\pi(z-n)}. \]
We find it enlightening to place our results in this more general context by considering two necessary conditions for a pair of sequences to generate Fourier interpolation bases.  First, we show that the existence of a kernel function with properties similar to those of the function $H(w,s)$ is a prerequisite for Fourier interpolation. This observation yields a precise notion of duality between the two sequences involved in a Fourier interpolation basis, closely aligned with modular relations as for example studied in some generality by Bochner \cite{B}. Second, we discuss a recent density theorem of Kulikov \cite{Ku}, which is a version of the uncertainty principle valid for Fourier interpolation bases. We observe that there is a precise correspondence between Kulikov's density condition and the Riemann--von Mangoldt formula for the density of the nontrivial zeros of zeta and $L$-functions.

\subsection*{Outline of the paper} We will begin our discussion in Section~\ref{sec:generalkernel} with some general considerations as outlined in the preceding paragraph. We have included in this section also a brief subsection pointing out that Fourier interpolation bases generate families of ``crystalline'' measures, a topic that in our context goes back to Guinand \cite{Gu} and that recently has received notable attention.
See for example the recent papers of Kurasov and Sarnak~\cite{KS},
Lev and Olevskii~\cite{LO}, and Meyer~\cite{Me}.

We then proceed in Section~\ref{sec:modint} to construct the modular integrals that are used to build the Dirichlet series $D(w,s)$ referred to above. This requires a fairly comprehensive discussion of modular forms for the theta group. This section builds largely on ideas that go back to Knopp~\cite{K3}, with an important additional ingredient from~\cite{RV}, namely, the construction of modular integrals using contour integrals with modular kernels. 

Based on the groundwork laid in Section~\ref{sec:modint}, we may proceed to prove a weak version of Theorem~\ref{thm:zeta}. By this we mean the following: We may prove that \eqref{eq:add} holds for functions $f$ that are analytic in a sufficiently wide strip and that has sufficient decay at $\pm \infty$. This is our rationale for proceeding to the proof of Theorem~\ref{thm:zeta} in Section~\ref{sec:basisconstr} and the corresponding results for other $L$-functions and Dirichlet series in Section~\ref{sec:lfunkernels}, postponing the most technical part of the proof to the later Section~\ref{sec:estimates}.
We hope this choice of exposition will give the reader easier access to the main ideas underlying formula~\eqref{eq:add}.

Section~\ref{sec:estimates} contains precise estimates for the coefficients of $D(w,s)$, including bounds for associated partial sums.
The estimates obtained in this section appear to be close to optimal. Indeed, in certain ranges of the parameters that are involved, this may be concluded up to a logarithmic factor. By the results of this section, we obtain the precise quantitative restrictions on the function $f$ in Theorem~\ref{thm:zeta}. We also obtain, as will be shown in the final Section~\ref{sec:sqrt}, a new version of Theorem~A with rather mild constraints on the function $f$ being represented by the Fourier interpolation formula \eqref{eq:sqrt}.

\section{Generalities on Fourier interpolation}
\label{sec:generalkernel}

The main purpose of this section is to show that Fourier interpolation bases
generically arise from certain kernels that we will \change{refer to as Dirichlet series
kernels for suggestive reasons}. We do not explicitly use the results of this section
anywhere else in the paper, but it does provide motivation for some of our constructions.

We start from the assumption that any ``reasonable'' function $f$ can be
represented as
	\begin{equation} \label{eq:fex} f(x)= \sum_{\lambda\in \Lambda} f(\lambda) g_{\lambda}(x) + \sum_{{\lambda}^*\in \Lambda^*}
	\widehat{f}(\lambda^*) h_{\lambda^*}(x),
	\end{equation}
where $\Lambda$ and $\Lambda^*$ are two sequences of real numbers with no finite accumulation point and the associated functions $g_{\lambda}(x)$ and $h_{\lambda^*}(x)$ also are ``reasonable'', so that convergence of the two series is ensured. We also require that this representation behaves in the expected way under Fourier \change{transformation}, so that
	\[ \widehat{f}(x)= \sum_{\lambda\in \Lambda} f(\lambda) \widehat{g_{\lambda}}(x) + \sum_{{\lambda}^*\in \Lambda^*}
	\widehat{f}(\lambda^*) \widehat{h_{\lambda^*}}(x).\]

The basic phenomenon that we observe is that the two (general)
Dirichlet series
	\[ \sum_{\lambda^*\ge 0}  h_{\lambda^*}(x) e^{-2\pi i \lambda^* z}
	\quad \text{and} \quad
	\sum_{\lambda^*\le 0}  h_{\lambda^*}(x) e^{-2\pi i \lambda^* z} \]
admit meromorphic extension to~$\mathbb{C}$ for every~$x$, with
simple poles at~$x$ and at the points of the dual sequence~$\Lambda$.
Moreover, the two Dirichlet series are intertwined by a functional equation.
By duality, an analogous result holds if we reverse the roles of the two
sequences and replace $h_{\lambda^*}(x)$ by $\widehat{g_{\lambda}}(\xi)$.
Conversely, as will be demonstrated in concrete terms in later sections
of this paper, Dirichlet series kernels with such properties generate,
by contour integration, Fourier interpolation formulas, with the range
of validity depending on specific quantitative properties of these kernels.

Guided by the canonical case when the two sequences~$\Lambda$
and~$\Lambda^*$ consist of the same points $\pm\sqrt{n}$, $n=0,1,2,...$,
we will assume that one of the sequences satisfies a sparseness condition
asserting that there is an entire function vanishing on~$\Lambda$
whose growth is at most of order~2 and finite type in any horizontal strip.
In what follows, we will let~$\Lambda$ be the sequence enjoying
this property.

Before turning to precise results about general Dirichlet series
kernels, we would like to point out that more liberal assumptions
could certainly be made, such as~$\Lambda$ and~$\Lambda^*$ be multisets
(so that derivatives appear in~\eqref{eq:fex}) or the sequence~$\Lambda$
be located in a strip. The assumptions of Theorem~\ref{thm:genker}
below represent a trade-off between describing a general phenomenon
and avoiding excessive technicalities and inessential difficulties.

\subsection{The Dirichlet series kernel associated with $\Lambda$}
In what follows, we use the convention that a prime on a summation
sign, like in $\sum'_{\lambda^*}$, means that a possible term corresponding
to $\lambda^*=0$ should be divided by~2, while all other terms are
summed in the usual way.

\begin{theorem} \label{thm:genker}
	Let $\Lambda$ be a sequence of distinct real numbers such that
	there exists an entire function $G_{\Lambda}$ vanishing on $\Lambda$
	and satisfying the growth estimate $G_{\Lambda}(x+iy)\ll_{\eta} e^{cx^2}$
	for some positive $c$ in every strip $|y|\le \eta$. Let $\Lambda^*$
	be another locally finite subset of the real line, and suppose there exist
	associated sequences of functions $g_{\lambda}:\mathbb{R}\to \mathbb{C}$, $\lambda\in \Lambda$ and $h_{\lambda^*}:\mathbb{R}\to\mathbb{C}$, $\lambda^*\in \Lambda^*$
	with the following properties:
	\begin{itemize}
		\item[(a)] There exists a positive number $\eta_0$ such that for every real $x$, the two Dirichlet series
		$ E_{\pm}(x,z)\coloneqq2\pi i \sum'_{\mp \lambda^*\ge 0} h_{\lambda^*}(x)e^{-2\pi i \lambda^* z}$ converge absolutely for $\pm \im z\ge \eta_0$.
		\item[(b)] For every $\eps>0$ and $x$ in $\mathbb{R}$, $g_{\lambda}(x) e^{-\eps \lambda^2} \to 0$ when $|\lambda|\to \infty$.
		\item[(c)] For every $\eps>0$ and $z$ satisfying $|\im z|\ge \eta_0$, the function $f_{\eps,z}(x)\coloneqq e^{-\eps x^2}/(z-x)$
		can be represented as
		\begin{equation}\label{eq:fe}
		f_{\eps,z}(x)=\sum_{\lambda\in \Lambda} f_{\eps,z}(\lambda) g_{\lambda}(x) + \sum_{{\lambda}^*\in \Lambda^*}
		\widehat{f_{\eps,z}}(\lambda^*) h_{\lambda^*}(x).
		\end{equation}
	\end{itemize}
	Then for every $x$, the functions $z\mapsto E_{\pm}(x,z)$ extend to meromorphic functions with simple poles at $x$ and every point $\lambda$ in $\Lambda$ with respective residues $\pm 1$ and $\pm g_{\lambda}(x)$, and the functional equation
	\[ E_{+}(x,z)=-E_{-}(x,z) \]
	holds.
\end{theorem}
Note that the two assumptions (a) and (b) guarantee that the two
series in \eqref{eq:fe} converge absolutely.

\begin{proof}[Proof of Theorem~\ref{thm:genker}]
We fix $x$ and consider the function
	\[ F_{\eps}(z)\coloneqq \sum_{\lambda\in \Lambda} f_{\eps, z}(\lambda) g_{\lambda}(x), \]
which by (b) represents a meromorphic function in $\mathbb{C}$.
By \eqref{eq:fe}, we may write
	\begin{equation} \label{eq:Fout} F_{\eps}(z)=f_{\eps,z}(x)- \sum_{{\lambda}^*\in \Lambda^*}
	\widehat{f_{\eps,z}}(\lambda^*) h_{\lambda^*}(x) \end{equation}
when $|\im z| \ge \eta_0$. We may use assumption (a)
to control the sum on the right-hand side of \eqref{eq:Fout} on the
two lines $\im z = \pm\eta_0$. To this end, assume first that $\im z = \eta_0$.
Then since
	\begin{equation} \label{eq:fourier} \widehat{f_{\eps,z}}(\xi)=\frac{-2\pi^{3/2} i}{\sqrt{\eps}} \int_{-\infty}^{0} e^{-2\pi iw z} e^{-\pi^2 \eps^{-1}(\xi-w)^2} dw,\end{equation}
we have
	\begin{equation} \label{eq:fourier2} \widehat{f_{\eps,z}}(\xi) \ll \begin{cases} e^{2\pi \eta_0 \xi}, & \xi \le 0 ,\\ e^{-\pi^2 \eps^{-1} \xi^2}, & \xi > 0, \end{cases}\end{equation}
uniformly when $\im z=\eta_0$ and $0<\eps \le 1$, say. The same argument applies to
	\begin{equation} \label{eq:fourier3} \widehat{f_{\eps,z}}(\xi)=\frac{2\pi^{3/2} i}{\sqrt{\eps}} \int_{0}^{\infty} e^{-2\pi iw z} e^{-\pi^2 \eps^{-1}(\xi-w)^2} dw \end{equation}
when  $\im z=-\eta_0$, and hence, by \eqref{eq:Fout} and (a), $F_{\eps}(z)$ is uniformly bounded on $|\im z|=\eta_0$ for $0<\eps\le 1$.
This along with our assumption on the sequence $\Lambda$ implies that the function
	\[ F_{\eps}(z) G_{\Lambda}(z) e^{-cz^2} \]
is bounded on $|\im z|=\eta_0$, uniformly for $0<\eps\le 1$. It is also clear by assumption (b) and the sparseness of
$\Lambda$ that there exist $t_n\to \infty $ such that $F_{\eps}(t_n+i y)\to 0$ for any fixed $\eps$, uniformly when $|y|\le \eta_0$.
Similarly, there exist $\tau_n\to -\infty $ such that $F_{\eps}(\tau_n+i y)\to 0$ for any fixed $\eps$, uniformly when $|y|\le \eta_0$. Hence, by the maximum modulus principle, $F_{\eps}(z) G_{\Lambda}(z) e^{-cz^2}$ is bounded in the strip $|\im z|\le \eta_0$. Now a normal family argument shows that when $\eps \to 0$, $F_{\eps}(z) $ tends locally uniformly to a meromorphic function $F(z)$ with simple poles at the sequence $\Lambda$. On the other hand, it follows from \eqref{eq:fourier} and \eqref{eq:fourier3} along with the uniform bound \eqref{eq:fourier2} that
	\[ F(z)=\begin{cases}\frac{1}{z-x}+E_+(x,z), & \im z=\eta_0 \\
	\frac{1}{z-x}-E_-(x,z), & \im z=- \eta_0 . \end{cases}  \]
This relation yields the asserted meromorphic continuation of the two functions $E_{\pm}(x,z)$ as well as the functional equation $E_+(x,z)=-E_-(x,z)$.
\end{proof}


\subsection{The Dirichlet series kernel associated with $\Lambda^*$}
In the preceding section, we put a sparseness condition on $\Lambda$ to control the growth of the entire function $G_{\Lambda}$. In contrast, the dual sequence $\Lambda^*$ could be arbitrarily dense, and this means that the Phragm\'{e}n--Lindel\"{o}f-type argument used above would not work to establish an analogue of Theorem~\ref{thm:genker}. This obstacle may be circumvented by relying instead on Theorem~\ref{thm:genker}. To avoid unnecessary technicalities, we begin by stating a result that follows quite easily from Theorem~\ref{thm:genker} without giving the exact analogue that we are aiming for.

In what follows, we will use the function
	\[ \Phi(x,w)\coloneqq  \int_{-i\eta_0}^{i\eta_0} e^{2\pi i zw} E_-(x,z) dz\]
\change{several times}.
Here the integral is to be interpreted in the principal value sense,
should $z\mapsto E_-(x,z)$ have a simple pole at~0.
We observe that $w\mapsto \Phi(x,w)$ is an entire function for every real~$x$. It will also be convenient to employ the usual notation $H(x)$ for
the Heaviside step function.

\begin{theorem}\label{thm:dualker}
	Let the assumptions be as in Theorem~\ref{thm:genker} and assume in addition the following:
	\begin{itemize}
		\item[(d)] There exists a positive number $\nu_0$ such that the
		two Dirichlet series
		$E^*_{\pm}(x,w)\coloneqq 2\pi i \sum_{\pm \lambda\ge 0}' g_{\lambda}(x) e^{2\pi i \lambda w}$
		converge absolutely for $\pm \im w \ge \nu_0$.
		\item[(e)] We have $E_{\pm}(x, t_n +i\eta)\ll e^{c |t_n|}$ for
		some~$c>0$, uniformly for $|\eta|\le \eta_0$ and
		for suitable sequences $\{t_n\}_{n\ge1}$ tending to $\pm \infty$.
	\end{itemize}
	Then
	\begin{align*} E^*_{+}(x,w)
	& =  - 2\pi i e^{2\pi i xw}H(x)  +  \sum_{\lambda^*\ge 0}\!\phantom{}^{'} h_{\lambda^*}(x)\frac{e^{-2\pi \eta_0 (w-\lambda^*)}}{(w-\lambda^*)}
	+  \sum_{\lambda^*\le 0}\!\phantom{}^{'} h_{\lambda^*}(x)\frac{e^{2\pi \eta_0 (w-\lambda^*)}}{(w-\lambda^*)}   + \Phi(x,w) \\
	E^*_{-}(x,w)
	& = - 2\pi i e^{2\pi i xw} - E_+^*(x,w) .\end{align*}
\end{theorem}

We see from the latter two expressions that the two functions $E_{\pm}^*(x,w)$ have
meromorphic extensions to $\mathbb{C}$. We also observe that in order to obtain the
desired counterpart to the functional equation {\color{blue} $E_+(x,z)=-E_-(x,z)$}, we should apply Fourier transformation in the
variable~$x$. Such a step would require some additional mild assumptions that we prefer
not to specify. If we take this step for granted and denote the Fourier transforms of
$x\mapsto E^*_{\pm}(x,w)$ by $\widehat{E^*_{\pm}}(\xi,w)$, then we get the desired
equation $\widehat{E^*_{+}}(\xi,w) =-\widehat{E^*_{-}}(\xi,w)$
and also that $\widehat{E^*_{\pm}}(\xi,w)$ has simple poles at~$\xi$ and
each point $-\lambda^*$ of $-\Lambda^*$ with respective residues~$\pm 1$
and~$\pm\widehat{h_{\lambda^*}}(\xi)$.

One could imagine situations in which assumption (d) fails, for instance because the nodes~$\lambda$ come arbitrarily close to each other. This could be remedied by using only assumption (e) to define $E_{\pm}^*(x,w)$ in a slightly more involved way in terms of convergence of sequences of partial sums. We find that nothing essential is lost by refraining from entering such technicalities.

\begin{proof}[Proof of Theorem~\ref{thm:dualker}]
We set
	\[ F_+(w)\coloneqq \frac{1}{2\pi i} \int_{-i\eta_0}^{\infty-i\eta_0} e^{2\pi i wz} E_-(x,z) dz \]
when $\im w>0$. This function is well-defined since $E_-(x,z)$ is uniformly bounded on the line of integration by assumption (a) of Theorem~\ref{thm:genker}. By absolute convergence of the Dirichlet series representation of $E_-(x,z)$, we may integrate termwise and obtain that
	\begin{equation} \label{eq:Fdef} F_+(w)=- \frac{1}{2\pi i} \sum_{\lambda^*\ge 0}\!\phantom{}^{'} h_{\lambda^*}(x)\frac{e^{2\pi \eta_0 (w-\lambda^*)}}{(w-\lambda^*)} .\end{equation}
Here the series on the right-hand side converges absolutely when $w$ is not one of the points~$\lambda^*$, and hence $F(w)$ extends to a meromorphic function in~$\CC$ with poles at~$\lambda^*$ with $\lambda^*\ge 0$. Adding the constraint that $\im w \ge \max(c,\nu_0)$, we now move the line of integration to $z=\xi+i\eta_0$, $\xi>0$. Using the functional equation $E_+(x,z)=-E_-(x,z)$ and assumption (e), we then find by the residue theorem that
	\begin{equation} \label{eq:Fsec}F_+(w)  =  -e^{2\pi i xw}H(x)-\sum_{\lambda\ge 0}\!\phantom{}^{'} g_{\lambda}(x) e^{2\pi i \lambda w}   - \frac{1}{2\pi i} \int_{i\eta_0}^{\infty+i\eta_0} e^{2\pi i wz} E_+(x,z) dz  + \frac{\Phi(x,w)}{2\pi i} .  \end{equation}
Hence using the Dirichlet series representation of $E_+(x,z)$ and integrating termwise in the first integral on the right-hand side of \eqref{eq:Fsec}, we obtain the alternate representation
	\begin{align*} F_+(w)  = & - e^{2\pi i xw}H(x)-\sum_{\lambda\ge 0}\!\phantom{}^{'} g_{\lambda}(x) e^{2\pi i \lambda w} \\ &
	+ \frac{1}{2\pi i} \sum_{\lambda^*\ge 0}\!\phantom{}^{'} h_{\lambda^*}(x)\frac{e^{-2\pi \eta_0 (w-\lambda^*)}}{(w-\lambda^*)}  + \frac{\Phi(x,w)}{2\pi i} . \end{align*}
Combining this with \eqref{eq:Fdef}, we find that
	\begin{align} \label{eq:Eplus} E_+^*(x,w)\coloneqq 2\pi i \sum_{\lambda\ge 0}\!\phantom{}^{'} g_{\lambda}(x) e^{2\pi i \lambda w}
	= & - 2\pi i e^{2\pi i xw}H(x)  +  \sum_{\lambda^*\ge 0}\!\phantom{}^{'} h_{\lambda^*}(x)\frac{e^{-2\pi \eta_0 (w-\lambda^*)}}{(w-\lambda^*)} \\
	& +  \sum_{\lambda^*\le 0}\!\phantom{}^{'} h_{\lambda^*}(x)\frac{e^{2\pi \eta_0 (w-\lambda^*)}}{(w-\lambda^*)}   + \Phi(x,w) , \nonumber \end{align}
which yields the required expression for $E_+^*(x,w)$.
By similar calculations applied to the function
	\[ F_-(w)\coloneqq \frac{1}{2\pi i} \int_{-\infty-i\eta_0}^{-i\eta_0} e^{2\pi i wz} E_-(x,z) dz \]
for $\im w \le \max(c,\nu_0)$, we arrive at the representation
	\begin{align*}
	E_-^*(x, w) = & - 2\pi i e^{2\pi i xw}H(-x)
	- \sum_{\lambda^*\ge 0}\!\phantom{}^{'} h_{\lambda^*}(x)\frac{e^{-2\pi \eta_0 (w-\lambda^*)}}{(w-\lambda^*)} \\
	& - \sum_{\lambda^*\le 0}\!\phantom{}^{'} h_{\lambda^*}(x)\frac{e^{2\pi \eta_0(w- \lambda^*)}}{(w-\lambda^*)}  - \Phi(x,w).
	\end{align*}
Combining this formula with \eqref{eq:Eplus}, we obtain the required
relation between $E_+^*(x,w)$ and $E_-^*(x,w)$.
\end{proof}

\subsection{Examples}$\phantom{}$
We illustrate the above discussion with two examples where the corresponding
Fourier interpolation identity is known: the Whittaker--Shannon interpolation
formula and the Fourier interpolation formula from~\cite{RV}.
Theorem~\ref{thm:zeta}, one of the main results of this
paper, yields a third example that will be treated in Section~\ref{sec:zetakernels}; a large family of related formulas will then be presented in the subsequent Section~\ref{subsect:other}.
\subsubsection{The Paley--Wiener case.}
Suppose that $f$ is such that the periodized function
	\[ F(y)\coloneqq \sum_{n\in \mathbb{Z}} \widehat{f}(y+n) \]
is well-defined and in $L^2(-1/2,1/2) $. Then we may express $f$ as
	\begin{align*} \label{eq:gencard} f(x) & =\sum_{n \in  \mathbb{Z}} f(n) \operatorname{sinc}(\pi(x-n)) + \int_{-\infty}^{\infty} \left(\widehat{f}(y)-F(y) \mathds{1}_{[-1/2,1/2]}(y)\right)  e^{2\pi i xy} dy \\
	\nonumber & =  \sum_{n \in  \mathbb{Z}} f(n) \operatorname{sinc}(\pi(x-n)) + \int_{|y|\ge 1/2} \widehat{f}(y)e^{2\pi i xy}\Big(1-\sum_{n\neq 0}\mathds{1}_{[n-1/2,n+1/2]}(y)e^{-2\pi i xn}\Big) dy.\end{align*}
We may think of this formula as representing the degenerate case when
$\Lambda=\mathbb{Z}$ and $\Lambda^*=(-\infty,-1/2]\cup [1/2,\infty)$.
The associated kernels are
	\begin{align*} E_{\pm}(x,z) & \coloneqq 2\pi i\int_{\mp y \ge 1/2} e^{-2\pi i y(z-x)}\Big(1-\sum_{n\neq 0}\mathds{1}_{[n-1/2,n+1/2]}(y)e^{-2\pi i xn}\Big) dy ,\\
	\widehat{E^*_{\pm}}(\xi,w) & : =\mathds{1}_{[-1/2,1/2]}(\xi)\Big(\pi i+2\pi i \sum_{n=1}^{\infty} e^{\pm 2\pi i  n (w-\xi)}\Big). \end{align*}
We may in this case compute their meromorphic continuations explictly:
	\begin{align*} E_{\pm}(x,z) & = \mp \Big(\frac{e^{\pi i (z-x)}}{z-x} - \frac{\pi e^{-\pi i z}}{\sin \pi z} \operatorname{sinc} \pi(z-x) \Big),\\
	\widehat{E^*_{\pm}}(\xi,w) & =\pm \pi  \mathds{1}_{[-1/2,1/2]}(\xi) \cot\pi(w-\xi). \end{align*}
The latter kernel has a pole of residue $\pm 1$ at $\xi$; all other poles are located in $\Lambda^*$, and the collection of all such poles when $\xi$ varies in $[-1/2,1/2]$ is indeed the entire set $\Lambda^*$.

\subsubsection{The $\sqrt{n}$ case.}
We can reinterpret the results of~\cite{RV} in terms of Dirichlet series
kernels as follows.
As $\Lambda=\Lambda^*=\{\pm\sqrt{n}\}_{n\in\ZZ}$, all of the identities
can be symmetrized over $x\mapsto -x$. In particular, this implies that
we may assume $E_{-}(x,z)=E_{+}(-x,-z)$. It is convenient to split the kernels
into even and odd parts. First, we look at the even kernels
$\frac{1}{2}(E_{+}(x,z)+E_{+}(-x,z))$.
Since for all even Schwartz functions $f$ we have
	\[f(x) = \sum_{n\ge0}a_n(x)f(\sqrt{n})
	+\sum_{n\ge0}\widehat{a_n}(x)\widehat{f}(\sqrt{n}),\]
we get
	\begin{align*}
	\frac{1}{2}(E_{+}(x,z)+E_{+}(-x,z))
	= 2\pi i\sum_{n\ge0}\widehat{a_n}(x)e^{2\pi i \sqrt{n}z}
	,\\
	\frac{1}{2}(E_{+}^{*}(x,z)+E_{+}^{*}(-x,z))
	= 2\pi i\sum_{n\ge0}a_n(x)e^{2\pi i \sqrt{n}z}
	.
	\end{align*}
Theorem~\ref{thm:genker} and Theorem~\ref{thm:dualker} in this case tell
us that the functions $z\mapsto E_{+}(x,z)$ and $z\mapsto E_{+}^{*}(x,z)$,
given by a general Dirichlet series over~$e^{2\pi i \sqrt{n} z}$, $n\ge0$ in the upper
half-plane extend to meromorphic functions in~$\CC$ with simple poles
at $\pm\sqrt{n}$, as well as a simple pole at $\pm x$. Moreover, the
analytic extension to the lower half-plane in each case is given by a general Dirichlet series over~$e^{-2\pi i \sqrt{n} z}$, $n\ge0$.

One can treat the odd kernels $\frac{1}{2}(E_{+}(x,z)-E_{+}(-x,z))$
similarly. In this case we use the interpolation formula for odd
functions~\cite[Thm.~7]{RV}
	\[f(x)= c_0(x)\frac{f'(0)+i\change{\hat{f}'(0)}}{2}
	+\sum_{n\ge1}c_n(x)\frac{f(\sqrt{n})}{\sqrt{n}}
	-\sum_{n\ge1}\widehat{c_n}(x)\frac{\widehat{f}(\sqrt{n})}{\sqrt{n}}.\]
From this we obtain
	\begin{align*}
	\frac{1}{2}(E_{+}(x,z)-E_{+}(-x,z))
	= (2\pi i)\big(-\widehat{c_0}(x)(\pi i z)-\sum_{n\ge 1}
	\frac{\widehat{c_n}(x)}{\sqrt{n}}e^{2\pi i \sqrt{n}z}\big)
	\end{align*}
and an analogous expression for the dual odd
kernel~$\frac{1}{2}(E_{+}^{*}(x,z)-E_{+}^{*}(-x,z))$.
A new feature in the odd case is a pole of order two at $z=0$, which
corresponds to the fact that the interpolation formula involves~$f'(0)$
and~\change{$\hat{f}'(0)$}.

\subsection{The joint density of $\Lambda$ and $\Lambda^*$}
We now come to a basic necessary condition for existence of formulas like \eqref{eq:fex}, that was recently established by Kulikov \cite{Ku}. This condition yields a joint bound for the two counting functions $N_{\Lambda}(T)$ and $N_{\Lambda^*}(W)$, which we define as the number of points from the respective sequences to be found in the two intervals $[-T,T]$ and $[-W,W]$. Kulikov made the  assumptions that $N_{\Lambda}(T)\ll T^L$ for some positive integer $L$ and that the functions $g_{\lambda}(x)$ be polynomially bounded in the two variables $\lambda$ and $x$. Assuming also the validity of \eqref{eq:fex} for all functions $f$ with $C^\infty$-smooth and compactly supported Fourier transform, he showed that for every $\eta>0$, there exists a positive constant $C$ such that
	\begin{equation} \label{eq:TF} N_{\Lambda}(T) +N_{\Lambda*}(W)\ge 4WT -C \log^{2+\eta}(4WT)\end{equation}
holds whenever $W, T \ge 1$. This result relies on sharp estimates of Karnik, Romberg, and Davenport
\cite{KRD} for the eigenvalue distribution of time-frequency
localization operators. We may view \eqref{eq:TF} as a manifestation of
the uncertainty principle as discussed for instance in the work of
Slepian~\cite{S}. To simplify matters, we have again suppressed the
possibility that $\Lambda$ and $\Lambda^*$ be multi-sets which however
is accounted for in~\cite{K}.

We observe that in the $\sqrt{n}$ case, $N_{\Lambda}(T)=2T^2+O(1)$ and $N_{\Lambda^*}(W)=2W^2+O(1)$, so that \eqref{eq:TF} holds since
$T^2+W^2\ge 2WT$. To relate Kulikov's bound to Theorem~\ref{thm:zeta}, we let $\Lambda$ consist of the points $(\rho-1/2)/i$ and $\Lambda^*$ be the sequence of points $\pm (\log n)/(4\pi)$ for $n\ge 1$. Then $N_{\Lambda}(T)=2N(T)$ and $N_{\Lambda^*}(W)=e^{4\pi W}+O(1)$, where we in the first relation use the standard notation $N(T)$ for the usual counting function for the nontrivial zeros of~$\zeta(s)$. Then \eqref{eq:TF} yields
	\[ 2N(T)+e^{4\pi W}\ge 4WT -C \log^{2+\eta} (4WT). \]
If we now set $W=(\log T-\log (2\pi))/(4\pi)$, then we get
	\[ N(T)\ge \frac{T}{2\pi} \log \frac{T}{2\pi e} - C \log^{2+\eta} T , \]
which clearly holds in view of the Riemann-von Mangoldt formula
	\begin{equation} \label{eq:RvM}
	N(T) = \frac{T}{2\pi} \log \frac{T}{2\pi e} + O(\log T).
	\end{equation}
There is a similar precise relation between Kulikov's bound \eqref{eq:TF} and the Riemann-von Mangoldt formula for any $L$-function to which the methods developed in this paper apply. We will return to this point in Subsection~\ref{subsect:other}.

We should like to emphasize that Kulikov does not assume minimality of the system of functions $g_{\lambda}(x)$ and $h_{\lambda^*}(x)$. It seems reasonable to expect that an assumption about minimality should imply a sparseness condition that would complement \eqref{eq:TF}. It would be interesting to see if a general version of the Riemann--von Mangoldt formula \eqref{eq:RvM} (though with a less precise remainder term) could be obtained as a consequence of \eqref{eq:TF} along with such a sparsity condition.

\subsection{Fourier interpolation and crystalline measures} It is immediate that a formula like \eqref{eq:fex} should imply that the distributional Fourier transform of
\[ \mu_x:=\delta_x - \sum_{\lambda\in \Lambda} g_{\lambda}(x) \delta_{\lambda} \]
will be
\[ \widehat{\mu_x}=\sum_{\lambda^*\in \Lambda*} h_{\lambda^*}(x) \delta_{\lambda^*},\]
where as usual $\delta_{\xi}$ is \change{the unit mass} at the point $\xi$. This means
that any Fourier interpolation formula as a byproduct generates a whole family of measures
that are crystalline in an appropriate sense. We refer to \cite{KS, LO, Me} for some
interesting recent results on crystalline measures and Fourier quasicrystals. It is our
impression that the results of the present paper, while perhaps shedding some light on
Dyson's thoughts on the Riemann hypothesis in his acclaimed lecture~\change{\cite{Dy}},
adds further evidence to the common belief (see \cite{KS}) that a classification of such
measures would
probably be very difficult to obtain.

\section{Modular integrals for the theta group}
\label{sec:modint}
In this section we construct a family of special functions (modular integrals)
on the complex upper half-plane $\HH\coloneqq\{\tau\in\CC\colon \im \tau>0\}$
whose Mellin transforms form the building blocks for our
Dirichlet series kernels.

The definition of these functions is most naturally viewed in terms of
Eichler cohomology for the theta group. Nevertheless, they have a simple
elementary description: we are interested in 2-periodic analytic functions
$F\colon\HH\to\CC$ that are of moderate growth (see below) and satisfy
    \begin{equation}\label{eq:feq-preliminary}
    F(\tau) - \eps(\tau/i)^{-k}F(-1/\tau)
    = (\tau/i)^{-s} - \eps(\tau/i)^{s-k}.
    \end{equation}
Here $\eps\in\{\pm1\}$, $k\ge 0$, and $s\in\CC$.
In what follows, we will always interpret the expression $(z/i)^{\alpha}$
for $z\in\HH$ as the principal branch, i.e., $(z/i)^{\alpha}$ takes
the value $x^{\alpha}$ for $z=ix$, $x>0$
(equivalently, $(z/i)^{\alpha}=e^{\alpha\,\mathrm{Log}(z/i)}$).

The condition~\eqref{eq:feq-preliminary} is not sufficient to uniquely pinpoint
the function $F$. Nevertheless, it determines~$F$ uniquely modulo a
finite-dimensional space of modular forms if we additionally require~$F$ to be
of moderate growth. Following Knopp~\cite{K3}, we say that
a function $\phi\colon\HH\to\CC$ is of moderate growth if
    \[\phi(\tau) \ll \im(\tau)^{-\alpha}+|\tau|^{\beta},\qquad \tau\in\HH,\]
where $\alpha$ and $\beta$ are some positive constants.
Equivalently, $\phi$ is of moderate growth if and only if for some $r>0$
we have $|\phi(i\frac{1-z}{1+z})|\ll(1-|z|)^{-r}$ for all~$z$ in the unit
disk $|z|<1$.

For a 2-periodic function $F$ moderate growth is tantamount to
having a Fourier expansion
	\[F(\tau) = \sum_{n\ge 0}a_ne^{\pi i n\tau},\qquad \tau\in\HH,\]
where the sequence $\{a_n\}_{n\ge0}$ has polynomial growth.
To make the solution unique, we require in addition that
the first few coefficients $a_n$ vanish. More precisely,
we require $a_n=0$ for $n<\nu_{\eps}$, where we set
    \begin{equation} \label{eq:nupmdef}
    \nu_{-}=\nu_{-}(k)\coloneqq \Big\lfloor\frac{k+2}{4}\Big\rfloor,\qquad
    \nu_{+}=\nu_{+}(k)\coloneqq \Big\lfloor\frac{k+4}{4}\Big\rfloor.
    \end{equation}

\begin{theorem}
    \label{thm:modint}
    If $\phi\colon\HH\to\CC$ is an analytic function of moderate growth,
    then for any $k\ge0$ and $\eps\in\{\pm 1\}$ there exists a unique
    2-periodic analytic function $F\colon\HH\to\CC$	of moderate growth with
    a Fourier expansion of the form
    \[F(\tau) = \sum_{n\ge \nu_{\eps}}a_ne^{\pi i n \tau},
    \qquad \tau\in\HH\]
    such that
    \begin{equation}
    \label{eq:fksfeq_general}
    F(\tau) - \eps(\tau/i)^{-k}F(-1/\tau) =
    \phi(\tau) - \eps(\tau/i)^{-k}\phi(-1/\tau).
    \end{equation}
\end{theorem}
The proof of uniqueness will come as a simple corollary of some basic
properties of modular forms for the theta group (see Proposition~\ref{prop:mfdim}), while the proof of existence follows from Proposition~\ref{prop:modkereq}.

Let us denote the function~$F$ from Theorem~\ref{thm:modint}
by $F_{k}^{\eps}(\tau,\phi)$. Since $\phi_s(\tau)=(\tau/i)^{-s}$ is of
moderate growth in~$\HH$ for any $s\in\CC$, there is a unique function
$F_k^{\pm}(\tau,s)\coloneqq F_k^{\pm}(\tau,\phi_s)$ with a Fourier expansion
	\begin{equation*}
	F_k^{\pm}(\tau,s) = \sum_{n\ge \nu_{\pm}}
	\alpha_{n,k}^{\pm}(s)e^{\pi i n \tau}
	\end{equation*}
such that
	\begin{equation*}
	F_k^{\eps}(\tau,s) - \eps(\tau/i)^{-k}F_k^{\eps}(-1/\tau,s) =
	(\tau/i)^{-s} - \eps(\tau/i)^{s-k}.
	\end{equation*}
This is exactly the function that we are interested in.

\begin{remark}
For $k>2$ the existence part of Theorem~\ref{thm:modint} follows from the
results of Knopp on Eichler cohomology~\cite{K3}. Instead of this we use a
construction with contour integrals as in~\cite{RV}
(in Section~\ref{sec:modkernels} below we will sketchily explain the motivation
behind this construction). The main reason for doing this is, first, because
the construction works for all $k\ge0$, and second, since it can be used to
give relatively good estimates for the size of the coefficients
$\alpha_{n,k}^{\pm}(s)$ as $n\to\infty$, at least in the range $0\le k\le 2$.

Let us also note that for $k=0$ the existence of the decomposition
\eqref{eq:fksfeq_general} is related to the result of Hedenmalm
and Montes-Rodriguez~\cite{HMR} that the system of functions $e^{i\pi nx}$,
$e^{i\pi n/x}$, $n\in \ZZ$ is weak-star complete in~$L^{\infty}(\RR)$.
\end{remark}

\subsection{Preliminaries on the theta group}
The group $\SL_2(\RR)$ of $2\times 2$ real matrices with determinant~1
acts in the usual way on the upper half-plane~$\HH$ by
    \[\gamma \tau = \frac{a\tau+b}{c\tau+d},\quad
    \gamma=\pmat abcd\in\SL_2(\RR).\]
Since the matrix $-I=(\smat{-1}{0}{0}{-1})$ acts trivially, we will
work with the group $\PSL_2(\RR)\coloneqq \SL_2(\RR)/\{\pm I\}$ instead, but we
still prefer to write the elements of $\PSL_2(\RR)$ as matrices.
Let us denote
    \[S\coloneqq \pmat{0}{-1}{1}{0}, \quad T\coloneqq \pmat1101.\]

The theta group $\Gamma_{\theta}\subset \PSL_2(\ZZ)$ is the subgroup
generated by $S$ and $T^2$.
The group~$\Gamma_{\theta}$ consists of all the elements of $\PSL_2(\ZZ)$
congruent to $(\smat 1001)$ or $(\smat 0110)$ modulo~2
(see~\cite[p.7 Cor.~4]{K5}).
The only relation between the generators of~$\Gamma_{\theta}$
is $S^2=1$. This implies that any element $\gamma\in\Gamma_{\theta}$ can
be written in a unique way as
$\gamma=S^{\eps_0}T^{2m_1}ST^{2m_2}\dots ST^{2m_k}S^{\eps_1}$,
where $\eps_j\in\{0,1\}$, which we call the canonical word
or the canonical representation for $\gamma$.

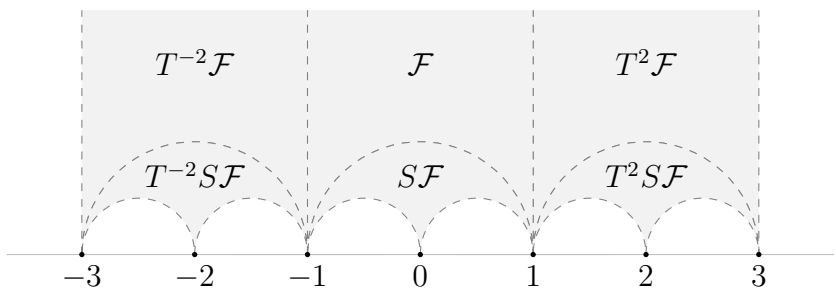
\begin{figure}[h]
    \centering
    \begin{tikzpicture}
    \definecolor{cv0}{rgb}{0.95,0.95,0.95}
    \definecolor{cv1}{rgb}{0.90,0.90,0.90}
    \clip(-9,-0.5) rectangle (7,3.5);

    \begin{scope}[scale=0.5, xshift=5cm]
    \draw[lightgray] (-5,0) -- (5,0);
    \fill[color=cv0]  (3,6.5) -- (3,0) arc (0:180:3) -- (-3,6.5);
    \draw[gray,dashed] (-3,0)  --  (-3,6.5);
    \draw[gray,dashed] (3,0)  --  (3,6.5);
    \draw[gray,dashed] (-3,0) arc  (180:0:3);
    \fill[color=cv0]  (-3,0) arc (180:0:3) arc (0:180:1.5) arc (0:180:1.5);
    \draw[gray,dashed]  (-3,0) arc (180:0:3) arc (0:180:1.5) arc (0:180:1.5);
    \draw (0,1.5) node[above]{$T^2S\bF$};
    \draw (0,4.5) node[above]{$T^2\bF$};
    \end{scope}

    \begin{scope}[scale=0.5, xshift=-7cm]holomorphic
    \draw[lightgray] (-5,0) -- (5,0);
    \fill[color=cv0]  (3,6.5) -- (3,0) arc (0:180:3) -- (-3,6.5);
    \draw[gray,dashed] (-3,0)  --  (-3,6.5);
    \draw[gray,dashed] (3,0)  --  (3,6.5);
    \draw[gray,dashed] (-3,0) arc  (180:0:3);
    \fill[color=cv0]  (-3,0) arc (180:0:3) arc (0:180:1.5) arc (0:180:1.5);
    \draw[gray,dashed]  (-3,0) arc (180:0:3) arc (0:180:1.5) arc (0:180:1.5);
    \draw (0,1.5) node[above]{$T^{-2}S\bF$};
    \draw (0,4.5) node[above]{$T^{-2}\bF$};
    \end{scope}

    \begin{scope}[scale=0.5, xshift=-1cm]
    \draw[lightgray] (-5,0) -- (5,0);
    \fill[color=cv0]  (3,6.5) -- (3,0) arc (0:180:3) -- (-3,6.5);
    \draw[gray,dashed] (-3,0)  --  (-3,6.5);
    \draw[gray,dashed] (3,0)  --  (3,6.5);
    \draw[gray,dashed] (-3,0) arc  (180:0:3);
    \draw (0,4.5) node[above]{$\bF$};
    \fill[color=cv0]  (-3,0) arc (180:0:3) arc (0:180:1.5) arc (0:180:1.5);
    \draw[gray,dashed]  (-3,0) arc (180:0:3) arc (0:180:1.5) arc (0:180:1.5);
    \draw (0,1.5) node[above]{$S\bF$};
    \fill[black] (-9,0) circle (0.07) node[below] {$-3$};
    \fill[black] (-6,0) circle (0.07) node[below] {$-2$};
    \fill[black] (-3,0) circle (0.07) node[below] {$-1$};
    \fill[black] (0,0) circle (0.07) node[below] {$0$};
    \fill[black] (3,0) circle (0.07) node[below] {$1$};
    \fill[black] (6,0) circle (0.07) node[below] {$2$};
    \fill[black] (9,0) circle (0.07) node[below] {$3$};
    \end{scope}
    \end{tikzpicture}
    \label{fig:gammathetadomain}
    \caption{Fundamental domain for $\Gamma_{\theta}$ and some of its translates}
\end{figure}

A fundamental domain for $\Gamma_{\theta}$ is given by
(see Figure~\ref{fig:gammathetadomain})
    \[\bF \coloneqq \{z\in\HH \colon -1<\re z<1 ,\, |z|>1\}.\]
Since $\bF$ is a fundamental domain for the group $\Gamma_{\theta}$,
for any $\tau\in\HH$ there exists an element
$\gamma=\gamma_{\tau}\in\Gamma_{\theta}$ such that
$\gamma\tau$ is in $\ol{\bF}$. Moreover, if $\tau$ does not
belong to the set $\bigcup_{\gamma\in\Gamma_{\theta}}\partial \bF$
(which is nowhere dense and of measure $0$), then the
element $\gamma$ is unique, and otherwise there are
at most two such elements: $\{\gamma,S\gamma\}$
or $\{\gamma,T^{2}\gamma\}$.
The element $\gamma_{\tau}$ can be found by repeatedly performing the following operation: first apply some power of $T^{2}$ to get $\tau$ into the
strip $\{|\re \tau|\le1\}$, and then, if the resulting point is not yet in
the fundamental domain, apply the inversion~$S$.

\subsection{Modular forms for the theta group}
In this subsection we will collect the necessary basic facts about
modular forms for the theta group. A more detailed
exposition can be found in~\cite[Ch.~6]{BK}.

Let $\theta(\tau)$ be the Jacobi theta function
    \begin{equation*} \label{eq:thetadef}
    \theta(\tau) \coloneqq \sum_{n\in \ZZ}e^{\pi i n^2 \tau}.
    \end{equation*}
The function~$\theta\colon\HH\to\CC$ is holomorphic,
and it satisfies the transformations
    \[(\tau/i)^{-1/2}\theta(-1/\tau) = \theta(\tau)
    ,\qquad \theta(\tau+2)=\theta(\tau),\]
which correspond to the two generators of the theta group~$\Gamma_{\theta}$.
More generally, for any $\gamma=(\smat abcd)\in\Gamma_{\theta}$
with $c>0$ or $c=0,d>0$, we have
    \[\theta(\tau) =
    \zeta_\gamma(c\tau+d)^{-1/2}\theta\Big(\frac{a\tau+b}{c\tau+d}\Big),\]
where $(c\tau+d)^{-1/2}$ is the principal branch and $\zeta_{\gamma}$
is a certain $8$th root of unity that can be written explicitly
in terms of Jacobi symbols (see~\cite[Th.~7.1]{M}). Finally,
as a corollary of the Jacobi triple product identity, one has
$\theta(\tau) = \frac{\eta^5(\tau)}{\eta^2(2\tau)\eta^2(\tau/2)}$,
where $\eta(\tau)\coloneqq q^{1/24}\prod_{n\ge1}(1-q^n)$ is
the Dedekind eta function, and thus we see that $\theta(\tau)$ does not vanish
anywhere in~$\HH$. Here and in what follows we define the nome~$q$ by
$q\coloneqq e^{2\pi i \tau}$ and for arbitrary rational number $r$
we will interpret $q^r$ as $e^{2\pi i r\tau}$.

\subsubsection{The theta automorphy factor.}
We define the theta automorphy factor $j_{\theta}(\tau,\gamma)$ by
    \[j_{\theta}(\tau,\gamma) \coloneqq  \frac{\theta(\tau)}{\theta(\gamma\tau)}
    ,\qquad \gamma\in\Gamma_{\theta}.\]
It satisfies $j_{\theta}(\tau,\gamma_1\gamma_2)=j_{\theta}(\tau,\gamma_2)j_{\theta}(\gamma_2\tau,\gamma_1)$ and $j_{\theta}(\tau,\gamma)^8=(c\tau+d)^{-4}$. We define the
slash operator in weight $k$ with theta automorphy factor by
    \[(f|_{k}\gamma)(\tau) = j_{\theta}^{2k}(\tau,\gamma)f(\gamma\tau).\]
It is easy to see that this formula defines a right action of $\Gamma_{\theta}$
on the space of functions $f\colon\HH\to\CC$. More generally,
let $\chi_{\eps}\colon\Gamma_{\theta}\to\{\pm1\}$, where $\eps=\pm 1$
be the homomorphism defined by $\chi_{\eps}(T^2)=1$
and $\chi_{\eps}(S)=\eps$. We then define
    \[(f|_{k}^{\eps}\gamma)(\tau)
    := \chi_{\eps}(\gamma)j_{\theta}^{2k}(\tau,\gamma)f(\gamma\tau).\]
Note that all of the above definitions remain valid for real $k\ge0$
(and in fact for all complex~$k$),
if we interpret $j_{\theta}^{2k}(\tau,\gamma)$
as $\frac{\theta^{2k}(\tau)}{\theta^{2k}(\gamma\tau)}$
and $\theta^{2k}(\tau)$ using the principal branch, i.e.,
	\[
	\theta^{a}(\tau) \coloneqq  \exp\left(a\int_{i\infty}^{\tau}\frac{\theta'(z)}{\theta(z)}dz\right)
	,  \qquad   a\in\CC  .
	\]

\subsubsection{Modular forms for $\Gamma_{\theta}$.}
We define $M_k(\Gamma_{\theta},\eps)$ to be the space of holomorphic
modular forms of weight $k$ with respect to the above slash action, i.e.,
$f\colon\HH\to\CC$ is in $M_k(\Gamma_{\theta},\eps)$ if and only
if $f$ is a holomorphic function of moderate growth and $f|_k^{\eps}\gamma=f$
for all $\gamma\in\Gamma_{\theta}$.
We also denote by $M_k^{!}(\Gamma_{\theta},\eps)$ the space of
weakly holomorphic modular forms of weight $k$: a holomorphic function
$f\colon\HH\to\CC$ belongs to $M_k^{!}(\Gamma_{\theta},\eps)$
if $f|_k^{\eps}\gamma=f$ for all $\gamma\in\Gamma_{\theta}$
and its Fourier expansion at each of the cusps has at most finitely many
negative powers of $q$ (i.e.,~$f$ has at worst poles at the cusps).

If we let
$J(\tau)=J_{+}(\tau)\coloneqq\frac{16}{\lambda(\tau)(1-\lambda(\tau))}=(\frac{\theta(\tau)}{\eta(\tau)})^{12}$
and $J_-(\tau)\coloneqq1-2\lambda(\tau)$, where $\lambda(\tau)$
is the modular lambda invariant, then $J_{\pm}$ is in $M_0^{!}(\Gamma_{\theta},\pm)$.
Moreover, $J_+(\tau)$ is a Hauptmodul for the group~$\Gamma_{\theta}$ and
it maps the fundamental domain~$\bF$ conformally
onto the cut plane $\CC\sm(-\infty,64]$, as shown in Figure~\ref{fig:jconf}.
In particular, since $J(\tau)$ is a Hauptmodul, any
$f\in M_k^{!}(\Gamma_{\theta},\pm)$ can be written
as $f(\tau)=\theta^{2k}(\tau)J_{\pm}(\tau)P(J(\tau))$, where $P$ is some
Laurent polynomial. (A priori $P$ can be a rational function,
but since $f$ has poles only at the cusps and
$J(\tau)$ takes values $0$ and $\infty$ at the two cusps,
the poles of $P$ must be contained in $\{0,\infty\}$.)
Note that the identity
	\[M_k^{!}(\Gamma_{\theta},\pm) = \theta^{2k}J_{\pm}\CC[J,J^{-1}]\]
makes sense for all complex values of $k$ if we interpret $\theta^{2k}(\tau)$
as the principal branch.

\begin{figure}[h]
    \centering
    \begin{tikzpicture}
    \definecolor{cv0}{rgb}{0.95,0.95,0.95}
    \definecolor{cv1}{rgb}{0.90,0.90,0.90}
    \clip(-9,-0.5) rectangle (7,3.5);
    \begin{scope}[scale=0.5,xshift=-9cm]
    \draw[lightgray] (-5,0) -- (5,0);
    \fill[color=cv0] (-3,0) arc (180:90:3) -- (0,6.5) -- (-3,6.5);
    \fill[color=cv1] (0,3) arc (90:0:3) -- (3,6.5) -- (0,6.5);
    \draw[green] (-3,0)  --  (-3,6.5);
    \draw[green] (3,0)  --  (3,6.5);
    \draw[red] (-3,0) arc  (180:0:3);
    \draw[lightgray,dashed] (0,3)  --  (0,6.5);
    \draw (0.5,4.5) node[above]{$\bF$};

    \fill[black] (0,3) circle (0.07) node[below] {$i$};
    \fill[black] (-3,0) circle (0.07) node[below] {$-1$};
    \fill[black] (3,0) circle (0.07) node[below] {$1$};

    \draw (-4,4.5) node[above]{$\HH$};
    \end{scope}

    \draw [->] (-1.5,1) -- (0,1);
    \draw (-0.7,1) node[above] {$J$};

    \begin{scope}[scale=0.5,xshift=7cm]
    \fill[color=cv0] (-4,2.5) -- (4,2.5) -- (4,6.5) -- (-4,6.5);
    \fill[color=cv1] (-4,2.5) -- (4,2.5) -- (4,-4.5) -- (-4,-4.5);

    \draw[green] (-4,2.5)  --  (-1,2.5);
    \draw[red] (-1,2.5)  --  (1,2.5);
    \draw[lightgray,dashed] (1,2.5)  --  (4,2.5);
    \fill[black] (-1,2.5) circle (0.07) node[above] {$0$};
    \fill[black] (1,2.5) circle (0.07) node[above] {$64$};

    \draw[blue,dashed] (-1,2.5)  --  (-1,1.5);
    \draw[blue,dashed] (-1,1.5)  --  (1,1.5);
    \draw[blue,dashed] (1,1.5)  --  (1,2.5);
    \draw (0,0.6) node[above]{{\small $\ell$}};
    \draw (-3,4.5) node[above]{$\mathbb{C}$};
    \end{scope}
    \end{tikzpicture}
    \caption{$J(z)$ as a conformal map.}
    \label{fig:jconf}
\end{figure}
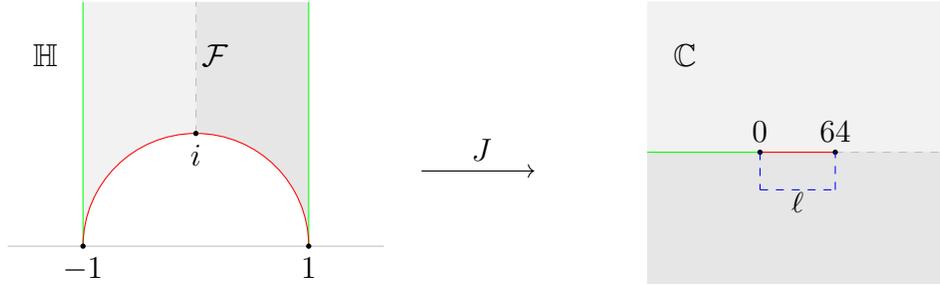

Since $J(\tau)$ has a pole at the cusp at $\infty$,
if $f\in M_k(\Gamma_{\theta},\eps)$, then
$f(\tau)=\theta^{2k}(\tau)J_{\eps}(\tau)p(1/J(\tau))$, where $p(x)\in\CC[x]$
is now a polynomial (without constant term if $\eps=+$). From
    \begin{equation} \label{eq:asymptexp}
    \begin{split}
    \tfrac{1}{2}(\tfrac{\tau}{i})^{-1/2}\theta(1-\tfrac{1}{\tau})
    &= q^{1/8}+q^{9/8}+q^{25/8}+\dots,\\
    -2^{-12}J(1-\tfrac{1}{\tau})   &= q+24q^2+300q^3+\dots,\\
    8\,J_-(1-\tfrac{1}{\tau}) &=
    q^{-1/2}+20q^{1/2}-62q^{3/2}+\dots
   	\end{split}
    \end{equation}
we see that
$f$ is in $M_k(\Gamma_{\theta},+)$ if and only if $\deg(p)\le \nu_{+}(k)$
and
$f$ is in $M_k(\Gamma_{\theta},-)$ if and only if $\deg(p)\le \nu_{-}(k)-1$.
Thus we get the following (see~\cite[Thm.~6.3]{BK}).
\begin{proposition}
    \label{prop:mfdim}
    We have $\dim M_k(\Gamma_{\theta},\eps)=\nu_{\eps}(k)$,
    where $\nu_{\pm}$ are defined in~\eqref{eq:nupmdef}.
    Moreover, any $f\in M_k(\Gamma_{\theta},\eps)$ with
    a Fourier expansion of the form
    $f(\tau)=\sum_{n\ge \nu_{\pm}}c_ne^{\pi i n\tau}$
    must vanish identically.
\end{proposition}
Note that this immediately implies uniqueness in Theorem~\ref{thm:modint},
since any two 2-periodic solutions of the functional
equation~\eqref{eq:fksfeq_general} differ by an element
of $M_k(\Gamma_{\theta},\eps)$, which must vanish by
Proposition~\ref{prop:mfdim}.

Finally, let us record some simple asymptotic relations between various
functions in the fundamental domain $\bF$. For $z\to i\infty$, we have
$\im (1-1/z)=\im (z)|z|^{-2}\asymp \im (z)^{-1}$
and $J(1-1/z)\sim -4096e^{2\pi i z}$,
so that $\log|J(z)|\asymp -\im (z)^{-1}$,
as $z$ tends to $\pm 1$ in the fundamental domain.
From this we deduce that, when expressed in terms of $w=J(z)$, as $w\to 0$
(which again corresponds to $z\to\pm1$ inside the fundamental domain),
we have $\im (z) \asymp \frac{1}{\log|w^{-1}|}$, and
therefore
    \[ |\theta(z)|^2 \asymp |w|^{1/4}\log|w^{-1}|.\]
Moreover, since $J_{-}(z)^2 = 1-64/J(z)$, we get that
$J_{-}(z) = \pm\sqrt{1-64/w}$.

We also record here the following identity
    \begin{equation} \label{eq:jderiv}
    J'(z) = {-\pi i}\,\theta^4(z)J(z)J_{-}(z) .
    \end{equation}
In particular, this implies that if we set $w=J(z)$, then
	\begin{equation} \label{eq:theta4dz}
	\theta^4(z)dz = \pi^{-1}w^{-1/2}(64-w)^{-1/2}dw.
	\end{equation}

\subsection{Modular kernels}
\label{sec:modkernels}
We define the following two-variable meromorphic functions
on the upper half-plane:
	\begin{align} \label{eq:modkernels}
	\begin{split}
	\mathcal{K}_k^{+}(\tau,z)
	&\coloneqq \theta^{2k}(\tau)\theta^{4-2k}(z)\frac{J^{\nu_+}(z)}{J^{\nu_+}(\tau)}\frac{J(\tau)J_{-}(z)}
	{J(\tau)-J(z)} =
	\sum_{n=\nu_{+}}^{\infty} g_{n,k}^{+}(z)q^{n/2}, \\
	\mathcal{K}_k^{-}(\tau,z)
	&\coloneqq \theta^{2k}(\tau)\theta^{4-2k}(z)\frac{J^{\nu_-}(z)}{J^{\nu_-}(\tau)}\frac{J(\tau)J_{-}(\tau)}
	{J(\tau)-J(z)} =
	\sum_{n=\nu_{-}}^{\infty} g_{n,k}^{-}(z)q^{n/2}.
	\end{split}
	\end{align}
Here we view the series on the right as formal power series in the variable
$q^{1/2}$. Note that by construction $\mathcal{K}_{k}^{\pm}(\tau,z)$ is a
meromorphic modular form of weight $k$ in $\tau$ (respectively of weight $2-k$ in $z$)
for each fixed $z$ (respectively~$\tau$). For fixed $\tau$ it has simple poles
for $z\in \Gamma_{\theta}\tau$ and no other singularities in $\HH$.
Moreover,~\eqref{eq:jderiv} implies that the residue
of $\mathcal{K}_{k}^{\pm}(\tau,z)$ at $z=\tau$ is $(\pi i)^{-1}$. From
the above estimates for $\theta(z)$ and $J(z)$ near the cusp at $z=\pm1$,
we get that for any fixed $\tau$ the function $\mathcal{K}_{k}^{\pm}(\tau,z)$
is rapidly decreasing as $z\to\pm 1$ non-tangentially.

Let us also record some important properties of the coefficients
$g_{n,k}^{\pm}(z)$.
\begin{proposition}\label{prop:gnkproperties}
	The functions $g_{n,k}^{\pm}\colon \HH\to\CC$, $n\ge \nu_{\pm}$ belong to
	$M_{2-k}^!(\Gamma_{\theta},\mp)$, vanish at the cusp at~$\pm1$,
	and satisfy
	\begin{equation}
	\label{eq:qexpproperty}
	g_{n,k}^{\pm}(\tau) = q^{-n/2} + O(q^{-(\nu_{\pm}-1)/2})
	,\qquad \tau\to i \infty.
	\end{equation}
\end{proposition}
\begin{proof}
	The first two claims follow trivially from the definition, while
	the last statement is proved in exactly the same way as Theorem~3
	in~\cite{RV} (see also earlier papers by Asai, Kaneko, and Ninomiya~\cite[Sec.~3]{AKN} and by Zagier~\cite{Z2}
	where analogues of $g_{n,k}^{\pm}$ for the full modular group appear).
\end{proof}

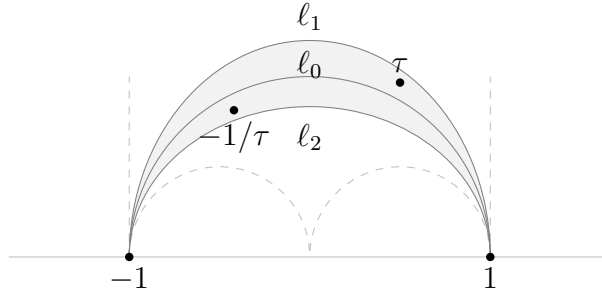
\begin{figure}[h]
	\centering
	\begin{tikzpicture}
	\definecolor{cv0}{rgb}{0.95,0.95,0.95}
	\definecolor{cv1}{rgb}{0.90,0.90,0.90}
	\clip(-9,-0.5) rectangle (7,3.5);
	\begin{scope}[scale=0.8,xshift=-1cm]
	\draw[lightgray] (-5,0) -- (5,0);
	\draw[lightgray,dashed] (-3,0) arc  (180:0:1.5);
	\draw[lightgray,dashed]  (3,0) arc  (0:180:1.5);
	\draw[lightgray,dashed]  (-3,0) -- (-3,3);
	\draw[lightgray,dashed]  (3,0) -- (3,3);
	\fill[color=cv0] (-3,0)  arc (180:0:3 and 3.6) arc (0:180:3 and 2.5);
	\draw[gray] (-3,0) arc (180:0:3 and 3.6);
	\draw[gray] (-3,0) arc (180:0:3 and 2.5);
	\draw[gray] (-3,0) arc  (180:0:3);
	\draw (0,3.6) node[above]{$\ell_1$};
	\draw (0,2.8) node[above]{$\ell_0$};
	\draw (0,2.4) node[below]{$\ell_2$};
	\fill[black] (1.5,2.9) circle (0.07) node[above] {$\tau$};
	\fill[black] (-1.26,2.44) circle (0.07) node[below] {$-1/\tau$};

	\fill[black] (-3,0) circle (0.07) node[below] {$-1$};
	\fill[black] (3,0) circle (0.07) node[below] {$1$};


	\end{scope}
	\end{tikzpicture}
	\label{fig:contour}
	\caption{Deforming the contour of integration}
\end{figure}

\begin{proposition} \label{prop:modkereq}
    If $\phi\colon\HH\to\CC$ is a holomorphic function of moderate
    growth, then
    \begin{equation} \label{eq:contourintegral}
    F_k^{\pm}(\tau,\phi) \coloneqq \frac{1}{2}\int_{-1}^{1}\mathcal{K}_k^{\pm}(\tau,z)\phi(z)dz
    ,\qquad    \tau\in\bF,
    \end{equation}
    where the integral is taken over a semicircle in the upper half-plane,
    admits an analytic continuation to~$\HH$ that satisfies the conditions
    of Theorem~\ref{thm:modint}.
\end{proposition}
\begin{proof}
We only sketch the proof, since it essentially repeats the proof of
Proposition~2 in~\cite{RV}. \change{The main idea is to show that the contour integral
in~\eqref{eq:contourintegral} extends analytically from~$\bF$ to the neighboring
fundamental domain $S\bF$, that the extension satisfies~\eqref{eq:fksfeq_general} in $\bF\cup
S\bF$, and from there to extend it iteratively to all of~$\HH$ using the functional
equation.}

Note that the integral is well-defined since $\mathcal{K}_{k}^{\pm}(\tau,z)$
has exponential decay in $\im(z)^{-1}$ as $z\to\pm 1$ non-tangentially
and $\phi$ is bounded there by some power of $\im(z)^{-1}$.
Let us denote the right-hand side of~\eqref{eq:contourintegral}
by \change{$G_0(\tau)$}, $\tau\in\bF$.
Since the only singularities of the kernel $z\mapsto
\mathcal{K}_k^{\pm}(\tau,z)$ are at $z\in\Gamma_{\theta}\tau$,
\change{$G_0(\tau)$} extends analytically across the vertical lines
$\tau\in\HH$, $\re \tau=\pm 1$ and the resulting analytic extension
is 2-periodic. Let us show that it also extends across the semicircle
and the extension satisfies the functional equation~\eqref{eq:fksfeq_general}.

Let $\ell_0$ denote the semicircle, and consider two other paths $\ell_1$
and $\ell_2$ as in Figure~\ref{fig:contour} such that $\ell_2$ is
the image of $\ell_1$ under the inversion $z\mapsto -1/z$,
with all three paths oriented from~$-1$ to~$1$. Let us define
\change{$G_{1}(\tau)\coloneqq
\frac{1}{2}\int_{\ell_1}\mathcal{K}_k^{\pm}(\tau,z)\phi(z)dz$}.
Note that \change{$G_1$} defines an analytic function in the
region~$\mathcal{U}$ (the shaded region in Figure~\ref{fig:contour})
between $\ell_1$ and~$\ell_2$.
Let~$\tau$ be a point in the region between $\ell_0$ and $\ell_1$.
Then the residue theorem tells us that \change{
	\[G_0(\tau) - G_1(\tau) = \phi(\tau),\]
}so that \change{$G_1(\tau)+\phi(\tau)$} provides an analytic extension of \change{$G_0$}
to $\mathcal{U}$. Moreover, we automatically get~\eqref{eq:fksfeq_general}
since \change{$G_1(\tau)=\pm (\tau/i)^{-k}G_1(-1/\tau)$} for $\tau\in\mathcal{U}$
because of the corresponding property of $\mathcal{K}_{k}^{\pm}(\tau,z)$.

To obtain an analytic extension to all of $\HH$ we simply define $F(\tau)$
for $\tau\in \gamma^{-1}\ol{\bF}$ as \change{$G_0|_{k}^{\pm}\gamma+\phi_{\gamma}$},
where $\{\phi_\gamma\}_{\gamma\in\Gamma_{\theta}}$ is the $\Gamma_{\theta}$-cocycle generated by $\phi_{T^2}=0$ and
$\phi_S=\phi-\phi|_{k}^{\pm}S$ (see Section~\ref{sec:cocycleestimates}).
Since the neighboring regions of $\gamma^{-1}\ol{\bF}$
are $\gamma^{-1}S\ol{\bF}$ and $\gamma^{-1}T^{\pm2}\ol{\bF}$,
the above continuation properties of $F_0$ imply that $F$ is well-defined and
analytic on~$\HH$.

Finally, since $\phi$ is of moderate growth and $F(\tau)$ is an
automorphic integral for the cocycle generated by~$\phi$,
by the main result of~\cite{K4} we get that
$F_{k}^{\pm}(\tau,\phi)\coloneqq F(\tau)$ is also of moderate growth. \end{proof}

We will prove more precise statements about the
growth of~$F_{k}^{\pm}(\tau,\phi)$ in Section~\ref{sec:estimates}.

\begin{remark} \label{rem:contourintegral}
Let us give a brief explanation of why one would expect a formula
like~\eqref{eq:contourintegral}. Assume that $k=0$ and the sign is ``+''.
Then the function that we are looking for, when written in
terms of $w=J(\tau)$, is a holomorphic function on $\overline{\CC}\sm[0,64]$
with prescribed jumps along the segment $[0,64]$. By the
classical Sokhotski-Plemelj formula, such a function
is given by an integral $\int_{0}^{64}\frac{A(s)}{w-s}ds$,
where $A(s)$ is the jump at the point $s$. When
expressed in terms of the upper half-plane variables~$\tau$ and~$z$,
the Cauchy kernel simply becomes $\mathcal{K}_0^{+}(\tau,z)$ and
we obtain~\eqref{eq:contourintegral}.
To get the formula in the general case we simply divide both sides of the
functional equation~\eqref{eq:fksfeq_general}
by $\theta^{2k}(\tau)J_{\eps}(\tau)J^{n}(\tau)$ for an appropriate
value $n\in\ZZ$ to reduce to the case $k=0$, $\eps=+$.
Let us also mention that, in the case of $\PSL_2(\ZZ)$, such integrals
have previously appeared in a work of Duke, Imamo\={g}lu,
and T\'{o}th~\cite{DIT}.
\end{remark}

\begin{remark}
	The functions $\mathcal{K}_{k}^{\pm}(\tau,z)$ are sometimes called
	Green's functions, see, e.g., Eichler's paper~\cite[p.~121]{E}.
	For $k>2$ one can instead use the Poincar\'e series
	\[
	P_k^{\pm}(\tau,z)
	\coloneqq \frac{1}{2}\sum_{\gamma\in\Gamma_{\theta,\infty}
		\backslash\Gamma_{\theta}}
	\chi_{\pm}(\gamma)j_{\theta}^{-2k}(\gamma,\tau)
	\frac{e^{\pi i \gamma\tau}+e^{\pi i z}}{e^{\pi i\gamma\tau}-e^{\pi i z}},
	\]
	\change{(where $\Gamma_{\theta,\infty}$ denotes the subgroup of $\Gamma_{\theta}$
	generated by $T^2$)} which differs from $\mathcal{K}_{k}^{\pm}(\tau,z)$ by an element
	of $M_k^!(\Gamma_{\theta},\pm)\otimes M_{2-k}^!(\Gamma_{\theta},\mp)$.
\end{remark}

\subsection{Definition and basic properties of $F_{k}^{\pm}(\tau,s)$}
Using the result of Proposition~\ref{prop:modkereq} we can now precisely
define the special functions~$F_{k}^{\pm}$.

For $k\ge 0$ we define $F_{k}^{\pm}\colon \HH\times\CC\to\CC$ by
    \begin{equation} \label{eq:fksintegral}
    F_k^{\pm}(\tau,s) \coloneqq \frac{1}{2}\int_{-1}^{1}\mathcal{K}_k^{\pm}(\tau,z)(z/i)^{-s}dz,
    \qquad \tau\in\bF
    \end{equation}
and by analytic continuation in $\tau$ if $\tau$ is in $\HH\sm\bF$.
The function $F_k^{\pm}(\cdot,s)$ is 2-periodic and has a Fourier expansion
    \begin{equation} \label{eq:fksfourier}
    F_k^{\pm}(\tau,s) = \sum_{n=\nu_{\pm}}^{\infty}\alpha_{n,k}^{\pm}(s) e^{\pi i n\tau} ,
    \end{equation}
where $\alpha_{n,k}^{\pm}(s)$ are given by
    \begin{equation} \label{eq:alphankdef}
    \alpha_{n,k}^{\pm}(s) := \frac{1}{2}\int_{-1}^{1}g_{n,k}^{\pm}(z)(z/i)^{-s}dz,
    \end{equation}
and $g_{n,k}^{\pm}$ are weakly holomorphic modular forms of weight $2-k$ defined by~\eqref{eq:modkernels}. The coefficients $\alpha_{n,k}^{\pm}(s)$
are of polynomial growth in~$n$ for any fixed $s\in\CC$ and
    \begin{equation} \label{eq:fksfeq}
    F_k^{\pm}(\tau,s) \mp (\tau/i)^{-k}F_k^{\pm}(-1/\tau,s)
    = (\tau/i)^{-s} \mp (\tau/i)^{s-k}, \qquad \tau\in\HH.
    \end{equation}
Finally, for any fixed $\tau\in\HH$ the function $F_{k}^{\pm}(\tau,s)$
is an entire function of~$s$ and it satisfies
    \begin{equation} \label{eq:fkssym}
    F_k^{\pm}(\tau,k-s) = \mp F_k^{\pm}(\tau,s).
    \end{equation}
The last claim follows from the uniqueness part of Theorem~\ref{thm:modint}
since $F_{k}^{\pm}(\tau,k-s)$ is 2-periodic in~$\tau$ and
satisfies the same functional equation as $F_{k}^{\pm}(\tau,s)$, up to sign.

Finally, we remark that $\alpha_{n,k}^{\pm}(s)$ is an entire
function of exponential type. More precisely, by
making a change of variable \change{$z=ie^{2\pi it}$} in~\eqref{eq:alphankdef} we obtain
	\begin{equation} \label{eq:alphapw}
	\alpha_{n,k}^{\pm}(s) = -\pi\int_{-1/4}^{1/4}g_{n,k}^{\pm}(ie^{2\pi i t})e^{-2\pi i t(s-1)}dt,
	\end{equation}
so that $\alpha_{n,k}^{\pm}(s)$ is the Fourier transform
of a $C^{\infty}$-smooth function with support in $[-1/4,1/4]$.
An analogous calculation also shows that for $\tau\in\bF$ the
function $s\mapsto F_{k}^{\pm}(\tau,s)$ is also the Fourier transform of a
smooth function with support in $[-1/4,1/4]$. Similarly,
we get the following result.
\begin{proposition} \label{prop:rapiddecay}
	Let $k\ge0$. Then there exists $c>0$ such that for all $x>1$ we have
	\[|F_k^{\pm}(ix,s)-\alpha_{0,k}^{\pm}(s)|
	\ll_k e^{\frac{\pi}{2}|\im s|}e^{-\pi x-c\sqrt{|\im s|}} .\]
\end{proposition}
\begin{proof}
	Making the change of variable $z=e^{it}$ in the definition we get
	\[F_k^{\pm}(ix, s)-\alpha_{0,k}^{\pm}(s)
	= \frac{1}{2}\int_{-\pi/2}^{0}
	(\mathcal{K}_k^{\pm}(ix,ie^{it})-g_{0,k}^{\pm}(ie^{it}))
	(e^{-ist}\mp e^{i(s-k)t})dt.\]
	If $x\ge 2$, then using the leading terms of the asymptotic
	expansions~\eqref{eq:asymptexp} and the fact that $J(ix)>100$ for $x\ge 2$,
	we get
	\[|\mathcal{K}_k^{\pm}(ix,ie^{it})-g_{0,k}^{\pm}(ie^{it})|
	\ll_k \exp(-\pi x-\tfrac{\kappa_{\pm}}{\cos t}),\qquad t\in (-\pi/2,\pi/2),\]
	where $\kappa_{+}:=2\pi(1-\{k/4\})$ and $\kappa_{-}:=2\pi(1-\{(k-2)/4\})$.
	Thus we have
	\begin{align*}
	|F_k^{\pm}(ix, s)-\alpha_{0,k}^{\pm}(s)|
	&\ll_k e^{-\pi x}
	\int_{0}^{\pi/2}e^{|\im s|t-\frac{\kappa_{\pm}}{\cos t}} dt\\
	&=e^{\frac{\pi}{2}|\im s|}e^{-\pi x}
	\int_{0}^{\pi/2}e^{-|\im s|t-\frac{\kappa_{\pm}}{\sin t}}dt
	\le \change{\tfrac{\pi}{2}}e^{\frac{\pi}{2}|\im s|}e^{-\pi x}e^{-\change{2}\sqrt{\kappa_{\pm}|\im s|}},
	\end{align*}
	where we have used the inequalities $\frac{1}{\sin t} \ge \frac{1}{t}$ and
	$at+bt^{-1}\ge 2\sqrt{ab}$. Since $\kappa_{\pm}>0$, this proves
	the claim.

	If $1<x<2$, then we split the integral as $\int_{-\pi/2}^{-\pi/4}+\int_{-\pi/4}^{0}$. Then we use the same
	estimate for the first integral, while the second integral is of
	size $\ll_k e^{\frac{\pi}{4}|\im s|}$, which can be seen by deforming the
	contour of integration (similarly to what was done in the proof of
	Proposition~\ref{prop:modkereq}) so as to avoid large values of the
	denominator $J(\tau)-J(z)$ in $\mathcal{K}_{k}^{\pm}(\tau,z)$.
\end{proof}

\subsubsection{Relation to the interpolation bases for the $\sqrt{n}$ case}
Next, let us relate~$\alpha_{n,k}^{\pm}(s)$ to the functions
$b_{n}^{\pm}$ and $d_{n}^{\pm}$ constructed in~\cite{RV}. If we define
    \begin{equation} \label{eq:sqrtndef}
    \begin{split}
    b_{n}^{\pm}(x) &\,\change{:=}\, \frac{1}{2}\int_{-1}^{1}g_{n,1/2}^{\pm}(z)e^{\pi i z x^2}dz,\\
    d_{n}^{\pm}(x) &\,\change{:=}\, \frac{1}{2}\int_{-1}^{1}g_{n,3/2}^{\pm}(z)xe^{\pi i z x^2}dz,
    \end{split}
    \end{equation}
(the sign notation differs from that of \cite{RV} so that $b_n^{\pm}$
and $d_n^{\pm}$ in our context coincide with respectively $b_n^{\mp}$ and $d_n^{\mp}$ from~\cite{RV})
then a routine calculation shows that
    \begin{equation} \label{eq:reltosqrtn}
    \begin{split}
    \Gamma_{\RR}(s)\alpha_{n,1/2}^{\pm}(s/2) = 2\int_{0}^{\infty}b_{n}^{\pm}(x)x^{s-1}dx,\\
    \Gamma_{\RR}(s)\alpha_{n,3/2}^{\pm}(s/2) = 2\int_{0}^{\infty}d_{n}^{\pm}(x)x^{s-2}dx,
    \end{split}
	\end{equation}
where we again use the notation $\Gamma_{\RR}(s):=\pi^{-s/2}\Gamma(s/2)$.
We remark here that in~\cite[Prop.~1, Prop.~3]{RV} it is proved
that~$b_{n}^{\pm}(x)$ is an even Fourier eigenfunction with eigenvalue $\mp1$,
$d_{n}^{\pm}(x)$ is an odd Fourier eigenfunction with eigenvalue $\pm i$,
and moreover that
	\begin{equation} \label{eq:sqrtninterpolation}
	b_{n}^{\pm}(\sqrt{m})=d_{n}^{\pm}(\sqrt{m})=\delta_{n,m}
	, \qquad m\ge 1.
	\end{equation}
All these properties can be easily checked directly from the definition,
using~\eqref{eq:qexpproperty}.

\subsubsection{Special values}
\label{sec:fkspval}
We conclude this section by giving explicit evaluations
of $F_{k}^{\pm}(\tau,s)$ for some special values of~$s$. We do this
using the fact that~\eqref{eq:fksfourier} and~\eqref{eq:fksfeq} uniquely
determine $F_{k}^{\pm}(\tau,s)$ as a function of $\tau$, so that
if we can find a 2-periodic function $f(\tau)$ that
satisfies~~\eqref{eq:fksfeq}, then necessarily
$F_{k}^{\pm}(\tau,s)-f(\tau)$ belongs to $M_k(\Gamma_{\theta},\pm)$.

A trivial example is $s=0$, where we can take $f(\tau)=1$.
Thus $F_{k}^{\pm}(\tau,0)=1-g(\tau)$, where $g(\tau)=0$ if $\nu_{\pm}=0$
and otherwise $g(\tau)$ is the unique modular form
in $M_{k}(\Gamma_{\theta},\pm)$ with the $q$-expansion
$g(\tau) = 1+O(q^{\nu_{\pm}/2})$. In particular,
	\begin{equation}\label{eq:spvalue0}
	F_{k}^{-}(\tau,0)=1,\quad\qquad
	F_{k}^{+}(\tau,0)=1-\theta^{2k}(\tau),\quad\qquad
	0\le k<2.
	\end{equation}

Similarly, from~\eqref{eq:fksfeq} we see that
$F_{k}^{+}(\tau,k/2)$ is in $M_{k}(\Gamma_{\theta},+)$ and looking at the
$q$-expansion, we see that in fact
	\[F_{k}^{+}(\tau,k/2)=0,\quad\qquad 0\le k<2.\]

A more interesting example is the identity
	\[F_{2}^{-}(\tau,1) =
	\frac{\pi}{3}(-E_2(\tau/2)+5E_2(\tau)-4E_2(2\tau)),\]
where $E_2$ is the weight~2 Eisenstein series,
$E_2(\tau)=1-24\sum_{n\ge1}\sigma(n)q^n$ \change{(here $\sigma(n)=\sum_{d|n}d$ is the
divisor sum function)}. To see this we use that by the well-known functional equation
	\[E_2(\tau)-\tau^{-2}E_2(-1/\tau) = \frac{6}{\pi}(\tau/i)^{-1}\]
we have $F_{2}^{-}(\tau,1)-\frac{\pi}{3}E_2(\tau)\in M_2(\Gamma_{\theta},-)$
and note that the space $M_2(\Gamma_{\theta},-)$ is one-dimensional,
spanned by $E_2(\tau/2)-4E_2(\tau)+4E_2(2\tau)$. As a corollary, we have
	\[\alpha_{n,2}^{-}(1) = 8\pi(\sigma(n)-5\sigma(n/2)+4\sigma(n/4)),\qquad n\ge 1,\]
where we define $\sigma(x)=0$ if $x\not\in\NN$.

\section{The Dirichlet series kernel associated with zeros of $\zeta(s)$}
\label{sec:zetakernels}
In this section we assume that the weight $k$ is a positive
real number and consider the function $F_{k}^{\pm}(\tau,s)$ given by
the Fourier expansion
	\[F_{k}^{\pm}(\tau,s) := \sum_{n\ge\nu_{\pm}}
	\alpha_{n,k}^{\pm}(s)e^{\pi i n\tau},\]
where, as before, $\nu_{-}=\lfloor(k+2)/4\rfloor$
and $\nu_{+}=\lfloor (k+4)/4\rfloor$.
For convenience we extend the definition of $\alpha_{n,k}^{\pm}(s)$
to all $n\ge0$ by setting $\alpha_{n,k}^{\pm}(s)\,\change{:=}\,0$, $0\le n < \nu_{\pm}$.

\subsection{The Mellin transform of $F_{k}^{\pm}(\tau,s)$}
\label{sec:akernel}
Let us define $\Aa_k^{\pm}(w,s)$ by
	\begin{equation} \label{eq:akdef}
	\Aa_k^{\pm}(w,s) \coloneqq
	\int_{0}^{\infty}(F_{k}^{\pm}(it,s)-\alpha_{0,k}^{\pm}(s))t^{w-1}dt
	= \pi^{-w}\Gamma(w)\sum_{n\ge1}\frac{\alpha_{n,k}^{\pm}(s)}{n^w}.
	\end{equation}
Since for fixed $s$ the sequence $\{\alpha_{n,k}^{\pm}(s)\}$ grows
polynomially, the above Dirichlet series converges absolutely for
sufficiently large~$\re w$. Similarly, for fixed $s$ we have
(by~\eqref{eq:fksfeq})
	\begin{align*}\label{eq:fkasymp}
	F_k^{\pm}(it,s) &= \alpha_{0,k}^{\pm}(s) + O(e^{-\pi t}),\qquad &t\to\infty ,\\
	F_k^{\pm}(it,s) &=  \pm \alpha_{0,k}^{\pm}(s)t^{-k}
	+t^{-s}\mp t^{s-k}+ O(t^{-k}e^{-\pi/t}),\qquad
	&t\to0+ ,
	\end{align*}
and hence the integral in~\eqref{eq:akdef} converges absolutely
for $\re w > \max(k,\re s,\re(k-s))$.

\begin{proposition}
	The function $w\mapsto \Aa_k^{\pm}(w,s)$ extends to a meromorphic
	function in $\CC$ with simple poles at $w=s$, $k-s$ with \change{respective}
	residues~$1$, $\mp 1$, and at most simple poles at $w=0$, $k$ with \change{respective}
	residues $-\alpha_{0,k}^{\pm}(s)$, $\pm\alpha_{0,k}^{\pm}(s)$. Moreover,
	the function $\Aa_k^{\pm}$ satisfies the two functional equations
	\begin{equation} \label{eq:dirkerfeq}
	\begin{split}
	\Aa_k^{\pm}(k-w,s) = \pm\Aa_k^{\pm}(w,s), \\
	\Aa_k^{\pm}(w,k-s) = \mp\Aa_k^{\pm}(w,s).
	\end{split}
	\end{equation}
	Finally, the function $w\mapsto \Aa_k^{\pm}(w,s)$ is bounded
	in lacunary vertical strips $\{u+iv\mid a\le u\le b, |v|\ge T\}$ for
	sufficiently large $T>0$.
\end{proposition}
\begin{proof}
	The claims follow from the general result of Bochner~\cite[Th.~4]{B}
	combined with~\eqref{eq:fksfeq}.

	More specifically, using the standard trick (see, e.g.,~\cite{Z})
	by splitting the integral defining $\Aa_{k}^{\pm}(w,s)$
	as $\int_{0}^{1}+\int_{1}^{\infty}$ and applying~\eqref{eq:fkssym}
	to the part $\int_{0}^{1}$ we obtain
	\begin{equation} \label{eq:akintsym}
	\begin{split}
	\Aa_k^{\pm}(w,s) =
	-\alpha_{0,k}^{\pm}(s)(w^{-1}\pm (k-w)^{-1})+(w-s)^{-1}\pm (k-w-s)^{-1}\\
	+\int_{1}^{\infty}(F_k^{\pm}(it,s)-\alpha_{0,k}^{\pm}(s))(t^{w-1}\pm t^{k-w-1})dt.
	\end{split}
	\end{equation}
	Since the integral defines an analytic function of $(w,s)$,
	this immediately implies meromorphic continuation with
	given simple poles, and since the integral is clearly bounded
	in vertical strips, we also obtain boundedness in lacunary strips
	for~$w\mapsto\Aa_k^{\pm}(w,s)$.
	Finally, the functional equations~\eqref{eq:dirkerfeq} follow trivially
	from~\eqref{eq:akintsym} and~\eqref{eq:fkssym}.
\end{proof}

Note that $\alpha_{0,k}^{+}(s)=0$ for $k\ge 0$ and $\alpha_{0,k}^{-}(s)=0$ for
$k\ge 2$ hence in these cases (which correspond to $\nu_{\pm}>0$)
the only singularities of $w\mapsto \Aa_k^{\pm}(w,s)$ are the
simple poles at $s$ and $k-s$.

\begin{remark}
Bochner's Converse Theorem~\cite[Th.~7.1]{BK},~\cite[Th.~4]{B}
implies that the function $w\mapsto \Aa_k^{\pm}(w,s)$ is
essentially uniquely defined by the first equation in~\eqref{eq:dirkerfeq}
and its poles. Let us make this
precise in the case when $\nu_{\pm}>0$:
assume that $\psi(w)=\sum_{n\ge\nu_{\pm}}a_nn^{-w}$
is convergent in some right half-plane and extends to a meromorphic function
in~$\CC$ such that $(w-s)(w-k+s)\psi(w)$ is entire of finite order
and $\Psi(w)=\pi^{-w}\Gamma(w)\psi(w)$ satisfies
$\Psi(k-w)=\pm\Psi(w)$. Then $\Psi(w)$ is a multiple of $\Aa_k^{\pm}(w,s)$.
Thus we see that these functions are in some sense universal:
if $\psi(w)$ is any Dirichlet series such that  $\psi(w)P(w)$ is entire of
finite order for some polynomial~$P$ and such that $\Psi(k-w)=\pm \Psi(w)$, then
	\[\Psi(w) = L_f(w) +
	\sum_{j}c_j\frac{\partial^{m_j}}{\partial s^{m_j}}\Aa_k^{\pm}(w,s_j)\]
for some $c_j,s_j\in\CC$, where $f\in M_k(\Gamma_{\theta},\pm)$
and $L_f(w)=\int_{0}^{\infty}(f(it)-a_0(f))t^{w-1}dt$. This gives a
considerable strengthening of the abundance principle of Knopp~\cite{K2}.
\end{remark}

Using the estimates from Section~\ref{sec:estimates} we get
quite precise information about the behavior
of $w\mapsto \Aa_k^{\pm}(w,s)$ in vertical strips.
\begin{lemma} \label{lem:Aest}
	Suppose that $k/2 \le \re s < k<2$ and let $\kappa=\max(k,1)$ .
	Then
	\begin{equation*}
	|\Aa_k^{\pm}(u+iv,s)| \le
	C(s) \pi^{-u}|\Gamma(u+iv)|
	(1+|v|)^{\kappa+\eps-u}, \qquad k-\kappa-\eps \le u\le \kappa+\eps, \\
	\end{equation*}
	for every $\eps>0$ and all sufficiently big $|v|$,
	where $C(s)>0$ depends only on~$s$.
\end{lemma}
A weaker version of Lemma~\ref{lem:Aest}, with a cruder bound on the growth in the vertical direction, may be obtained directly from Proposition~\ref{prop:modkereq}. This would in turn suffice to establish a weaker version of Theorem~\ref{thm:zeta}, as alluded to in the introduction.
\begin{proof}[Proof of Lemma~\ref{lem:Aest}]
	By the Cauchy--Schwarz inequality, we have
	\[
	\left(\sum_{2^l \le n < 2^{l+1}} |\alpha^{\pm}_{n,k}(s)| n^{-u}\right)^2
	\le 2^{l(1-2u)} \sum_{2^l \le n < 2^{l+1}} |\alpha^{\pm}_{n,k}(s)|^2  .\]
	Therefore, by applying Proposition~\ref{prop:squaresum}
	from Section~\ref{sec:estimates} we get
	\[ \sum_{n\le x} |\alpha^{\pm}_{n,k}(s)|^2 \ll_s x^{k+|k-1|}\log^2\!x,\]
	and thus the Dirichlet series representing
	$w\mapsto A_{k}^{\pm}(w,s)$ converges
	absolutely for $\re w > \kappa$.
	Let us set
	\[D(w):=\frac{\pi^w}{\Gamma(w)}\Aa_k^{\pm}(w,s). \]
	Note that $D(w)$ can have poles only at $w=k,s,k-s$, since the potential
	pole at $w=0$ is canceled by the pole at $w=0$ of $\Gamma(w)$.
	Since the Dirichlet series converges absolutely for $\re w>\kappa$,
	for arbitrary fixed $\eps>0$ we have
	\[ D(\kappa+\eps + i v)\ll 1 \quad \text{and} \quad D(k-\kappa-\eps + i v)\ll (1+|v|)^{2\kappa-k+2\eps} .\]
	Here the second inequality follows because of the absolute convergence
	of the Dirichlet series, and the second follows by the functional equation
	for $\Aa_{k}^{\pm}(k-w,s)$. Now
	\[ F(w) \coloneqq D(w) (w-1)(w-s)(w-(k-s)) \]
	is an entire function satisfying
	\[ F(\kappa+\eps + i v)\ll (1+|v|)^3 \quad \text{and} \quad F(k-\kappa-\eps + i v)\ll (1+|v|)^{2\kappa-k+3+2\eps} .\]
	The deduction of the functional equation for $D(w,s)$ implies
	the crude bound
	\[ |F(u+iv)| \ll |v|^3|\Gamma(u+iv)|^{-1} \]
	in the strip $k-\kappa-\eps \le u  \le \kappa+\eps$.
	We may therefore use the Phragm\'{e}n--Lindel\"{o}f principle
	in a familiar way (see for example \cite[Sect. 5.65]{Tf}) to conclude
	that $F(w) ((\kappa+2\eps-w)i)^{(w-\kappa-\eps)-3}$ is a bounded analytic
	function in that strip.
\end{proof}

To conclude the discussion of~$\Aa_k^{\pm}(w,s)$,
note that, as a corollary of~\eqref{eq:spvalue0}, we obtain
	\begin{equation} \label{eq:alplus}
	\Aa_{l/2}^{+}(w,0) = -\pi^{-w}\Gamma(w)\sum_{n\ge1}\frac{r_l(n)}{n^w},
	\qquad l=1,2,3,
	\end{equation}
where $r_l(n)$ is the number of representations of~$n$ as a sum of squares
of~$l$ integers. In particular,
$\Aa_{1/2}^{+}(w,0)=-2\pi^{-w}\Gamma(w)\zeta(2w)$.
Similarly,~\eqref{eq:spvalue0} implies
	\begin{equation} \label{eq:alminus}
	\Aa_{l/2}^{-}(w,0) = \Aa_{l/2}^{-}(w,l/2) = 0,
	\qquad l=1,2,3.
	\end{equation}
Using the results from Section~\ref{sec:fkspval} one can also
easily obtain explicit expressions for $\mathcal{A}_{l/2}^{\pm}(w,0)$
for other values of~$l$.

\subsection{Construction of the Dirichlet series kernels}
\label{sec:basisconstr}

We will now construct, in accordance with the general setup of
Section~\ref{sec:generalkernel}, the Dirichlet series kernels corresponding
to the interpolation formula of Theorem~\ref{thm:zeta}. This will provide
us with an explicit construction of the functions $U_n(z)$ and $V_{\rho,j}(z)$
and will be used to prove our main result.

We define the Dirichlet series kernels $H_{\pm}(w,s)$ by
	\begin{equation} \label{eq:zetakernel}
	\begin{split}
	&H_{-}(w,s) \coloneqq
	\frac{\zeta^{*}(s)}{2}\frac{\Aa_{1/2}^{-}(w/2,s/2)}{\zeta^{*}(w)}
	= \sum_{n\ge1}\frac{h_n^{-}(s)}{n^{w/2}},
	\qquad s\ne 0,1,\\
	&H_{+}(w,s) \coloneqq
	\frac{\zeta^{*}(s)}{2}\Big(\frac{\Aa_{1/2}^{+}(w/2,s/2)}{\zeta^{*}(w)}-\alpha_{1,1/2}^{+}(s/2)\Big)
	= \sum_{n\ge2}\frac{h_n^{+}(s)}{n^{w/2}},
	\qquad s\ne 0,1,
	\end{split}
	\end{equation}
where $\zeta^{*}(s)=\Gamma_{\RR}(s)\zeta(s)$ is the completed zeta
function. Formulas~\eqref{eq:alplus} and~\eqref{eq:alminus} show
that both functions extend analytically also to $s=0,1$.
Note that the coefficients $h_{n}^{\pm}(s)$ can be computed using
M\"obius inversion as
	\begin{equation} \label{eq:hnmoebius}
	h_n^{\pm}(s) = \frac{\zeta^{*}(s)}{2}\sum_{d^2|n} \mu(d)\alpha_{n/d^2,1/2}^{\pm}(s/2).
	\end{equation}
From this and~\eqref{eq:alplus},~\eqref{eq:alminus} we see that $h_n^{\pm}(s)$,
$n\ge 1$ are entire (we also set $h_1^{+}(s)\coloneqq 0$).

Since $\zeta^*(1-s)=\zeta^*(s)$, the functional equations~\eqref{eq:dirkerfeq}
imply
	\begin{equation} \label{eq:dirkerfeq2}
	\begin{split}
	H_{\pm}(1-w,s) = \pm H_{\pm}(w,s), \\
	H_{\pm}(w,1-s) = \mp H_{\pm}(w,s).
	\end{split}
	\end{equation}
Moreover, from~\eqref{eq:akintsym} and the fact that $\zeta^{*}(w)$
has simple poles at $w=0,1$, we see that for a fixed $s$ the
function $w\mapsto H_{\pm}(w,s)$ has simple poles at $w=s$ and $w=1-s$
with residues $1$ and $\mp 1$, and a pole at $w=\rho$ of order at
most $m(\rho)$ for each nontrivial zero $\rho$ of $\zeta(s)$.

From the above discussion we see that, if we define
$E_{\pm}(w,s)=H_{+}(w,s)\pm H_{-}(w,s)$, then $w\mapsto E_{+}(w,s)$
(respectively $w\mapsto E_{-}(w,s)$) is a Dirichlet series with poles at $w=\rho$
for nontrivial zeros of $\zeta$ and at $w=s$ (respectively $w=1-s$). According to
the setup of Section~\ref{sec:generalkernel}, this suggests that $E_{\pm}(w,s)$
(up to an appropriate linear change of variables that maps the critical line
to~$\RR$) are Dirichlet series kernels associated to a Fourier interpolation
formula with $\Lambda=\{(\rho-1/2)/i \colon \zeta^{*}(\rho)=0\}$
and $\Lambda^{*}=\{\pm\log n/(4\pi)\}_{n\ge1}$.
Motivated by this we define $U_n^{\pm}(z)$ as
	\begin{equation} \label{eq:basisun}
	\begin{split}
	U_n^{\pm}(z) \coloneqq h_n^{\mp}(1/2+iz).
	\end{split}
	\end{equation}
Similarly, we define $V_{\rho,j}^{\pm}(z)$ by
	\begin{equation*} \label{eq:basisvrho}
	H_{\mp}(w,1/2+iz)
	= \sum_{j=0}^{m(\rho)-1} i^j\frac{j!\,V_{\rho,j}^{\pm}(z)}{(w-\rho)^{j+1}}+O(1)
	,\quad w\to \rho,
	\end{equation*}
or, equivalently,
	\begin{equation} \label{eq:basisvrho2}
	V_{\rho,j}^{\pm}(z) \coloneqq \frac{1}{2\pi i}\int_{|w-\rho|=\eps}
	\frac{i^{-j}(w-\rho)^j}{j!}H_{\mp}(w,1/2+iz)dw,
	\end{equation}
where $\eps$ is chosen so that $\eps<|\rho-1/2\pm z|$ and $\eps<|\rho-\rho'|$
for all $\rho'\ne \rho$ such that $\zeta^{*}(\rho')=0$.

We now turn to rigorously proving Theorem~\ref{thm:zeta}.

\subsection{Proof of Theorem~\ref{thm:zeta}}
\label{sec:maintheoremproof}
As in the proof of Lemma~\ref{lem:Aest} we will use certain estimates for the
coefficients $\alpha_{n,k}^{\pm}(s)$ that are somewhat technical
and are proved in Section~\ref{sec:estimates}.

We define the auxiliary functions
	\begin{equation*}
	\begin{split}
	D_{-}(w,s) \coloneqq
	\frac{\Gamma_{\RR}(s)}{2}\sum_{n\ge1}\frac{\alpha_{n,1/2}^{-}(s/2)}{n^{w/2}}
	,\qquad
	D_{+}(w,s) \coloneqq
	\frac{\Gamma_{\RR}(s)}{2}\sum_{n\ge1}
	\frac{\alpha_{n,1/2}^{+}(s/2)-\alpha_{1,1/2}^{+}(s/2)}{n^{w/2}},
	\end{split}
	\end{equation*}
and note that
	\[H_{\pm}(w,s) \coloneqq \frac{\zeta(s)}{\zeta(w)} D_{\pm} (w,s).\]

\begin{lemma}\label{lem:Dest}
	Suppose that $1/2 \le \re s <1$. Then
    \begin{align}
    |D_{\pm}(u+iv,s)| & \le C(s) (1+|v|)^{(2+\eps-u)/2}, \quad u\le 2+\eps, \label{eq:PL} \\ \label{eq:Dpart}
    H_{\pm}(1+\eps+iv,s) & = \sum_{n\le x} h_{n}^{\pm}(s)n^{-(1+\eps+iv)/2} + O\big((1+|v|) x^{-\eps/3 } \big) \end{align}
for every $\eps>0$, where $C(s)$ is a positive
constant that depends only on~$s$.
\end{lemma}

\begin{proof}
The first claim follows from Lemma~\ref{lem:Aest}.
We next prove~\eqref{eq:Dpart}. Set $A(x)\coloneqq \sum_{n\le x} h_n^{\pm}(s)$.
Using summation by parts, we get
    \[H_{\pm}(w,s)=\sum_{n\le x}h_n^{\pm}(s) n^{-w/2}+\frac{w}{2} \int_x^{\infty} (A(y)-A(x)) y^{-w/2-1} dy \]
when $\re w>1$, so it suffices to show that
    \[ A(x) \ll_{\eps} x^{1/2+\eps/6} \]
for every $\eps>0$. But this holds because
    \[ A(x)=\frac{\zeta^{*}(s)}{2}\sum_{d\le \sqrt{x}} \mu(d) \sum_{n\le x/d^2} \alpha_{n,1/2}^{\pm}(s/2) \ll \sqrt{x} \sum_{d\le \sqrt{x}} \frac{1}{d} \le \sqrt{x} \log x \]
when $1/2\le \re s <1$ by Proposition~\ref{prop:partsum}.
\end{proof}

We also require an additional lemma which is a result
from a paper by Ramachandra and Sankaranarayanan~\cite[Thm.~2]{RS}.
\begin{lemma} \label{lem:Ind} There exists a positive constant $c$ such that
    \[ \min_{T\le t \le T+T^{1/3}} \max_{1/2\le \sigma \le 2} |\zeta(\sigma+it)|^{-1} \le \exp\big(c(\log\log T)^2\big) \]
    holds for all sufficiently large $T$.
\end{lemma}

We now consider a general function $f$ in $\mathcal{H}_1$ and prove a representation that splits naturally into an even and an odd part, so that the even part yields Theorem~\ref{thm:zeta}. We set
	\[ F(s)\coloneqq f\left(\frac{s-1/2}{i}\right) \]
let
	\[ F_+(s)\coloneqq\frac{F(s)+F(1-s)}{2}
	   \quad \text{and} \quad
	   F_-(s)\coloneqq\frac{F(s)-F(1-s)}{2} \]
be respectively the even and odd part of $F(s)$.
The proof naturally splits into three parts.

\subsubsection*{Proof of \eqref{eq:add}}
Consider the operator
	\begin{equation} \label{eq:defT} (R_{\delta}F_{\delta})(s)  \coloneqq \frac{1}{4\pi i} \int_{c-i\infty}^{c+i\infty} \left[H_{-\delta}(w,s)  - H_{-\delta}(1-w,s) \right] F_{\delta}(w) dw ,
	\end{equation}
where $\delta=\pm$ and both $1-c < \re s < c$ and $c>1$.
The proof follows the usual scheme of computing these integrals in
two different ways. First, using the functional equation for
$H_{\pm}(1-w,s)$, we express the integrand in \eqref{eq:defT}
as $F_{\delta} (w)$ times a Dirichlet series in $w$. We then apply
\eqref{eq:Dpart} and the assumption that $f$ is in $\mathcal{H}_1$
to infer that
	\begin{equation} \label{eq:rhs1}
	\begin{split}
	(R_{\delta}F_{\delta})(s) & = \frac{1}{2\pi i} \int_{c-i\infty}^{c+i\infty} H_{-\delta}(w,s) F_{\delta}(w)   dw \\
	&  =\frac{1}{2\pi i} \sum_{n\le x}
	h_{n}^{-\delta}(s) \int_{c-i\infty}^{c+i\infty} n^{-w/2} F_{\delta}(w) dw+O(x^{-\eps/3}) \\
	& = \sum_{n\le x} (\mathcal{M}^{-1}F_{\delta})(\sqrt{n})
	h_n^{-\delta}(s)+O(x^{-\eps/3}) ,
	\end{split}
	\end{equation}
where $\mathcal{M}^{-1}F_{\delta}$ is the inverse Mellin transform
of $F_{\delta}$. Since $U_n(z)=U_n^{+}(z)$ are defined by~\eqref{eq:basisun},
this gives the first sum on the right-hand side of~\eqref{eq:add} .

On the other hand, by viewing the integral in \eqref{eq:defT} as two contour integrals and using the residue theorem, we find that
	\begin{equation} \label{eq:rhs2}
	(R_{\delta}F_{\delta})(s)= F_{\delta}(s) +  \sum_{\rho: 0<\gamma \le T} \operatorname{Res}_{w=\rho}H_{-\delta}(w,s)F_{\delta}(w) + E(s,T),
	\end{equation}
where
	\[ E(s,T) \ll \max_{1/2\le u\le 2}|\zeta(u+iT)|^{-1}  T^{-1/4+\eps}
	+\int_{|v|>T} |f(c+it)| (1+|v|)^{(2-c+\eps)/2} dv \]
for $\eps $ small enough.
Indeed, $D_{-\delta}(u+iv, s) \ll (1+|v|)^{3/4+\eps}$ by
Lemma~\ref{lem:Dest} and $F_{\delta}(u+iv)\ll (1+|v|)^{-1} $
for sufficiently small~$\eps$ by the assumption that~$f$ is in $\mathcal{H}_1$.
Given a large positive integer $k$, we choose $T_k$ in $[2^k, 2^{k+1}]$ such that
	\[ \max_{1/2 \le u \le 2} |\zeta(u+iT_k)|^{-1} \le
	\exp\big(c(\log\log T_k)^2\big), \]
which is feasible in view of Lemma~\ref{lem:Ind}.
Let us define \change{$U_n=\frac{n^{-1/4}}{2\pi}U_n^{+}$} and $V_{\rho,j}=V_{\rho,j}^{+}$
as in~\eqref{eq:basisun} and~\eqref{eq:basisvrho2}.
Then we obtain~\eqref{eq:add} by comparing the right-hand sides
of~\eqref{eq:rhs1} and~\eqref{eq:rhs2} and letting $k\to \infty$.

\subsubsection*{Rapid decay of the basis functions $U_n(z)$ and $V_{\rho,j}(z)$}
By~\eqref{eq:hnmoebius}, the functions $h_n^{\pm}(s)$ have rapid decay in
vertical strips since
$|\zeta^{*}(\sigma+it)|=O(t^{a_{\sigma}}e^{-\pi t/4})$
and by~\eqref{eq:alphapw} we have
that $|\alpha_{m,1/2}^{\pm}((\sigma+it)/2)|=O(e^{\pi t/4}(1+|t|)^{-k})$
for all $k>0$. Thus $U_n(z)$ is rapidly decaying.

To get the corresponding property of $V_{\rho,j}(z)$,
we use \eqref{eq:basisvrho2} to infer that it is sufficient to show that
$s\mapsto \zeta^{*}(s)\Aa_{1/2}^{\pm}(w,s)$ is rapidly decaying
in vertical strips (uniformly for $w$ in compact sets).
In view of~\eqref{eq:akintsym}, it is enough to check that
	\[s\mapsto
	\zeta^{*}(s)\int_{1}^{\infty}(F_k^{\pm}(it,s)-\alpha_{0,k}^{\pm}(s))(t^{w-1}\pm t^{k-w-1})dt\]
is rapidly decaying in vertical strips.
This is clear from Proposition~\ref{prop:rapiddecay}.

\subsubsection*{The interpolatory properties \eqref{eq:interp}}
From~\eqref{eq:reltosqrtn} and~\eqref{eq:hnmoebius}
we see that~$h_{n}^{\pm}(s)$ is the Mellin transform of
	\[ u_n^{\pm}(x)\coloneqq \sum_{k=1}^{\infty}
	\sum_{d^2|n} \mu(d) b^{\pm}_{n /d^2}(kx).\]
Therefore, for $m, n\ge 1$, from the interpolatory
properties of $b_m^{\eps}(x)$~\eqref{eq:sqrtninterpolation} we conclude that
	\[ u_n^{\eps} (\sqrt{m})
	= \sum_{k=1}^{\infty}
	  \sum_{d^2|n} \mu(d) b^{\eps}_{n /d^2}(k\sqrt{m})
	= \sum_{d| \sqrt{n/m}} \mu(d)
	= \begin{cases} 1, & m=n, \\ 0, & m\neq n. \end{cases}\]
Since~$U_n$ is an even function and \change{$\widehat{U_n}(\xi)=n^{-1/4}e^{\pi \xi} u_n^{-}(e^{2\pi \xi})$}, this implies the interpolatory properties of~$U_n$ at $\frac{\log m}{4\pi}$. By definition $h_n^{\pm}(s)$ vanishes at $s=\rho$ to the same order
as $\zeta^{*}(s)$, and we therefore also get
that $U_n^{(j)}(\frac{\rho-1/2}{i})=0$ for $0\le j < m(\rho)$.

Next, let us check the interpolatory properties of $V_{\rho,j}(z)$.
From~\eqref{eq:basisvrho2} we immediately see
that $V_{\rho,j}^{(j')}(\frac{\rho'-1/2}{i})=0$ for $0\le j'< m(\rho')$,
where $\rho'\ne \rho$ is a different nontrivial zero of~$\zeta$
with $\im \rho'>0$.
The property $V_{\rho,j}^{(j')}(\frac{\rho-1/2}{i})=\delta_{j,j'}$ for
$0\le j'< m(\rho)$ again follows from~\eqref{eq:basisvrho2}, since
for any $\eps>0$ such that $\eps<|\rho-\rho'|$ for all $\rho'\ne \rho$,
the difference
	\[
	V_{\rho,j}^{(j')}(z) - \frac{1}{2\pi i}\int_{|w-\rho|=\eps}
	\frac{i^{j'-j}(w-\rho)^j}{j!}
	\frac{\partial^{j'} H_{-}}{\partial s^{j'}}(w,1/2+iz)dw
	\]
is equal to $0$ for $\eps<|\rho-1/2\pm z|$ and to $\delta_{j,j'}$
for $\eps>|\rho-1/2\pm z|$.

Finally, we need to verify that
$\widehat{V_{\rho,j}}(\pm \frac{\log n}{4\pi})=0$. To this end, we consider
	\[U_{\pm}(w,x) \coloneqq \Gamma_{\RR}(w)\sum_{n\ge 1}\frac{\sum_{k\ge1}b_n^{\pm}(kx)}{n^{w/2}}\]
and note that $H_{\pm}(w,s)$ is the Mellin transform in $x$
of $U_{\pm}(w,x)/\zeta^{*}(w)$. The function $U_{\pm}(w,x)$ continues
analytically to a meromorphic function of $w$ in $\CC$ since it is
the Mellin transform of $F_{1/2}^{\pm}(it,\phi_x)$,
where $\phi_x(z)=\frac{1}{2}(\theta(zx)-1)$.
Setting $x=\sqrt{m}$ and using~\eqref{eq:sqrtninterpolation}, we obtain
	\[U_{\pm}(w,\sqrt{m})=\Gamma_{\RR}(w)
	\sum_{\substack{n\ge 1\\ n=k^2m}}\frac{1}{n^{w/2}}
	= \zeta^{*}(w)m^{-w/2}\]
first for $\re w$ sufficiently large, and then, by analytic continuation,
for all $w\ne 0,1$.
\change{The numbers $\widehat{V_{\rho,j}}(\xi)$, $j\ge0$ are simply the coefficients
of the principal part of the Laurent series of $w\mapsto U_{\pm}(w,e^{2\pi
\xi})/\zeta^{*}(w)$ at $w=\rho$. Since for $\xi=\frac{\log m}{4\pi}$ the latter 
function specializes to an entire function~$w\mapsto m^{-w/2}$, we get that
$\widehat{V_{\rho,j}}(\pm \frac{\log m}{4\pi})=0$}.

This concludes the proof of Theorem~\ref{thm:zeta}.\qed

\subsection{Relation with the Riemann--Weil formula} We return briefly to the viewpoint mentioned in the introduction, namely that we may think of the Riemann--Weil explicit formula as expressing a
linear functional $W$ in terms of our interpolation basis. This functional $W$ acts on functions $f$ in $\mathcal{H}_1$ and is defined by the left-hand side of \eqref{eq:RW}:
\begin{equation} \label{eq:Rdef} Wf:=\frac{1}{2\pi} \int_{-\infty}^{\infty} f(t)  \left(\frac{\Gamma'(1/4+it/2)}{\Gamma(1/4+it/2)} - \log \pi \right)dt+f\Big(\frac{i}{2}\Big)+f\Big(\frac{-i}{2}\Big).\end{equation}
By the equality in \eqref{eq:RW} and the interpolatory properties of the basis functions, we then find that
\[ WU_{n}=\begin{cases}  \pi^{-1} \Lambda(n)/\sqrt{n}, & n \ \text{a square} \\
                                                  0, & \text{otherwise,} \end{cases} \]
while
\[ WV_{\rho,j}=\begin{cases} 2, & j=0 \\
                                                  0, & j>0 . \end{cases} \]
It would be interesting to know whether these curious properties of the basis functions could be obtained without resorting to the Riemann--Weil formula. In the same vein, it may be worthwhile searching for further relations between our Fourier interpolation formula and the Riemann--Weil explicit formula.

\section{Fourier interpolation with zeros of Dirichlet $L$-functions \\ and other Dirichlet series}
\label{sec:lfunkernels}

The methods developed above give without much additional effort Fourier interpolation formulas associated with the nontrivial zeros of Dirichlet $L$-functions and, more generally, of functions that are representable by Dirichlet series and satisfy a functional equation of the form
$L^{*}(k-s)= \pm L^{*}(s)$, for
	\begin{equation*} \label{eq:FEq}
	L^{*}(s) = r^{s/2}\Gamma_{\RR}(s)L(s)
	\qquad \mbox{or}\qquad
	L^{*}(s) = r^{s/2}\Gamma_{\CC}(s)L(s)
	\end{equation*}
where $r$ is some positive number. Here we use the common notation
$\Gamma_{\RR}(s)\coloneqq\pi^{-s/2}\Gamma(s/2)$,
$\Gamma_{\CC}(s)\coloneqq2(2\pi)^{-s}\Gamma(s)$.
We will now present some key results of this kind, with emphasis on features
that have not appeared earlier in our treatise.

We begin with the case of Dirichlet characters $\chi$. We recall that the gamma
factor appearing in the functional equation for
$L(s,\chi)$ differs depending on whether $\chi$ is even or odd, i.e., on whether  $\chi(-n)=\chi(n)$  or $\chi(-n)=-\chi(n)$. This leads to a principal difference between the respective Fourier interpolation bases, and it is natural to treat the two cases separately. In either case, however, we will need the following analogue of Lemma~\ref{lem:Ind}.
\begin{lemma} \label{lem:Ind2} Let $\chi$ be a primitive Dirichlet character. There 
exists a positive constant \change{$c=c_{\chi}$} such that
\begin{equation} \label{eq:Lchar} \min_{T\le t \le 2T} \max_{1/2\le \sigma \le 2} |L(\sigma+it,\chi)L(\sigma+it, \overline{\chi})|^{-1} \le \exp\big(c(\log\log T)^2\big) \end{equation}
holds for all sufficiently large $T$.
\end{lemma}
The proof is word for word the same as that in \cite{RS}, with $\zeta(s)$ replaced by $L(s,\chi)L(s,\overline{\chi})$. We could prove the same bound for $\min_{T\le t \le 2T} \max_{1/2\le \sigma \le 2} |L(\sigma+it,\chi)|^{-1}$, which would suffice when $\chi$ is real. In the complex case, however, \eqref{eq:Lchar} is useful because we integrate along segments that cross the critical strip. Invoking the functional equation for $L(s,\chi)$, we then need lower bounds for both $|L(s,\chi)|$ and $|L(s,\overline{\chi})|$ along such segments.

\subsection{Fourier interpolation bases associated with even primitive characters} Theorem~\ref{thm:zeta} only deals with even functions, but the result extends painlessly to arbitrary functions in $\mathcal{H}_1$, as there is a similar interpolation formula for odd functions. In the case of complex characters, it is less natural to decompose the interpolation formula into even and odd parts. We therefore take the opportunity to state and prove in one stroke an interpolation formula for arbitrary functions.

\begin{theorem}\label{thm:Leven}
Let $\chi$ be a primitive even Dirichlet character to the modulus $q$ for some $q\ge 2$. There exist two sequences of rapidly decaying entire functions $U_n(z)$, $n \in \mathbb{Z}$, and $V_{\rho,j}(z)$, $0\le j <m(\rho)$, with $\rho$ ranging over the nontrivial zeros of $L(\chi, s)$, such that for every function $f$ in $\mathcal{H}_1$ and every $z=x+iy$ in the strip $|y|<1/2$ we have
	\begin{align}
	\nonumber
	f(z) \!=\! & \lim_{N\to \infty} \sum_{0<|n|\le N} \widehat{f}\left( \frac{\sgn (n)(\log |n| +\log q)}{4\pi}\right) U_n(z)\! +
	\! \left(\frac{f(-i/2)}{L^*(0,\chi)}+\frac{f(i/2)}{L^*(1,\chi)}\right)U_0(z) \\ &\label{eq:addLeven}
     +\lim_{k\to \infty} \sum_{|\gamma|\le T_k}
	\sum_{j=0}^{m(\rho)-1} f^{(j)}\left(\frac{\rho-1/2}{i}\right) V_{\rho,j}(z)
	\end{align}
for some increasing sequence of positive numbers $T_1, T_2, ...$ tending to $\infty$ that
\change{is independent of~$f$ and~$z$}. Moreover, the functions $U_n(z)$ and
$V_{\rho,j}(z)$ enjoy the following interpolatory properties:
\begin{equation} \label{eq:interpeven} \begin{array}{rclrcl} U_n^{(j)}\left(\frac{\rho-1/2}{i}\right) & = & 0,   & \widehat{U}_n\left(\frac{\sgn(n')(\log |n'|+\log q)}{4\pi}\right)& =& \delta_{n,n'} ,\\
V^{(j')}_{\rho,j}\left(\frac{\rho'-1/2}{i}\right)& = & \delta_{(\rho,j),(\rho',j')},  & \widehat{V}_{\rho,j}\left(\frac{\sgn (n)(\log |n|+\log q)}{4\pi}\right)& = &0,
\end{array}  \end{equation}
with $\rho, \rho'$ ranging over the nontrivial zeros of $L(s,\chi)$, $j,j'$ over all nonnegative integers less than or equal to respectively $m(\rho)-1, m(\rho')-1$, and $n, n'$ are in $\mathbb{Z}\sm \{0\}$.
\end{theorem}

The distinguished function $U_0(z)$ satisfies $U_0(0)=U_0(1)=1/2$ as well as
\begin{equation} \label{eq:Int0} U_0^{(j)}\left(\frac{\rho-1/2}{i}\right)  =  0 \quad \text{and}  \quad \widehat{U}_0\left(\frac{\sgn(n)(\log |n|+\log q)}{4\pi}\right) = 0 \end{equation}
when, as above, $\rho$ ranges over the nontrivial zeros of $L(s,\chi)$, $j$ over all nonnegative integers less than or equal to $m(\rho)-1$, and $n$ over all integers different from  $0$. The formula takes however a more involved and interesting form at the special points $0$ and $1$ as will be exhibited after we have established the theorem. As in the proof of Theorem~\ref{thm:zeta}, we will resort to the change of variable $z=(s-1/2)/i$ and use Mellin instead of Fourier transform. Most of the proof follows the same lines as that of Theorem~\ref{thm:zeta}, and we therefore omit some of the computations.
\begin{proof}[Proof sketch of Theorem~\ref{thm:Leven}]
We set
	\[ H_\delta (w,s;\chi) \coloneqq
	\frac{L^{*}(s,\chi)}{2L^{*}(w,\chi)} \Aa_{1/2}^{\delta}(w/2,s/2),\]
where $L^{*}(s,\chi)=q^{s/2}\Gamma_{\RR}(s)L(s,\chi)$.
These kernels satisfy the functional equations
	\begin{equation} \label{eq:fuL}  H_{\delta}(w,s;\chi)= \delta w(\chi) H_{\delta}(1-w,s;\overline{\chi}) ,\end{equation}
where $w(\chi)$ is the root number of $L(s,\chi)$.
We now consider the operator
	\begin{align} \label{eq:Tchi} (T_\chi F)(s) &  : =  \frac{1}{4\pi i} \int_{c-i\infty}^{c+i\infty} \left[H_{+}(w,s;\chi) F(w) + w(\chi)H_{+}(1-w,s;\overline{\chi}) F(1-w)\right] dw \\ \nonumber
	& \ + \frac{1}{4\pi i} \int_{c-i\infty}^{c+i\infty} \left[H_{-}(w,s;\chi) F(w) - w(\chi) H_{-}(1-w,s;\overline{\chi}) F(1-w)\right] dw ,\end{align}
where both $1-c < \re s < c$ and $c>1$, with the additional assumption
that $F(s)$ is analytic in $1-c\le \re s
	\le c$ and
	\[ \sup_{1-c\le \sigma \le c} \int_{\infty}^{\infty} |F(\sigma+it)| (1 +|t|)dt<\infty . \]
Following the corresponding computation in Section~\ref{sec:maintheoremproof},
we may compute the integrals on the right-hand side of \eqref{eq:Tchi} and get
	\begin{align*} (T_{\chi}F)&(s) =\frac{1}{4}L^{*}(s,\chi) \sum_{n=1}^{\infty}(\mathcal{M}^{-1} F)((qn)^{1/2}) \sum_{d^2|n} \mu(d)\chi(d) (\alpha^+_{n/d^2} (s) +\alpha_{n/d^2}^{-}(s) ) \\
	& + \frac{1}{4} L^{*}(s,\chi) w(\chi) \sum_{n=1}^{\infty}(\mathcal{M}^{-1} F)((qn)^{-1/2})(qn)^{-1/2} \sum_{d^2|n} \mu(d)\overline{\chi}(d) (\alpha^-_{n/d^2} (s) -\alpha_{n/d^2}^{+}(s) ), \end{align*}
where we have set $\alpha_{n}^{\pm}(s)\coloneqq \alpha_{n,1/2}^{\pm}(s/2)$.
On the other hand, viewing the two integrals on the right-hand side of \eqref{eq:Tchi} as contour integrals, we may use the functional equations \eqref{eq:fuL} and the residue theorem to deduce that
	\begin{align*} (T_{\chi}F)(s)&= f(s)-\alpha_0^{-}(s) q^{s/2} L(s,\chi)\left(\frac{F(0)}{L^*(0,\chi)}+\frac{F(1)}{L^*(1,\chi)}\right) \\
	& + \frac{1}{2}\lim_{k\to\infty} \sum_{|\gamma|\le T_k} \big(\operatorname{Res}_{w=\rho} H_+(w,s;\chi)F(w)+\operatorname{Res}_{w=\rho}H_-(w,s;\chi)F(w)\big)
	\end{align*}
for some sequence $T_k\to \infty$. To achieve this, we use the sub-convexity bound $L(1/2+it)\ll q^{1/2} |t|^{1/6} \log (q|t|) $ (see \cite[p.~149]{HB}), along with Lemma~\ref{lem:Dest} and Lemma~\ref{lem:Ind2}. We arrive at \eqref{eq:addLeven} by equating our two expressions for $(T_{\chi}F)(s)$ and changing back to Fourier transforms and the variable $z=(s-1/2)/i$.
\end{proof}
To evaluate \eqref{eq:addLeven} at $s=1$, we recall from~\eqref{eq:alplus} that
	\[\Aa_{1/2}^{+}(w/2,1/2)=-2\zeta^{*}(w). \]
Plugging this into \eqref{eq:addLeven}, we arrive at
	\begin{align}
	\nonumber f(-i/2)+\frac{L^*(1,\chi)}{L^*(0,\chi)}f(i/2) & = q^{1/2} L(1,\chi) \sum_{n=1}^\infty (qn)^{-1/2} \widehat{f}\Big(\frac{\log n + \frac{1}{2} \log q}{2\pi}\Big)
	\prod_{p|n} (1-\chi(p)) \\
	\label{eq:fonehalf}
	& - q^{1/2} L(1,\chi)w(\chi) \sum_{n=1}^\infty (qn)^{-1/2} \widehat{f}\Big(\frac{-\log n - \frac{1}{2} \log q}{2\pi}\Big)
	\prod_{p|n} (1-\overline{\chi}(p)) \\
	& \nonumber +q^{1/2} L(1,\chi) \lim_{k\to\infty} \sum_{|\gamma|\le T_k} \operatorname{Res}_{w=\rho} \frac{\zeta(w)}{q^{w/2} L(w,\chi)} f(w),
	\end{align}
which has a curious resemblance with the Riemann--Weil Formula.
%

\subsection{Fourier interpolation bases associated with odd characters}

In this case of odd characters, the gamma factor for $L(s,\chi)$
is $\pi^{-(s+1)/2} \change{\Gamma((s+1)/2)}$, and thus we will use
the function $\Aa_{3/2}^{\pm}(\frac{w+1}{2},\frac{s+1}{2})$
which involves the same gamma factor. Note that
the abscissa of convergence and the abscissa of absolute convergence
of $w\mapsto \Aa_{3/2}^{\pm}(\frac{w+1}{2},\frac{s+1}{2})$
are both equal to $2$, and we therefore need to require that functions be
analytic in a strip of width $3+\eps$. As in the preceding cases, we
need a growth condition in the strip, but we may require less at the boundary
of the strip because the Dirichlet series kernel converges absolutely there.
We find it convenient to restrict to functions~$f$ that are analytic in
the strip $|\im\, z|<3/2+\eps$ and satisfy
	\[ |f(x+iy)|\ll (1+|x|)^{-1-\eps} \]
for some $\eps>0$, where we in the latter inequality assume
that $|y|\le (3+\eps)/2$. We let~$\mathcal{H}_2$ denote the space
of all such functions.

\begin{theorem}\label{thm:Lodd}
Let $\chi$ be a primitive odd Dirichlet character to the modulus $q$ for some $q\ge 3$. There exist two sequences of rapidly decaying entire functions $U_n(z)$, $n \in \mathbb{Z}$, and $V_{\rho,j}(z)$, $0\le j < m(\rho)$, with $\rho$ ranging over the nontrivial zeros of $L(\chi, s)$, such that for every function $f$ in $\mathcal{H}_2$ and every $z=x+iy$ in the strip $|y|<1/2$ we have
	\begin{align} \nonumber
	f(z)\!=\! & \lim_{N\to \infty} \sum_{0<|n|\le N} \widehat{f}\left( \frac{\sgn (n)(\log |n| +\log q)}{4\pi}\right)\! U_n(z)+\left(\frac{f(-3i/2)}{L^*(0,\chi)}\!+\!\frac{f(3i/2)}{L^*(1,\chi)}\right)\! U_0(z) \\ &+ \lim_{k\to \infty} \sum_{0<|\gamma|\le T_k}
	\sum_{j=0}^{m(\rho)-1} f^{(j)}\left(\frac{\rho-1/2}{i}\right) V_{\rho,j}(z)
	\label{eq:addodd}  \end{align}
for some increasing sequence of positive numbers $T_1, T_2, ...$ tending to $\infty$ that does not depend on neither $f$ nor on $z$. Moreover, the functions $U_n(z)$ and $V_{\rho,j}(z)$ enjoy the following interpolatory properties:
\begin{equation} \label{eq:interp2} \begin{array}{rclrcl} U_n^{(j)}\left(\frac{\rho-1/2}{i}\right) & = & 0,   & \widehat{U}_n\left(\frac{\sgn(n')(\log |n'|+\log q)}{4\pi}\right)& =& \delta_{n,n'} ,\\
V^{(j')}_{\rho,j}\left(\frac{\rho'-1/2}{i}\right)& = & \delta_{(\rho,j),(\rho',j')},  & \widehat{V}_{\rho,j}\left(\frac{\sgn (n)(\log |n|+\log q)}{4\pi}\right)& = &0,
\end{array}  \end{equation}
with $\rho, \rho'$ ranging over the nontrivial zeros of $\zeta(s)$, $j,j'$ over all nonnegative integers less than or equal to respectively $m(\rho)-1, m(\rho')-1$, and $n, n'$ are in $\mathbb{Z}\sm \{0\}$.
\end{theorem}

As in the case of even characters, the distinguished function $U_0(z)$ satisfies \eqref{eq:Int0}, and we also have $U_0(\pm 3i/2)=1/2$. The formula at the special points $\pm 3i/2$ will be somewhat more complicated than~\eqref{eq:fonehalf} and will instead of $\zeta(w)$ involve the Dirichlet
series
	\[-\sum_{n=1}^{\infty} r_3(n) n^{-(w+1)/2}, \]
where $r_3(n)$ is the number of representations of $n$ as the sum
of squares of~$3$ integers.

\begin{proof}
We define the Dirichlet series kernel
	\[ H_\delta (w,s;\chi) \coloneqq
	\frac{L^{*}(s,\chi)}{2L^{*}(w,\chi)}
	\Aa_{3/2}^{\delta}\Big(\frac{w+1}{2},\frac{s+1}{2}\Big)\]
and follow the same argument as in Theorem~\ref{thm:Leven}.
\end{proof}

\subsection{Fourier interpolation with zeros of other Dirichlet series}
\label{subsect:other}
We may deduce an abundance of further Fourier interpolation formulas based on the
techniques developed in Section~\ref{sec:zetakernels}, as will now be briefly
\change{explained. Detailed analysis of these generalizations falls outside the
scope of this paper, so we only sketchily outline the details, and in}
each case, we will only indicate the construction of the corresponding Dirichlet series
kernels $H_{\pm}(w,s)$ that lead to the interpolation formula with $\Lambda$
being the (multi-)set given by a suitable rotation of the nontrivial zeros of $L(s)$.

\subsubsection{\change{Dedekind zeta functions of imaginary quadratic fields and Hecke
$L$-functions of modular forms}}
First, we obtain Fourier interpolation formulas
associated with zeros of Dedekind zeta functions~$\zeta_K(s)$ for
imaginary quadratic fields~$K$. In this case we define
	\[H_{\pm}(w,s) \coloneqq \frac{\zeta_{K}^{*}(s)}{\zeta_{K}^{*}(w)}
	\Aa_{1}^{\pm}(w,s),\]
where $\zeta_K^{*}(s)=|\Delta_K|^{s/2}\Gamma_{\CC}(s)\zeta_K(s)$
is the completed zeta function and $\Gamma_{\CC}(s)=2(2\pi)^{-s}\Gamma(s)$.
More generally, this construction applies to products of two
Dirichlet $L$-functions whose (primitive) characters have different parity,
i.e., $L(s)=L(s,\chi_1)L(s,\chi_2)$ where $\chi_1$ is an even character
and $\chi_2$ is an odd character. In this case the corresponding sequence
of points on the Fourier side is $\Lambda^*=\{\pm\log (2n\sqrt{N})/(2\pi)\}$,
where $N$ is the conductor, i.e., $N=|\Delta_K|$ for $L(s)=\zeta_K(s)$
and $N=q_1q_2$ for $L(s)=L(s,\chi_1)L(s,\chi_2)$.

Next, using $\Aa_k^{\pm}(w,s)$ for an even integer $k\ge 2$
we can construct Fourier interpolation formulas associated to zeros of Hecke
$L$-functions of modular forms. More precisely, let {\color{blue} $f$ in $S_k(\Gamma_{0}(N))$
be} a normalized Hecke newform. Then $L(s,f)=\sum_{n\ge 1}a_nn^{-s}$ admits
an analytic continuation to~$\CC$ and satisfies $L^{*}(k-s,f)=L^{*}(s,f)$,
where $L^{*}(s,f)=N^{s/2}\Gamma_{\CC}(s)L(f,s)$ is the completed $L$-function.
In this case, we define $H_{\pm}(w,s)$ as
	\[H_{\pm}(w,s) \coloneqq \frac{L^{*}(s,f)}{L^{*}(w,f)}
	\Aa_{k}^{\pm}(w,s).\]
The formula again involves the sequence
$\Lambda^*=\{\pm\log (2n\sqrt{N})/(2\pi)\}$.

\subsubsection{\change{Dirichlet series without Euler products}}
Furthermore, we may obtain a continuous family of Fourier interpolation formulas associated with the sequence of points $\Lambda^*=\{0, \pm (\log n)/(4\pi), n\ge 1\}$ by starting from kernels of the form
	\begin{equation} \label{eq:Hgen}
	H^{\pm}_{k,\eps, s_0}(w,s)\coloneqq \frac{\Aa^{\eps}_{k}(s/2,s_0)}
	{\Aa^{\eps}_{k}(w/2,s_0)}
	\Aa^{\pm \eps}_{k}(w/2,s),
	\end{equation}
where $s_0$ is some fixed point satisfying $0\le \re s_0 \le k$.
The (multi-)set $\Lambda$ dual to $\Lambda^*$ will now be a suitable rotation
and translation of the zero set of $\Aa^{\eps}_{k}(w/2,s_0)$.

We may however need to put more severe restrictions on the functions represented by Fourier interpolation formulas associated with kernels such as those defined by \eqref{eq:Hgen}. Indeed, since the denominator of \eqref{eq:Hgen} in general can not be represented as an Euler product, the techniques from \cite{RS} must be abandoned. In the absence of a multiplicative structure, we may resort to the following classical argument, yielding a  much cruder bound than that of Lemma~\ref{lem:Ind}. We use then that the function $A(w)\coloneqq\Aa^{\eps}_{k}(w/2,s_0)$  grows polynomially in the vertical direction, and hence, by Jensen's formula, the number of zeros in a strip of width $1$ at height $T$ is $O(\log T)$. An application of the Borel--Carath\'{e}odory theorem shows that
	\[ \frac{|A(w)|}{\prod_{|\gamma-T|\le 1} |s-\rho|} \ge T^{-C} \]
for some constant $C$ when $|\im w-T| \le 1/2$. Hence we can find an $\eta$, $|\eta|\le1/ 2$, such that
	\begin{equation} \label{eq:weak}
	\int_{1-\sigma_0}^{\sigma_0} |F(\sigma+i(T+\eta)|^{-1} d\sigma
	\le T^C \int_{1-\sigma_0}^{\sigma_0}
	\prod_{|\gamma-T|\le 1} |\sigma+i(T+\eta)-\rho)|^{-1} d\sigma
	\ll T^{c\log\log T}
	\end{equation}
for some constant $c$. The bound in \eqref{eq:weak} is our replacement for Lemma~\ref{lem:Ind}, and, consequently, we need to require that functions decay accordingly.

The individual basis functions arising from kernels of the form \eqref{eq:Hgen} will have additional poles at $s=2s_0$ and $s=2(k-s_0)$, and hence they do not belong to any nice function space. The situation is particularly bad in the most natural case when $s_0$ is on the line of symmetry $\re s=k/2$. Then the inverse Fourier transform of any basis function is neither integrable nor in any $L^p$ space for $p<\infty$. This is a less attractive feature of ``basis decompositions'' stemming from \eqref{eq:Hgen}.

\subsubsection{\change{Further extensions and more complicated gamma factors}}
Further extensions are obtained if we apply algebraic operations that preserve functional equations in a suitable way. We may for instance take linear combinations of $L$-functions that satisfy the same functional equation. Moreover, we may, for an arbitrary polynomial $P$ of degree $n$, multiply for example a Dirichlet $L$-function by $r^{ns/2} P(r^{-(s-1/2)})$ with $r>1$ an integer and $P(0)\neq 0$. Then the polynomial
	\[ Q(z)=z^n P(1/z) \]
will appear in the functional equation, where $n$ is the degree of $P$. In other words, \emph{any} complex arithmetic progression with common difference $2\pi /\log r$ may be adjoined to a given multi-set $\Lambda$ stemming from the nontrivial zeros of an $L$-function to which our methods apply. This allows us, in particular, to establish a Fourier interpolation formula associated with every Dirichlet $L$-function $L(s,\chi)$, irrespective of whether $\chi$ is primitive or not. As should be clear from the preceding two subsections, by adjoining such an arithmetic progression, we will find that both the negative and positive part of the sequence $\Lambda^*$ are ``pushed away'' by $\log r/(4\pi)$ from the origin.


We can also treat Dirichlet series with more complicated gamma factors,
although in this case the results are less satisfactory in the sense
that while we get a Fourier interpolation formula, in general the resulting
functions $U_n$ and $V_{\rho,j}$ will no longer form a basis, and
we will not get the interpolatory properties like in~\eqref{eq:interp}.
For example, let $L(s)$ satisfy $L^{*}(1-s)=L^{*}(s)$ where
$L^{*}(s)=N^{s/2}\Gamma_{\RR}^{r_1}(s)\Gamma_{\CC}^{r_2}(s)L(s)$,
and let $L_0$ be another Dirichlet series such
that $L_0^{*}(s)=N_0^{s/2}\Gamma_{\RR}^{r_1-1}(s)\Gamma_{\CC}^{r_2}(s)\change{L_0}(s)$.
Then we can form a Dirichlet series kernel
	\[H_{\pm}(w,s) \coloneqq \frac{L^{*}(s)}{L^{*}(w)}
	L_0^{*}(w)\Aa_{1/2}^{\pm}(w/2,s/2),\]
which has the expected poles and leads to a Fourier interpolation formula
with zeros of~$L(s)$
and $\Lambda^*=\{0, \pm \frac{\log(nN/N_0)}{4\pi}, n\ge 1\}$.


\subsubsection{\change{Density of interpolation nodes}}
Let us finally note that Kulikov's inequality for $\Lambda$ and $\Lambda^*$ (see
\eqref{eq:TF}) will hold in the same marginal way as when $\Lambda$ consists of the zeros
of $\zeta(s)$, whenever these sequences (or multisets) are constructed as indicated above.
Indeed, we may in all \change{the above} cases establish a Riemann--von Mangoldt formula.
We may for instance observe that the number of nontrivial zeros $\rho=\beta+i\gamma$ of
$A(w)$ satisfying $|\gamma|\le T$ will be
	\[ \label{eq:RM} \frac{T}{\pi} \log \frac{T}{2\pi e} +O(\log T).\]
We follow the standard approach to prove this, i.e., we apply the argument principle along with the functional equation, and we use the Hadamard product of $A(w)$ to control its logarithmic derivative. If we adjoin arithmetic progressions to $\Lambda$ by multiplying, say, a Dirichlet $L$-function by $r^{ns/2}P(r^{-(s-1/2)})$, then there will be an additional term
	\[ \frac{nT}{\pi} \log r  \]
in the Riemann--von Mangoldt formula, balancing the ``repulsion'' by $n\log r/(4\pi)$ of the sequence $\Lambda^*$ from the origin.

\section{Estimates for the Fourier coefficients of $F_{k}^{\pm}(\tau,s)$}
\label{sec:estimates}
In this section we will derive estimates for the growth
of $\alpha_{n,k}^{\pm}(s)$ and certain related quantities as
a function of $n$. As before, we assume that $k\ge0$.

First, it will be convenient to define two quantities related
to $\gamma_{\tau}$, which we recall is the (generically unique)
element of $\Gamma_{\theta}$ that maps $\tau\in\HH$ to the
fundamental domain~$\bF$. The first quantity is $\bH(\tau)$,
the imaginary part of $\gamma_{\tau}\tau$, i.e.,
	\[\bH(\tau) \,\change{:=}\, \im\gamma_{\tau}\tau .\]
It is easy to see that the function $\bH\colon\HH\to\RR$ is continuous
and $\Gamma_{\theta}$-invariant.
The second quantity is $\bN\colon\HH\to\ZZ_{\ge0}$, which we define as one plus
the number of inversions~$S$ that appear in the canonical representation
of $\gamma_{\tau}$, i.e.,
	\[\bN(\tau) \,\change{:=}\,  j+\eps_0+\eps_1  ,\qquad  \gamma_{\tau}=
	S^{\eps_0}T^{2m_1}ST^{2m_2}\dots ST^{2m_j}S^{\eps_1}.\]
In cases when $\gamma_{\tau}$ is not uniquely defined, i.e.,
for $\tau\in\Gamma_{\theta}\partial\bF$, we set $\bN(\tau)$ to be the larger
\change{of the two possible values}.

\subsection{Estimates for the contour integral}
Our first step is to show that $F_k^{\pm}(\tau,s)$ is bounded
for $\tau\in\bF$. We prove this for the slightly more general functions
$F_{k}^{\pm}(\tau,\phi)$ from Theorem~\ref{thm:modint}, as long as $\phi$ is
bounded on the domain~$\mathcal{D}$ illustrated in Figure~\ref{fig:crescent}.
Explicitly we set
\[ \mathcal{D}\coloneqq\{\tau\in \HH\colon |\re\tau|<1,\; \sqrt{3/4}<|\tau|<\sqrt{4/3},\; |\tau\pm 1/2|>1/2\};\]
the particular shape of~$\mathcal{D}$ is not important as long as it contains
the geodesic from~$-1$ to~$1$ and lies in $\ol{\bF\cup S\bF}$.

\begin{figure}[h]
	\centering
	\begin{tikzpicture}
	\definecolor{cv0}{rgb}{0.95,0.95,0.95}
	\definecolor{cv1}{rgb}{0.90,0.90,0.90}
	\clip(-9,-0.5) rectangle (7,3);
	\begin{scope}[scale=0.8,xshift=-1cm]
	\draw[lightgray] (-5,0) -- (5,0);
	\draw[lightgray,dashed] (-3,0) arc  (180:0:1.5);
	\draw[lightgray,dashed]  (3,0) arc  (0:180:1.5);
	\draw[lightgray,dashed]  (-3,0) -- (-3,3);
	\draw[lightgray,dashed]  (3,0) -- (3,3);
	\fill[color=cv1] (-3,0) -- (-3,1.7320) arc (150:30:3.4641) -- (3,0) arc (0:60:1.5) arc (30:150:2.5980) arc (120:180:1.5);
	\draw[gray] (-3,0) -- (-3,1.7320) arc (150:30:3.4641) -- (3,0) arc (0:60:1.5) arc (30:150:2.5980) arc (120:180:1.5);
	\draw[gray,dashed] (-3,0) arc  (180:0:3);
	\draw (0,2.6) node[above]{$\mathcal{D}$};

	\fill[black] (-3,0) circle (0.07) node[below] {$-1$};
	\fill[black] (3,0) circle (0.07) node[below] {$1$};


	\end{scope}
	\end{tikzpicture}
	\caption{The domain $\mathcal{D}$.}
	\label{fig:crescent}
\end{figure}
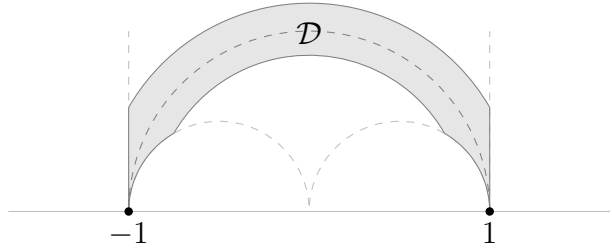

\begin{proposition} \label{prop:bdfdomain}
	Let~$k\ge0$ and $\phi\colon\HH\to\CC$ be analytic such that
	$|\phi(\tau)|\le C_{\phi}$ for $\tau$ in $\mathcal{D}$, where
	$\mathcal{D}$ is defined as above. Then for $\tau\in\bF$ we have
	\[|F_{k}^{\pm}(\tau,\phi)| \ll_k \begin{cases}
    C_{\phi},\qquad\qquad\qquad\qquad k\pm 1\not\in1+4\ZZ\,,\\
    C_{\phi}\im(\tau)^{-1},\qquad \mbox{otherwise}\,.
    \end{cases}\]
\end{proposition}
\begin{proof}
Without loss of generality we may assume that $\re\tau$ is in $(-1,0)$,
$|\tau|>1$, and that $|\tau-i|>1/10$ (for $\tau$ close to $i$ we get the claimed bound using the contour deformation from Figure~\ref{fig:contour}). We will also assume that
$\im \tau \le 1/2$ (for $\im \tau >1/2$ the claim follows from the same
argument, but with simpler estimates).
By definition of $F_k^{\pm}(\tau,\phi)$ we have
	\[F_k^{\pm}(\tau, \phi) = \frac{1}{2}\int_{-1}^{i}
	\mathcal{K}_k^{\pm}(\tau,z)(\phi(z)\mp (z/i)^{-k}\phi(-1/z))dz.\]
Set $\phi_{\pm}(z):=\frac{1}{2}(\phi(z)\mp (z/i)^{-k}\phi(-1/z))$.
Note that $\phi_{+}(i)=0$. Since~$\mathcal{D}$ is invariant
under~$z\mapsto -1/z$, by assumption we have
	\[|\phi_{\pm}(z)| \ll_{k} C_{\phi},\qquad z\in\mathcal{D}.\]

To estimate the integral, we change the variable of integration to $w=J(z)$
and deform the contour of integration in $w$ from the segment $[0,64]$ to
a curve $\ell$ in the lower half-plane (see Figure~\ref{fig:jconf})
in such a way that its preimage under $J$ in $\ol{\bF\cup S\bF}$ lies
in $\mathcal{D}$. Using~\eqref{eq:theta4dz} to make the change of variables,
we obtain
	\begin{align*}
	F_k^{-}(\tau, \phi) & =
	\frac{J_{-}(\tau)\theta^{2k}(\tau)}{J^{\nu_{-}-1}(\tau)}
	\int_{-1}^{i}\frac{J^{\nu_{-}}(z)}{\theta^{2k}(z)}
	\frac{1} {J(\tau)-J(z)} \phi_{-}(z) \theta^{4}(z)dz \\
	& = \frac{1}{\pi}\frac{J_{-}(\tau)\theta^{2k}(\tau)}{J^{\nu_{-}-1}(\tau)}
	\int_{0}^{64}\frac{1}{\theta^{2k}(t(w))}\frac{\phi_{-}(t(w))}
	{J(\tau)-w}w^{\nu_{-}-1/2}(64-w)^{-1/2}dw,
	\end{align*}
	where $t(w)$ is the inverse function to $z\mapsto J(z)$.
	Since $J(\tau)$ belongs to the upper half-plane and $w$ belongs to the lower,
	we have $|J(\tau)-w|\gg \sqrt{|J(\tau)|^2+|w|^2}$.
	From this we get
	\begin{align*}
	|F_k^{-}(\tau, \phi)|
	&\ll C_{\phi}\frac{|J_{-}(\tau)\theta^{2k}(\tau)|}{|J^{\nu_{-}-1}(\tau)|}
	\int_{\ell}\frac{|w|^{\nu_{-}-1/2}|64-w|^{-1/2}}{|\theta^{2k}(t(w))|\sqrt{|J(\tau)|^2+|w|^2}}|dw|\\
	&\ll_k C_{\phi}\frac{|J_{-}(\tau)\theta^{2k}(\tau)|}{|J^{\nu_{-}-1}(\tau)|}
	\Big(1+\int_{0}^{-ie^{-1}}\frac{|w|^{\nu_{-}-(k+2)/4}|64-w|^{-1/2}}{\log^k|w^{-1}|\sqrt{|J(\tau)|^2+|w|^2}}|dw|\Big)\\
	&\ll C_{\phi}\frac{|J_{-}(\tau)\theta^{2k}(\tau)|}{|J^{\nu_{-}-1}(\tau)|}
	\Big(1+\int_{0}^{e^{-1}}\log^{-k}(t^{-1})\frac{t^{-\{(k+2)/4\}}}{\sqrt{|J(\tau)|^2+t^2}}dt\Big).
	\end{align*}
	From this we obtain (using Lemma~\ref{lem:expint} below), for $k\not\in 2+4\ZZ$,
	\begin{equation*} \label{eq:estimate_contour_plus}
	|F_k^{-}(\tau, \phi)| \ll_k
	C_{\phi}\frac{|J_{-}(\tau)\theta^{2k}(\tau)|}
	{|J^{(k-2)/4}(\tau)|}
	\im (\tau)^{k}\ll_k C_{\phi}
	\end{equation*}
	when $\tau$ approaches the real line inside the fundamental domain. For $k\in
	2+4\ZZ$ using the same argument we obtain
    \begin{equation*} \label{eq:estimate_contour_plus}
   	|F_k^{-}(\tau, \phi)| \ll_k C_{\phi}\im(\tau)^{-1}\,.
   	\end{equation*}

	For $F_k^{+}$ we calculate
	\begin{align*}
	F_k^{+}(\tau, \phi) &=
	\frac{\theta^{2k}(\tau)}{J^{\nu_{+}-1}(\tau)}
	\int_{-1}^{i}\frac{J^{\nu_{+}}(z)}{\theta^{2k}(z)}\frac{J_{-}(z)}
	{J(\tau)-J(z)}\phi_{+}(z)\theta^{4}(z)dz \\
	&= \frac{1}{\pi}\frac{\theta^{2k}(\tau)}{J^{\nu_{+}-1}(\tau)}
	\int_{0}^{64}\frac{1}{\theta^{2k}(t(w))}\frac{\phi_{+}(t(w))}
	{J(\tau)-w}w^{\nu_{+}-1}dw
	\end{align*}
	and again using Lemma~\ref{lem:expint} we get
	\begin{align*}
	|F_k^{+}(\tau, \phi)| &\ll \frac{|\theta^{2k}(\tau)|}{|J^{\nu_{+}-1}(\tau)|}
	\int_{\ell}\frac{|\phi_{+}(t(w))|}
	{|\theta^{2k}(t(w))|}\frac{|w|^{\nu_{+}-1}}{\sqrt{|J(\tau)|^2+|w|^2}}|dw|
	\\
	&\ll_k C_{\phi}\frac{|\theta^{2k}(\tau)|}{|J^{\nu_{+}-1}(\tau)|}
	\Big(1+\int_{0}^{e^{-1}}\log^{-k}(t^{-1})\frac{t^{-\{(k+4)/4\}}}{\sqrt{|J(\tau)|^2+t^2}}dt\Big).
	\end{align*}
	Thus for $k\not\in 4\ZZ$ we have
	\begin{equation*} \label{eq:estimate_contour_minus}
	|F_k^{+}(\tau, \phi)| \ll_k C_{\phi}\frac{|\theta^{2k}(\tau)|}{|J^{k/4}(\tau)|}
	\im (\tau)^{k} \ll_k C_{\phi},
   	\end{equation*}
    and for $k\in4\ZZ$ we get $|F_k^{+}(\tau, \phi)| \ll_k C_{\phi}\im
    (\tau)^{-1}$.
\end{proof}

In the proof above we have used the following simple lemma.
\begin{lemma} \label{lem:expint}
	Let $\alpha$ be a real number and $\beta$ be a number in $(-1,0]$.
	Then as $T\to\infty$
	\[
	\int_{0}^{e^{-1}}\frac{t^{\beta} \log^{\alpha}(t^{-1})}{\sqrt{T^{-2}+t^2}}dt
	\asymp
    \begin{cases}
        T^{-\beta}\log^{\alpha} T \,,\qquad \beta < 0\,,\\
        \log^{\alpha+1} T \,,\qquad\quad \beta = 0,\;\alpha\ne -1\,.
    \end{cases}
	\]
\end{lemma}
\begin{proof}
	For $\beta< 0$ we have
	\begin{multline*}
	\int_{0}^{e^{-1}}
	\frac{t^{\beta} \log^{\alpha}(t^{-1})}{\sqrt{T^{-2}+t^2}}dt
	\asymp
	T\int_{0}^{T^{-1}}t^{\beta} \log^{\alpha}(t^{-1}) dt
	+\int_{T^{-1}}^{e^{-1}}t^{\beta-1}\log^{\alpha}(t^{-1}) dt
	\\
	= T E_{-\alpha}((1+\beta)\log T) \log^{1+\alpha} T
	   + \int_{1}^{\log T}e^{-\beta x}x^{\alpha}dx
	\asymp T^{-\beta}\log^{\alpha} T,
	\end{multline*}
	where $E_a(x)\coloneqq\int_{1}^{\infty}\frac{e^{-xt}}{t^a}dt\sim
	\frac{e^{-x}}{x}$, $x\to+\infty$, and the implied constants only depend
    on~$\alpha$ and~$\beta$. Similarly, for $\beta=0$ we have
	\begin{equation*}
	\int_{0}^{e^{-1}}
	\frac{\log^{\alpha}(t^{-1})}{\sqrt{T^{-2}+t^2}}dt
	\asymp
    T E_{-\alpha}(\log T) \log^{1+\alpha} T
	   + \frac{\log^{\alpha+1}(T)}{\alpha+1} \asymp \log^{\alpha+1} T\,. \qedhere
	\end{equation*}
\end{proof}

In particular, since $\phi(\tau)=(\tau/i)^{-s}$ is obviously bounded
in $\mathcal{D}$, the above proposition applies to $F_{k}^{\pm}(\tau,s)$
and shows that it is bounded in~$\bF$.

\subsection{Estimates for $F_k^{\pm}(\tau,\phi)$ near the real line}
\label{sec:cocycleestimates}
To estimate $F_k^{\pm}(\tau,\phi)$ away from the fundamental
domain we repeatedly apply periodicity and the functional
equation~\eqref{eq:fksfeq_general}.
Let us denote by $|$ the slash operator $|_k^{\pm}$ in weight $k$ twisted by
the character of~$\Gamma_{\theta}$ that sends $S$ to~$\pm1$.
To further simplify the notation we will write $F(\tau)$ instead
of $F_k^{\pm}(\tau,\phi)$.

Let us denote $\psi=2\phi_{\pm}$.
We define a 1-cocycle $\{\psi_{\gamma}\}_{\gamma\in\Gamma_{\theta}}$
in such a way that $\psi_{S}=\psi$ and $\psi_{T^2}=0$. In other words,
the functions $\psi_{\gamma}$ satisfy
	\[\psi_{\gamma_1\gamma_2}
	=\psi_{\gamma_2}
	+\psi_{\gamma_1}|\gamma_2
	,\qquad \gamma_1,\gamma_2\in\Gamma_{\theta}.\]
Any such $1$-cocycle is uniquely determined by $\psi_S$ and $\psi_{T^2}$
since $\Gamma_{\theta}$ is generated by $S$ and $T^2$, and
since the only relation between the generators is $S^2=1$
and by definition $\psi$ satisfies $\psi+\psi|S = 0$, the
family of functions $\{\psi_{\gamma}\}_{\gamma}$ is indeed well-defined.
Note that, more generally, for $\gamma_1,\dots,\gamma_n\in\Gamma_{\theta}$
we have
	\begin{equation}  \label{eq:cocycleeq}
	\psi_{\gamma_1\gamma_2\dots\gamma_n}
	=\psi_{\gamma_n}+\psi_{\gamma_{n-1}}|\gamma_n
	+\dots+\psi_{\gamma_1}|\gamma_2\dots\gamma_n.
	\end{equation}

Since $\gamma\mapsto F-F|\gamma$ and $\gamma\mapsto\psi_{\gamma}$
are 1-cocycles that take equal values on the group generators, we have
$F-F|\gamma=\psi_{\gamma}$ for all $\gamma\in\Gamma_{\theta}$.

Let $\tau\in\HH$ be such that $\re \tau\in(-1,1)$
and $|\tau|<1$, and consider the element $\gamma\in\Gamma_{\theta}$ that
sends $\tau_0\in\bigcup_{j\in\ZZ}(2j+\mathcal{F})$ to $\tau$.
Let us write $\gamma$ as $ST^{2m_n}S\dots T^{2m_1}S$, $m_i\in\ZZ\sm \{0\}$,
which we can assure by changing $\tau_0$ by a translation, if necessary.
Then, using~\eqref{eq:cocycleeq} and the fact that $\psi_{T^2}=0$ we get
\begin{equation} \label{eq:cocycleeq2}
\psi_{\gamma}
=\psi+\psi|T^{2m_1}S+\dots+\psi|T^{2m_{n}}S\dots T^{2m_1}S.
\end{equation}
Note that the slash action has the property
	\[|(f|_k\gamma)(\tau)|=|f(\gamma\tau)|
	\frac{\im(\gamma\tau)^{k/2}}{\im (\tau)^{k/2}}.\]
In particular,~\eqref{eq:cocycleeq2} implies
	\[|\psi_{\gamma}(\tau_0)|\im (\tau_0)^{k/2}
	\le |\psi(\tau_0)|\im (\tau_0)^{k/2}
	+|\psi(\tau_1)|\im (\tau_1)^{k/2}
	+\dots+|\psi(\tau_{n})|\im (\tau_{n})^{k/2},\]
where $\tau_j=T^{2m_j}S\dots T^{2m_1}S\tau_0$, and $\tau=S\tau_{n}$. Therefore,
\begin{equation} \label{eq:fkestimate}
|F(\tau)|\im (\tau)^{k/2}
\le
|F(\tau_0)|\im (\tau_0)^{k/2}+
|\psi(\tau_0)|\im (\tau_0)^{k/2}
+\sum_{i=1}^{n}|\psi(\tau_i)|\im (\tau_i)^{k/2}.
\end{equation}
\begin{proposition} \label{prop:fksbd}
	With the above notation assume that $|\psi(z)| \le C$ for $|z|\ge 1/2$.
	Then
	\begin{equation} \label{eq:fksgrowth}
	|F(\tau)|\im (\tau)^{k/2} \ll_k
	\begin{cases}
	C(\bH(\tau)^{k/2}+\bN(\tau)^{1-k/2}) \qquad\qquad k\in(0,2),\\
	C(\bH(\tau)+\log(1+\im(\tau)^{-1}))    \qquad k=2,\\
	\mathrlap{C(\bH(\tau)^{k/2}+1)}
	\hphantom{C(\bH(\tau)+\log(1+\im(\tau)^{-1}))}                 \qquad k>2.
	\end{cases}
	\end{equation}
\end{proposition}
\begin{proof}
	Since $|\tau_0|\ge1$ by induction we see that for $j\ge1$
	we have $\im \tau_j\le 1$ and $|\tau_j|\ge1$.
	By Proposition~\ref{prop:bdfdomain} the first two terms
	on the right of~\eqref{eq:fkestimate} are $\ll_k C(1+\bH(\tau)^{k/2})$
	and the remaining sum is trivially bounded by
	$C\sum_{j=1}^{\bN(\tau)}\im(\tau_j)^{k/2}$.
	Note that $\im \tau_j\le 2/(2j-1)$
	(see Lemma~2 in~\cite{RV}),
	so that $n\ll \im(\tau)^{-1}$.
	Thus
	\[
	\sum_{j=1}^{n}\im(\tau_j)^{k/2} \ll_k
	\begin{cases}
	(n+1)^{(2-k)/2}  									\qquad k\in (0,2),\\
	\mathrlap{\log(n+1)}\hphantom{(n+1)^{(2-k)/2}}      \qquad k=2,\\
	\mathrlap1\hphantom{(n+1)^{(2-k)/2}} 				\qquad k>2.
	\end{cases}
	\]
	Since $n+1=\bN(\tau)$ and $n\ll \im(\tau)^{-1}$ this implies the claim.
\end{proof}

\subsection{Estimates for the Fourier coefficients}
From now on we concentrate on the case
$\phi(\tau)=(\tau/i)^{-s}$. Using the estimates for $\bH(\tau)$
and $\bN(\tau)$ from Section~\ref{sec:boundbh} and Section~\ref{sec:boundbn}
below, we will now obtain various estimates for $\alpha_{n,k}^{\pm}(s)$.
For $0\le \re s\le k$ define
	\begin{equation} \label{eq:cs}
	c(s)\coloneqq \max_{|z|\ge 1/2} |(z/i)^{-s}\pm (z/i)^{s-k}| \le 2^{k+1}e^{\pi|\im s|/2}.
	\end{equation}

\begin{proposition} \label{prop:alphanbd}
	Let $k>0$ and $0\le\re s\le k$. Then
	\[\alpha_{n,k}^{\pm}(s) \ll_k
	\begin{cases}
	c(s)n^{k/2}(1+\log^{2-k} n),                \qquad k\in(0,2),\\
	\mathrlap{c(s)n(1+\log n),}
	\hphantom{c(s)n^{k/2}(1+\log(n)^{2-k})}     \qquad k=2,\\
	\mathrlap{c(s)n^{k-1},}
	\hphantom{c(s)n^{k/2}(1+\log(n)^{2-k})}     \qquad k>2,
	\end{cases}\]
	where $c(s)$ is defined in~\eqref{eq:cs}.
\end{proposition}
\begin{proof}
	We have $\alpha_{n,k}^{\pm}(s)=\frac{1}{2}\int_{i/n-1}^{i/n+1}F_{k}^{\pm}(\tau,s)e^{-\pi i n\tau}d\tau$.
	Therefore, for $0<k<2$ we have
	\[|\alpha_{n,k}^{\pm}(s)| \ll_k
	n^{k/2}\frac{c(s)}{2}\int_{-1}^{1}(\bH(x+i/n)^{k/2}+\bN(x+i/n)^{1-k/2})dx
	\ll_k c(s)n^{k/2}(1+\log^{2-k} n),\]
	where we have used Proposition~\ref{prop:estimpart}
	and Corollary~\ref{cor:estinversions}. Similarly, for $k=2$ we have
	\[|\alpha_{n,2}^{\pm}(s)|
	\ll_k n\frac{c(s)}{2}\int_{-1}^{1}(\bH(x+i/n)+\log n)dx
	\ll c(s)n(1+\log n),\]
	and for $k>2$ we have
	\[|\alpha_{n,k}^{\pm}(s)|
	\ll_k n^{k/2}\frac{c(s)}{2}\int_{-1}^{1}(\bH(x+i/n)^{k/2}+1)dx
	\ll c(s)n^{k-1},\]
	as claimed.
\end{proof}
Similarly, we get an estimate for sums of squares of the coefficients.
\begin{proposition}\label{prop:squaresum}
	Let $k\in(0,2)$ and $0\le\re s\le k$. Then for $x\ge 2$ we have
	\[
	\sum_{n\le x}|\alpha_{n,k}^{\pm}(s)|^2 \ll_k c(s)^2x^{k+|k-1|}\log^2 x.
	\]
\end{proposition}
\begin{proof}
	Using Proposition~\ref{prop:fksbd} we get
	\[\sum_{n\ge0}|\alpha_{n,k}^{\pm}(s)|^2t^{k}e^{-2\pi n t}
	= \frac{1}{2}\int_{-1+it}^{1+it}|F(\tau)|^2\im(\tau)^kd\tau
	\ll_k c(s)^2\int_{-1}^{1}(\bH(x+it)^k+\bN(x+it)^{2-k})dx.\]
	By Proposition~\ref{prop:estimpart} and Corollary~\ref{cor:estinversions}
	the integral on the right is bounded by $t^{-|k-1|}\log^2(1/t)$.
	Setting $t=1/x$ gives the claim.
\end{proof}
Note that from the proof of Lemma~\ref{lem:Aest} it follows that for $k\ge1$
and $0<\re s < k$ the Dirichlet series defining $\Aa_k^{\pm}(w,s)$
converges absolutely for $\re w>k$. Since for $1\le k<2$ the function
$w\mapsto \Aa_k^{-}(w,s)$ has a pole at $w=k$, we see that the above bound
is optimal up to powers of $\log x$ in this range.

Finally, we give an approximation to the partial sums $\sum_{n\le x}\alpha_{n,k}^{\pm}(s)$.

\begin{proposition}\label{prop:partsum} Let $0<k<2$ and $k/2\le\re s\le k$.
Then
	\[
	\sum_{n\le x}\alpha_{n,k}^{\pm}(s) =
	\pm \alpha_{0,k}^{\pm}(s)\frac{(\pi x)^{k}}{\Gamma(k+1)}
	+\frac{(\pi x)^{s}}{\Gamma(s+1)}+O(c(s)x^{k/2}\log^3\!x).
	\]
\end{proposition}
\begin{proof}
Let $x = N+1/2$, where $N\in\ZZ$ and define $S(x)=\sum_{n\le x}\alpha_{n,k}^{\pm}(s)$.
To simplify the notation we will write $F(\tau)=F_{k}^{\pm}(\tau,s)$.
Since $\sum_{n=0}^{N}e^{-\pi i n\tau}=\frac{e^{\pi i\tau}-e^{-N\pi i \tau}}{e^{\pi i \tau}-1}$ we have
	\begin{align*}
	S(x) = \frac{1}{2} \int_{-1+i/x}^{1+i/x}F(\tau)
	\frac{e^{\pi i\tau}-e^{-N\pi i\tau}}{e^{\pi i \tau}-1}d\tau
	= \frac{i}{4} \int_{-1+i/x}^{1+i/x}F(\tau)
 \frac{e^{-x \pi i\tau}}{\sin \frac{\pi \tau}{2}}d\tau .
	\end{align*}
(The integral $\int_{-1+i/x}^{1+i/x}\frac{F(\tau)e^{\pi i\tau}}{e^{\pi i\tau}-1}d\tau=-\sum_{m\ge1}\int_{-1+i/x}^{1+i/x}F(\tau)e^{\pi i m\tau}d\tau$
clearly vanishes.)
Note that $\frac{1}{\sin \frac{\pi\tau}{2}} = \frac{2}{\pi\tau}+O(\tau)$
for $\tau\in[-1+i/x,1+i/x]$, and integrating the $O(\tau)$ term gives
an error of the order at most $O(x^{k/2}\log^2\!x)$ as in the proof of
Proposition~\ref{prop:alphanbd}.

After applying the identity~\eqref{eq:fksfeq} to the part of the
integral with $\frac{2}{\pi \tau}$ we are left with
	\begin{align*}
	\frac{1}{2\pi i} \int_{-i+1/x}^{i+1/x}
	(\pm F(i/r)r^{-k}+r^{-s}\mp r^{s-k})
	\frac{e^{\pi xr}}{r}dr.
	\end{align*}
By inverse Laplace transform we have
	\begin{equation} \label{eq:invlaplace}
	\frac{1}{2\pi i}\int_{-i\infty +1/x}^{i\infty +1/x}\frac{e^{\pi x r}}{r^{\alpha+1}}dr =
	\frac{(\pi x)^{\alpha}}{\Gamma(\alpha+1)},
	\end{equation}
and thus
	\[\frac{1}{2\pi i}\int_{-i +1/x}^{i +1/x}(r^{-s}\mp r^{s-k})e^{\pi x r}r^{-1}dr =
	\frac{(\pi x)^{s}}{\Gamma(s+1)} \mp
	\frac{(\pi x)^{k-s}}{\Gamma(k-s+1)}+O(1).\]
Thus it remains to estimate
	\[
	\pm \frac{e^{\pi}}{2\pi} \int_{-1}^{1}
	F\Big(\frac{tx^2+ix}{1+t^2x^2}\Big)(it+1/x)^{-k-1}e^{\pi xit}dt.
	\]
We split this integral into two parts: $|t|\le \frac{1}{2\sqrt{x}}$
and $1\ge |t|\ge \frac{1}{2\sqrt{x}}$. To estimate the integral
for $|t|\le \frac{1}{2\sqrt{x}}$ we plug in the Fourier expansion of $F$.
By~\eqref{eq:invlaplace} the constant term gives a contribution of
\[\pm\alpha_{0,k}^{\pm}(s)\frac{(\pi x)^{k}}{\Gamma(k+1)} + O(x^{k/2}).\]
The rest of the terms are $\alpha_{n,k}^{\pm}(s)I_n$, where
	\[ I_n \,\change{:=}\, \pm \frac{e^{\pi}}{2\pi} \int_{-\frac{1}{2\sqrt{x}}}^{\frac{1}{2\sqrt{x}}}
\exp\Big(\frac{-\pi nx+ \pi i ntx^2}{1+t^2x^2}\Big)(it+1/x)^{-k-1}e^{\pi xit} dt.\]
We claim that $I_n\ll e^{-\pi n} x^{(k-1)/2}$ when $x\ge 16$, say. To see this, we begin by observing that, by symmetry and a trivial estimate, it suffices to consider
	\begin{equation} \label{eq:Inmod} J_n\coloneqq\int_{\frac{2}{x}}^{\frac{1}{2\sqrt{x}}}
	\exp\Big(\frac{-\pi nx+ \pi i ntx^2}{1+t^2x^2}\Big)(it+1/x)^{-k-1}e^{\pi xit} dt \end{equation}
and show that $J_n\ll e^{-\pi n} x^{(k-1)/2}$. To this end, we set
	\begin{align*} B(t) & \coloneqq\exp\Big(\frac{-\pi nx}{1+t^2x^2}\Big)|it+1/x|^{-k-1}, \\
	A(t)& \coloneqq\frac{\pi ntx^2}{1+t^2x^2}+\pi x t-(k+1)\im \log(it+1/x), \end{align*}
so that the integrand in \eqref{eq:Inmod} can be written as $B(t)\exp(iA(t))$. We observe that
	\[ B(t)\ll e^{-\pi n} x^{-(k+1)/2} \]
uniformly for $|t|\le 1/(2\sqrt{x})$. Moreover,
	\[ A'(t)=\frac{\pi n(1/x^2-t^2)}{(1/x^2+t^2)^2}
	+\pi x - \im \frac{i(k+1)}{it+1/x} \ll nx \]
uniformly for $2/x\le t \le 1/(2\sqrt{x})$. A calculus argument shows that $B(t)/A'(t)$ is a decreasing function on that interval, whence
a classical bound for oscillatory integrals  \cite[Lem.~4.3]{T} yields the asserted bound
	\[ J_n \ll \max_{2/x\le t \le 1/(2\sqrt{x})} \frac{B(t)}{|A'(t)|} \ll e^{-\pi n} x^{(k-1)/2}.\]
Summing this over all
$n\ge 1$, we obtain the asserted contribution $O(x^{(k-1)/2})$.

Finally, we split the integral over $|t|\ge \frac{1}{2\sqrt{x}}$
into intervals $[\frac{1}{2n+1},\frac{1}{2n-1}]$, $n \le \sqrt{x}+1/2$.
Using the change of variables $t\mapsto \frac{1}{2n+t}$ which sends
$[-1,1]$ to $[\frac{1}{2n+1},\frac{1}{2n-1}]$,
and noting that $\frac{(2n+t)^{-1}x^2+ix}{1+(2n+t)^{-2}x^2}$
is very close to $2n+t+\frac{4n^2}{x}i$, we get
	\[\int_{|t|\ge \frac{1}{2\sqrt{x}}}
	\ll \sum_{n\le \sqrt{x}+1}n^{k-1}\int_{-1}^{1}|F(t+4n^2i/x)|dt
	\ll_k x^{k/2}\log^2\!x\sum_{n\le \sqrt{x}+1}n^{-1} \ll x^{k/2}\log^3\!x.\]
This concludes the proof.
\end{proof}

\subsection{Estimate for $\bH(\tau)$}
\label{sec:boundbh}
\begin{proposition}\label{prop:estimpart}
	For $0<y<1/2$ we have
	\begin{equation} \label{eq:estimpart}
	\int_{-1}^{1}\bH(x+iy)^{\alpha}dx \;\ll_{\alpha}\;
	\begin{cases}
	\mathrlap{1} \hphantom{\log(y^{-1})}               \qquad \alpha\in(0,1),\\
	\log(y^{-1})                                       \qquad \alpha=1,\\
	\mathrlap{y^{1-\alpha}} \hphantom{\log(y^{-1})} \qquad \alpha>1.
	\end{cases}
	\end{equation}
\end{proposition}
\begin{proof}
Fix $y=N^{-1}$. Since $\im \frac{a\tau+b}{c\tau+d}=\frac{\im \tau}{|c\tau+d|^2}$
and $\gamma(\tau)\in\bF$ if and only if $\im \gamma(\tau)$ is maximal
among all $\gamma\in\Gamma_{\theta}$, we see that
	\[\bH(x)\coloneqq\bH(x+iy) =
	\max\left\{\frac{y}{(cx-d)^2+c^2y^2}
	\;\Big\vert\; (c,d)=1,\; 2|cd \right\}.\]
Without loss of generality we only consider $(c,d)$ with $c>0$.
Note that $N^{-1} \le \bH(x) \le N$ for all $x\in[-1,1]$. Let $\bH(x)\ge T>2$.
Therefore $(cx-d)^2+c^2N^{-2}\le N^{-1}T^{-1}$ for some $c,d$ with $c>0$,
which implies $c\le \sqrt{N/T}$ and $|x-d/c|\le \frac{1}{c\sqrt{NT}}$.
If $(c',d')$ is a different pair with $c'\le \sqrt{N/T}$
such that $|x-d'/c'|\le \frac{1}{c'\sqrt{NT}}$,
then
	\[\frac{1}{cc'}\le |d/c-d'/c'|\le \frac{1}{c\sqrt{NT}}+\frac{1}{c'\sqrt{NT}}\le \frac{2}{cc'T},\]
which is impossible. Hence the above inequality can hold only
for one pair $(c,d)$ with $c\le \sqrt{N/T}$. Conversely,
if $|x-d/c|\le \frac{1}{c\sqrt{NT}}$ and $c\le \sqrt{N/T}$,
then $\bH(x)\ge T/2$.
Let us denote
	\[u(T) \,\change{:=}\, \frac{2}{\sqrt{NT}}\sum_{c\le \sqrt{N/T}}\frac{2\phi(c)}{c},\]
where $\phi$ is Euler's totient function.
Then a simple estimate shows that $u(T)\le 4/T$ for $T<N$.
Moreover, for $T>2$ the measure of the set $\bH^{-1}([T,N])$
belongs to the interval $[u(2T),u(T)]$.
From this we see that for $k\ge 1$
	\[\mu(\bH^{-1}([2^{k},2^{k+1}]))\le
u(2^{k})\le 2^{2-k},\]
so that
	\[\int_{-1}^{1}\change{\bH}^{\alpha}(x)dx \le
	\change{2^{2+\alpha}}\sum_{k\le \log_2(N)}2^{(\alpha-1)k}+\change{2^{1+\alpha}},\]
which immediately implies~\eqref{eq:estimpart}.
\end{proof}

\subsection{Estimate for $\bN(\tau)$}
\label{sec:boundbn}
\begin{proposition}
	\label{prop:iterations}
	We have
	\[\int_{-1}^{1}\bN(x+iy)dx = \frac{2}{\pi^{2}}\log^2 y + O(\log y),
	\qquad y\to 0.\]
\end{proposition}
\begin{corollary} \label{cor:estinversions}
	For $y\in(0,1)$ we have
	\begin{align*}
	\int_{-1}^{1}\bN(x+iy)^{\alpha}dx &\ll
	\log^{2\alpha}(1+y^{-1}),\quad\quad 0<\alpha\le 1,\\
	\int_{-1}^{1}\bN(x+iy)^{\alpha}dx &\ll_{\alpha}
	y^{1-\alpha}\log^2(1+y^{-1}), \qquad \alpha>1.
	\end{align*}
\end{corollary}
\begin{proof}
	For $0<\alpha\le 1$ the claim follows from
	Proposition~\ref{prop:iterations} by H\"older's inequality.
	Since $\bN(x+iy)\ll 1+y^{-1}$, for $\alpha>1$ we have
	\[\int_{-1}^{1}\bN^{\alpha}(x+iy)dx
	\ll (1+y^{-1})^{\alpha-1}\int_{-1}^{1}\bN(x+iy)dx
	\ll (y/2)^{1-\alpha}\log^2(1+y^{-1}) ,\]
	from which we obtain the second claim.
\end{proof}

\begin{proof}[Proof of Proposition~\ref{prop:iterations}]
We set $\mathcal{U}_j\coloneqq\{\tau\colon |\tau|<1, \bN(\tau)\ge j+1\}$.
From the definition of $\bN(\tau)$ it follows
that $\mathcal{U}_1=D=\{\tau\in\HH\colon |\tau|<1\}$.
Moreover, from the description of the greedy algorithm for
computing $\gamma_{\tau}$, we have $\bN(\tau)=\bN(-1/\tau)+1$ for $|\tau|<1$,
which leads to the recursion
$\mathcal{U}_{j+1} = \bigsqcup_{n\ne 0}ST^{2n}\mathcal{U}_{j}$.
This implies that
	\[\mathcal{U}_{j+1} = \bigsqcup_{n_1,\dots,n_j\ne 0}
	ST^{2n_1}\dots ST^{2n_j}D,\quad j\ge 0.\]
By the definition of $\mathcal{U}_j$, we have
	\begin{equation} \label{eq:ndecompose}
	\bN(\tau)=1+\sum_{j\ge 1} \mathds{1}_{\mathcal{U}_j}(\tau)
	=1+\sum_{\gamma\in\mathcal{C}}\mathds{1}_{\gamma(D)}(\tau),
	\quad |\tau|<1,
	\end{equation}
where $\mathcal{C}$ denotes the set of all elements
of the form $\gamma=ST^{2n_1}\dots ST^{2n_j}$ for $j\ge 0$
with $n_1,\dots,n_j\ne 0$.

Note that the preimage (in $\Gamma_{\theta}$) of $(c,d)$ under the
map $(\smat abcd)\xmapsto{r} (c,d)$ is the coset $T^{2\ZZ}(\smat abcd)$.
Any such coset contains exactly one element of the
form $ST^{2n_j}\dots ST^{2n_1}S^{\delta}$. A simple inductive argument shows
that if $(\smat abcd)\in \mathcal{C}$, then $|c|<|d|$, and
if $(\smat abcd)\in \mathcal{C}S$, then $|c|>|d|$. Thus $r$ provides
a bijection between $\mathcal{C}$ and the set $\mathcal{N}$
of all pairs $(c,d)$ with $|c|<|d|$, $\gcd(c,d)=1$, and $c\not\equiv d\Mod{2}$,
modulo the equivalence relation $(c,d)\sim (-c,-d)$.

Define $\Phi(y)\coloneqq \int_{-1}^{1}\bN(x+iy)dx$ and
for $\re s>1$ consider
	\[\Psi(s) \,\change{:=}\, \int_{0}^{\infty}(\Phi(y)-\change{2})y^{s-1}dy.\]
For $\gamma=(\smat abcd)\in\Gamma_{\theta}$,
$\gamma(D)$ is a half-disk with endpoints $\gamma(\pm1)$ and radius
$\frac{1}{2}|\gamma(1)-\gamma(-1)|=|c^2-d^2|^{-1}$.
An elementary calculation then shows that for $\gamma=(\smat abcd)$ we have
	\[\int_{-1}^{1}\int_{0}^{\infty}
	\mathds{1}_{\gamma(D)}(x+iy) y^{s-1}dydx
	= |c^2-d^2|^{-s-1} \frac{\Gamma(s/2)\Gamma(3/2)}{\Gamma((s+3)/2)}.\]
Combined with~\eqref{eq:ndecompose} and the above description of~$\mathcal{C}$
we get
	\[\Psi(s) = \frac{\Gamma(s/2)\Gamma(3/2)}{\Gamma((s+3)/2)}
	\sum_{(c,d)\in\mathcal{N}}|c^2-d^2|^{-s-1}
	=\frac{\Gamma(s/2)\Gamma(3/2)}{\Gamma((s+3)/2)}
	\frac{\zeta_{\rm{odd}}^2(s+1)}{\zeta_{\rm{odd}}(2s+2)},\]
where $\zeta_{\rm{odd}}(s)=\sum_{n\ge 1}(2n-1)^{-s}$.
To obtain this identity we have used the bijection $(c,d)\mapsto (c-d,c+d)$
between $\mathcal{N}$ and the set of all pairs of coprime
odd integers $(m,n)$ with opposite signs, again modulo $(m,n)\sim (-m,-n)$.
The function $\Psi(s)$ is meromorphic in the half-plane $\re s>-1/2$ with
the only singularity at $s=0$ with principal part
	\[\Psi(s) = \frac{4}{\pi^2 s^3}+\frac{c_2}{s^2}+\frac{c_1}{s}+O(1),
	\quad s\to 0\]
for some $c_1,c_2\in\RR$. Since $|\Psi(u+iv)|\ll_u |v|^{-5/4}$ for $u>-1/2$,
by a standard application of the inverse Mellin transform we obtain
	\[\Phi(y) = 2\pi^{-2}\log^2 y - c_2\log y + c_1 + \change{2}
	+ O_{\eps}(y^{1/2-\eps}), \quad y\to 0. \qedhere\]
\end{proof}
The constants $c_1,c_2$ are explicit, albeit complicated, for instance,
	\[c_2=\frac{12+4\log 2-24\log \pi-288\zeta'(-1)}{3\pi^2} = 1.180066...\, .\]

\section{The Fourier interpolation basis of Theorem~A revisited}
\label{sec:sqrt}
It was shown in \cite[Prop. 4]{RV} that~\eqref{eq:sqrt}
holds if one assumes $f(x), \widehat{f}(x) \ll (1+|x|)^{-13}$.
Using the estimates from Section~\ref{sec:cocycleestimates} we may now weaken
these constraints.
\begin{theorem}\label{thm:sqrt}
	Suppose $f$ is an even integrable function on~$\RR$ such that
	also $\widehat{f}$ is integrable.
	Suppose also that both $f$ and $\widehat{f}$ are absolutely continuous and that the two integrability conditions
	\begin{align} \label{eq:fdiff} \int_{-\infty}^{\infty} |f'(x)|(1+|x|)^{1/2} \log^3(e+|x|) dx & < \infty ,\\
	\int_{-\infty}^{\infty} |(\widehat{f})'(\xi)|(1+|\xi|)^{1/2} \log^3(e+|\xi|) d\xi & < \infty\,  \nonumber \end{align}
	hold. Then we may represent $f$ as in \eqref{eq:sqrt} for every real $x$, with the two series in \eqref{eq:sqrt} being in general only conditionally convergent.
\end{theorem}
The proof of this theorem relies on the following proposition.

\begin{proposition}\label{prop:part12}
	For $x> 0$ and $N\ge 1$ we have
	\begin{equation} \label{eq:sumoh} \sum_{n\le N}b_{n}^{\pm}(x) =
	\pm 2b_{0}^{\pm}(x) N^{1/2} + O(N^{1/4}\log^3\!N) + O(\min (x^{-1/2}N^{1/4}, N^{1/2})) ,
	\end{equation}
	where $b_{n}^{\pm}$ are defined by~\eqref{eq:sqrtndef}
	and the implied constants in the \change{O} terms are absolute.
\end{proposition}
The proposition remains true when $x=0$, but in that case it is better to use the two expressions
	\begin{equation} \label{eq:part0}
	\sum_{n=1}^{N}b_{n}^{-}(0)=0 \quad \text{and} \quad
	\sum_{n=1}^{N}b_{n}^{+}(0)=-2N^{1/2} + O(1), \end{equation}
which are obvious consequences of the facts that $b_n^{-}(0)=0$ for all
$n\ge 1$ and $b_n^{+}(0)=-2$ if $n$ is a square and otherwise $b_n^{+}(0)=0$.
\begin{proof}[Proof of Proposition~\ref{prop:part12}]
From~\cite[Prop.~2]{RV} it follows that that
$F(\tau)=\sum_{n\ge 0}b_n^{\pm}(x)e^{ \pi i n\tau}$ is of moderate growth and
	\begin{equation} \label{eq:feqsqrtn}
	F(\tau)\mp (\tau/i)^{-1/2}F(-1/\tau) =
	e^{\pi i x^2\tau}\mp (\tau/i)^{-1/2}e^{\pi i x^2(-1/\tau)}.
	\end{equation}
Therefore, we have $F(\tau)=F_{1/2}^{\pm}(\tau,\phi)$,
where $\phi(\tau)=e^{\pi i x^2\tau}$. Since $\phi(\tau)$ is bounded
for $|\tau|\ge 1/2$,
Proposition~\ref{prop:fksbd} implies that~\eqref{eq:fksgrowth} holds
for $k=1/2$ and hence
	\[|b_n^{\pm}(x)|\ll n^{1/4}(1+\log^2 n).\]

Then we repeat the argument from the proof of Proposition~\ref{prop:partsum},
the only difference being that after applying~\eqref{eq:feqsqrtn}
we get, along with the two first terms on the right-hand side of \eqref{eq:sumoh}, the term
	\begin{equation}\label{eq:extra}
	\frac{1}{2\pi i} \int_{-i+1/N}^{i+1/N}
	(e^{-\pi x^2r}\pm r^{-1/2}e^{-\pi x^2/r})
	\frac{e^{\pi Nr}}{r}dr\, .
	\end{equation}
The first term in the integrand in \eqref{eq:extra} yields trivially a contribution that
is $O(\log N)$, and this can be absorbed in the first \change{$O$} term
in~\eqref{eq:sumoh}.
We estimate the contribution from the second term in the integrand of~\eqref{eq:extra}
trivially if $x\le N^{-1/2}$ and see that we then get a term that is $O(N^{1/2})$. When
$x\ge N^{-1/2}$,
we estimate the contribution to the integral from the interval $|\im r|\le
\max(1,2xN^{-1/2})$ trivially and use again the bound for oscillatory integrals
from~\cite[Lem.~4.3]{T} to deal with the remaining part. We obtain then a term that is
$O(x^{-1/2} N^{1/4})$ which yields the latter \change{$O$} term in~\eqref{eq:sumoh}.
\end{proof}

\begin{proof}[Proof of Theorem~\ref{thm:sqrt}]
We begin by showing that the two series in \eqref{eq:sqrt} converge. By symmetry, it suffices to consider the first of them. By partial summation, we find that
	\begin{equation} \label{eq:partsum1} \sum_{n=K+1}^{N} f (\sqrt{n}) a_n(x)
	=  f(\sqrt{N}) A(N)  - f(\sqrt{K}) A(K)   -\int_K^{N} f' (\sqrt{y})\frac{1}{2\sqrt{y}}  A(y) dy,\end{equation}
where $A(N):=\sum_{n\le N} a_n(x).$ By Proposition~\ref{prop:part12} and the relation $a_n(x)=(b^+_n(x)+b_n^-(x))/2$, we have
	\begin{equation} \label{eq:Asum}
	A(y)= -b_0^{-}(0) y^{1/2}+O(y^{1/4}\log^3\!y)\end{equation}
when $y\neq 0$, but in view of \eqref{eq:part0}, this is also true for $y=0$. Since the first term on the right-hand side of \eqref{eq:Asum} is smooth, we may now use integration by parts in \eqref{eq:partsum1} along with a change of variables to deduce that
	\begin{align} \label{eq:three} \sum_{n=K+1}^{N}  f (\sqrt{n}) \widehat{a}_n(x)
	& \ll  |f (\sqrt{N})| N^{1/4} \log^3\!N
	      +|f (\sqrt{K})| K^{1/4} \log^3\!K  \\
	&  + \int_{\sqrt{K}}^{\sqrt{N}} | f (y)| dy
	    +\int_{\sqrt{K}}^{\sqrt{N}} | f'  (y)|\, y^{1/2} \log^3\!y\, dy.
	    \nonumber
	\end{align}
The first two terms on the right-hand side of \eqref{eq:three} tend to $0$ when $K, N\to \infty$ because
	\[ \widehat{f}(y)=-\int_{y}^\infty D\widehat{f}(\xi) d\xi
	\ll y^{-1/2}\log^{-3}\!y \int_y^{\infty} |D\widehat{f} (\xi)|(1+|\xi|)^{1/2} \log^3(e+|\xi|) d\xi .\]
Here the integral to the right tends to $0$ when $y\to \infty$ in view of \eqref{eq:fdiff}. The two integrals on the right-hand side of
\eqref{eq:three} also tend to $0$ when $K, N\to \infty$ by the respective integrability conditions on $f$ and $f'$.

We now turn to the proof that equality holds in \eqref{eq:sqrt}. To this end, we follow the proof of \cite[Prop. 4]{RV}. We write
	\[ \mathcal{R}_Mf (x)\coloneqq M^{1/2} e^{-\pi x^2/M} \int_{-\infty}^{\infty} f(x-y) e^{-\pi My^2} dy \]
and
	\[ \widehat{\mathcal{R}_Mf}(x)\coloneqq
	M^{1/2} \int_{-\infty}^{\infty} \widehat{f}(x-y) e^{-\pi (x-y)^2/M-\pi My^2} dy.\]
It is plain that $\mathcal{R}_Mf(x)\to f(x)$ when $M\to \infty$, and  hence it suffices to prove that
	\[ \sum_{n=0}^{\infty}\big(\mathcal{R}_Mf(\sqrt{n}) - f(\sqrt{n})\big) a_n(x) \to 0 \quad \text{and} \quad \sum_{n=0}^{\infty}\big(\widehat{\mathcal{R}_Mf}(\sqrt{n}) - \widehat{f}(\sqrt{n})\big) \widehat{a_n}(x) \to 0  \]
when $M\to \infty$.  We consider only the latter convergence, as
the two cases are almost identical. For convenience, we write
	\[  \widehat{\Delta_M f }(y):=\widehat{\mathcal{R}_Mf}(y) - \widehat{f}(y)\quad  \text{and} \quad Dg:=g'.\]
By the same argument of partial summation and integration as used in the first
part of the proof, we find that
	\begin{equation} \label{eq:two} \sum_{n=1}^{\infty} \widehat{\Delta_M f }(\sqrt{n}) \widehat{a_n}(x)
	\ll  \int_{1}^{\infty} |\widehat{\Delta_M f } (y)| dy
	+\int_{1}^{\infty} |D \widehat{\Delta_M f } (y)|\, (1+y)^{1/2} \log^3(e+y) dy. \end{equation}
A routine argument, using that $\widehat{f}$ is integrable, shows that the first integral on the right-hand side of \eqref{eq:two} tends to $0$ when $M\to \infty$. To deal with the second integral, we write
	\begin{align} \nonumber D \widehat{\Delta_M f } (y)  & =   M^{1/2} \int_{-\infty}^{\infty}
	\big( D\widehat{f}(y-v) - D\widehat{f}(y)\big) e^{-\pi Mv^2} dv \\
	& \nonumber + M^{1/2} \int_{-\infty}^{\infty}
	D\widehat{f}(y-v)\big(e^{-\pi(y-v)^2/M}-1\big) e^{-\pi Mv^2} dv \\
	& \label{eq:finalterm} + M^{1/2} \int_{-\infty}^{\infty}
	\widehat{f}(y-v)2\pi (y-v) M^{-1} e^{-\pi(y-v)^2/M} e^{-\pi Mv^2} dv\end{align}
and apply again routine arguments, along with our integrability assumptions on $\widehat{f}$ and $D\widehat{f}$, to show that each of the corresponding three terms tends to $0$ when $M\to \infty$. We give the details only for the last term in \eqref{eq:finalterm}. To this end, it suffices to observe that
		\[  (1+y)^{1/2} \log^3 (e+y) \ll |y-v|^{\delta} + |v|^{\delta}  \]
for some $\delta$, $1/2 < \delta < 1$, so that
	\begin{align*} \int_1^{\infty}  M^{1/2} \Big|\int_{-\infty}^{\infty}
	\widehat{f}(y-v) 2\pi & (y-v) M^{-1}  e^{-\pi(y-v)^2/M} e^{-\pi Mv^2} dv \Big| (1+y)^{1/2} \log^3 (e+y) dy \\
	& \ll \Big(M^{\delta/2} \int_{-\infty}^{\infty} e^{-\pi M v^2} dv + \int_{-\infty}^{\infty} |v|^{\delta} e^{-\pi M v^2} dv \Big) \| \widehat{f}\|_1, \end{align*}
and the latter term tends to $0$ when $M\to \infty$ since $\delta<1$.
\end{proof}

As was observed in the introduction, formula \eqref{eq:sqrt} reduces to the Poisson summation formula when $x=0$ in view of \eqref{eq:poisson}. Since we have the more precise formula \eqref{eq:part0} (instead of \eqref{eq:Asum}) in that case, the above proof therefore shows that the Poisson summation formula is valid when $f$, $f'$, $\widehat{f}$, $(\widehat{f})'$ are all integrable. Somewhat related and refined conditions can be found in the work of Kahane and Lemari\'{e}-Rieusset~\cite{KL}. See also Gr\"{o}chenig's paper~\cite{G}, where it is shown that the Poisson summation formula holds for functions in the Feichtinger algebra.

On the other hand, by a classical example of Katznelson \cite{Ka}, there exist functions $f$ with both $f$ and $\widehat{f}$ in $L^1$ for which the Poisson summation formula fails. This shows that we need indeed an additional assumption, beyond integrability of $f$ and $\widehat{f}$, for the Fourier interpolation formula \eqref{eq:sqrt} to hold for every real~$x$.

\end{document}